\documentclass[reqno,11pt]{article}

\usepackage{amsmath}
\usepackage{amsthm}
\usepackage{amssymb}
\usepackage{amscd}
\usepackage{xypic}
\usepackage{verbatim}
\usepackage{mathabx}
\usepackage{graphicx}
\usepackage{palatino}

\usepackage[plainpages=false,colorlinks,hyperindex,pdfpagemode=None,bookmarksopen,linkcolor=red,citecolor=blue,urlcolor=blue]{hyperref}
\usepackage{pdflscape}
\usepackage{stmaryrd}

\usepackage{multirow}

\theoremstyle{plain}
\newtheorem{theorem}[subsubsection]{Theorem}

\newtheorem*{theorem*}{Theorem}
\newtheorem{proposition}[subsubsection]{Proposition}
\newtheorem*{proposition*}{Proposition}
\newtheorem{lemma}[subsubsection]{Lemma}
\newtheorem*{lemma*}{Lemma}

\newtheorem{corollary}[subsubsection]{Corollary}
\newtheorem*{corollary*}{Corollary}

\theoremstyle{definition}
\newtheorem{definition}[subsubsection]{Definition}
\theoremstyle{remark}
\newtheorem{remark}[subsubsection]{Remark}

\newtheorem{observation}[subsubsection]{Observation}

\newtheorem{example}[subsubsection]{Example}

\renewcommand{\comment}[1] {  }

\DeclareFontFamily{OT1}{rsfs}{}
\DeclareFontShape{OT1}{rsfs}{n}{it}{<-> rsfs10}{}
\DeclareMathAlphabet{\mathscr}{OT1}{rsfs}{n}{it}

\newcommand{\Res}{\mathrm{Res}}

\newcommand{\Z}{\mathbb{Z}}
\newcommand{\C}{\mathcal{C}}

\newcommand{\adele}{{\mathbb{A}_k}}

\newcommand{\CC}{\mathbb{C}}

\newcommand{\RR}{\mathbb{R}}
\newcommand{\Rplus}{{\RR^\times_+}}

\newcommand{\QQ}{\mathbb{Q}}

\newcommand{\Ind}{\operatorname{Ind}}
\newcommand{\Hom}{\operatorname{Hom}}
\newcommand{\End}{\operatorname{End}}
\newcommand{\Aut}{{\operatorname{Aut}}}

\newcommand{\Gm}{\mathbb{G}_m}
\newcommand{\Ga}{\mathbb{G}_a}

\newcommand{\PGL}{\operatorname{PGL}}

\newcommand{\SL}{\operatorname{SL}}

\newcommand{\tr}{\operatorname{tr}}

\newcommand{\spec}{\operatorname{spec}}
\newcommand{\Vol}{\operatorname{Vol}}
\newcommand{\diag}{{\operatorname{diag}}}
\newcommand{\adiag}{{\operatorname{adiag}}}

\newcommand{\disc}{{\operatorname{disc}}}

\newcommand{\Id}{\operatorname{Id}}

\newcommand{\Eis}{{\operatorname{Eis}}}

\newcommand{\cusp}{{\operatorname{cusp}}}

\newcommand{\pr}{{\operatorname{prEis}}}
\newcommand{\npr}{{\operatorname{npr}}}

\newcommand{\TF}{{\operatorname{TF}}}

\newcommand{\rest}{{\operatorname{rest}}}

\newcommand{\PW}{{\operatorname{PW}}}
\newcommand{\Laur}{{\operatorname{Laur}}}
\newcommand{\Nash}{{\operatorname{Nash}}}
\newcommand{\br}{{\operatorname{broad}}}

\renewcommand{\top}{{\operatorname{top}}}

\newcommand{\Dfrac}[2]{ 
  \ooalign{ 
    $\genfrac{}{}{1.2pt}0{#1}{#2}$\cr 
    $\color{white}\genfrac{}{}{.4pt}0{\phantom{#1}}{\phantom{#2}}$} 
}

\makeatletter
\newcommand{\subjclass}[2][1991]{ 
  \let\@oldtitle\@title 
  \gdef\@title{\@oldtitle\footnotetext{#1 \emph{Mathematics subject classification.} #2}} 
}
\newcommand{\keywords}[1]{ 
  \let\@@oldtitle\@title 
  \gdef\@title{\@@oldtitle\footnotetext{\emph{Key words and phrases.} #1.}} 
}
\makeatother

\begin{document}

\swapnumbers

\numberwithin{equation}{section}
\setcounter{tocdepth}{2}
\title{The Selberg trace formula revisited}
\author{Yiannis Sakellaridis \footnote{
Department of Mathematics and Computer Science, Rutgers University at Newark, 101 Warren Street, Smith Hall 216, Newark, NJ 07102. 
Email: \texttt{sakellar@rutgers.edu}.}}


\date{}

\maketitle

\begin{abstract}
A new approach to the Selberg trace formula, and more precisely its spectral side, is developed. The approach relies on a notion of ``Plancherel decomposition'' of ``asymptotically finite functions'', and may generalize to obtain a general relative trace formula. This is an incomplete first version that will be complemented by an account of the invariant trace formula.
\end{abstract}

\tableofcontents

\section{Introduction}

\subsection{The Selberg trace formula}

Let $H= \PGL_2$ over a global field $k$. For the introduction, we assume that $k=\QQ$, and for the rest of the paper we work with $k$: a number field, but the case of a function field is at worst verbatim, and often easier. We denote by $\adele$ the ring of adeles of $k$, and by $[H]$ the automorphic quotient $H(k)\backslash H(\adele)$ (and similarly for other groups).

The Selberg trace formula computes a number that ``does not exist'' (the trace of the operator $\mathcal R(\Phi dg)$, where $\mathcal R$ denotes the action of $H(\adele)$ on $L^2([H])$ and $\Phi$ is the space $\mathcal S(H(\adele))$ of Schwartz functions) by ``truncating'' the divergent portions of this number in both its geometric and its spectral expansions. 

In its non-invariant form, it arrives at a distribution on $H(\adele)$, which we will denote as $\TF_0$; the spectral side reads:
$$\TF_0(\Phi) =  \sum_{\pi \in \hat H^\Aut_\disc} \tr(\pi(\Phi)) + $$
\begin{eqnarray} + \frac{1}{4} \tr (M(0) \pi_0(\Phi))-  \frac{1}{4\pi i} \int_{-i\infty}^{i\infty} \tr(M(-z)M'(z)\pi_{z}(\Phi)) dz. \label{Selberg-spectral}\end{eqnarray}

The first sum is indexed by the discrete automorphic spectrum of $H$ (which in this case consists of the cuspidal representations and the quadratic idele class characters), while for the second line the notation is as follows:

Let $[H]_B = A(k)N(\adele)\backslash H(\adele)$, the ``boundary degeneration'' of $[H]$, where $B$ denotes a $k$-Borel subgroup, $N$ its unipotent radical, and $A$ the universal Cartan $A=B/N$. (The space is independent of a choice of $B$.) When we write, simply, $A$ we will mean the connected Lie group $A(\RR)^0 \simeq \RR^\times_+$. We consider the normalized action of the group $[A]$ (and hence also of $A$) on functions on $[H]_B$: 
\begin{equation}\label{action-normalized-H}
a\cdot f (x) = \delta^{-\frac{1}{2}}(a) f(ax).
\end{equation}

Then $\pi_s = C^\infty(A\backslash [H]_B, \delta^{\frac{s}{2}})$, the space of smooth functions on $[H]_B$ which are $\delta^{\frac{s}{2}}$-equivariant under the action of $A$, where $\delta$ is the modular character of the Borel subgroup. (Under the unnormalized action this would be $\delta^{\frac{1+s}{2}}$; in particular, $\Re(s) =0$ corresponds to unitary induction.)

Finally, $M(s): \pi_s \to \pi_{-s}$ is the standard intertwining operator:
$$ (M(s)f) (x) = \int_{N(\adele)} f(wnx) dn,$$
where $w = \begin{pmatrix}
 & 1 \\
 -1 &
\end{pmatrix}.$

Obviously, the representations $\pi_s$ are reducible, and one can further decompose with respect to characters of $[A]^1=$ the group of idele classes of norm one, but this is a compact group and the decomposition of \eqref{Selberg-spectral} to its eigencharacters is a trivial matter.

The purpose of this paper is to revisit the Selberg trace formula and give a conceptual explanation for the terms in \eqref{Selberg-spectral}. This approach should make it easier to develop the analogous theory for the general relative trace formula, something that has not been done yet. Of course, this is a long-term undertaking, as demonstrated by the fact that it contains Arthur's general trace formulas as special cases.

\subsection{Overview of the method}

The main idea behind the proof is to compute the Hilbert--Schmidt inner product of two convolution operators defined by Schwartz functions $\Phi_1$, $\Phi_2$ on $H$, instead of the trace of the convolution $\Phi =\Phi_1^\vee \star \Phi_2$, where $\Phi_1^\vee (g) = \Phi_1 (g^{-1})$. (The ``inner products'' in this paper are, actually, bilinear, by abuse of language.) This corresponds to viewing the group $X=H=\PGL_2$ as a $G=H\times H$-space (by left and right multiplication), and expressing the trace formula as the relative trace formula for the quotient $(X\times X)/G^\diag$. Formally, one ends up having to compute the inner product (on $[G]$) of the two kernel functions for the operators $\mathcal R(\Phi_i dg)$:
\begin{equation}\label{Kformal} \left<K_{\Phi_1},K_{\Phi_2}\right>.\end{equation}

The key point now is understanding what makes this inner product diverge: the kernel functions are \emph{``asymptotically finite functions on $[G]$ with trivial exponents''}, s.\ \S \ref{ssSchwartz} and Proposition \ref{propasfinite} -- this, essentially, means that in some directions towards infinity they behave like eigenfunctions of a multiplicative group, and otherwise are of rapid decay. (In the case of rank one, this is the notion of ``extended Schwartz space'' introduced by Casselman \cite{Cas-extended}.) If the exponents (= normalized eigencharacters) were not trivial, one would be able to define a regularized, invariant inner product of the two kernel functions (thus arriving at a canonical invariant trace formula).

However, the exponents being trivial (in a normalized sense), means that no such regularized inner product exists. Thus, the Selberg trace formula is really the \emph{constant coefficient} of the Laurent expansion that one obtains by \emph{deforming} exponents and calculating regularized inner products:
\begin{equation}\label{Kformal-s} \left<K_{\Phi_1},(K_{\Phi_2})_s\right>,\end{equation}
where $(K_{\Phi_2})_s$ is a deformation of $K_{\Phi_2}$ with exponent $s$, cf.\  \S \ref{ssvariation}. Because the expansion has a pole, this coefficient is non-invariant. Of course, this is an alternative way of producing the same expression as with truncation methods, s.\ Lemma \ref{constantcoeffs}.

Nonetheless, viewing the trace formula in terms of regularized inner products allows us to define and compute its spectral decomposition as a ``Plancherel formula'' for the inner product of two asymptotically finite functions. To develop such a Plancherel formula, one can employ the usual techniques of the Plancherel decomposition, due to Selberg and Langlands: pseudo-Eisenstein series, and contour shifts. The input for pseudo-Eisenstein series will now also be asymptotically finite functions (on the boundary degeneration $[G]_B= [H]_B \times [H]_B$), but otherwise the argument is formally the same. I develop such a Plancherel formula in complete generality in rank one, s.\ Theorem \ref{Plancherelrankone}. This is not only for ``training purposes'': such a Plancherel formula for asymptotically finite functions on $[H]$ can be applied to obtain the spectral decomposition of a relative trace formula for $\PGL_2$ -- s.\ Example \ref{torus-RTF}. In rank two, to simplify the analysis I restrict the Plancherel decomposition to the case where one of the two functions of the inner product is a kernel function, s.\ Theorem \ref{spectral-general-ranktwo}. However, it is essential to allow the second function to be general, in order to be able to approximate it by pseudo-Eisenstein series.

The spectral decomposition gives a conceptual explanation for the terms of \eqref{Selberg-spectral}: the term $\frac{1}{4} \tr (M(0) \pi_0(\Phi))$, for example, shows up in the decomposition of the inner product of two functions on $A^2 = \Rplus^2$: one is invariant under the anti-diagonal copy $A^\adiag$, and the other is asymptotically invariant under $A^\diag$; the Plancherel formula then includes a discrete term corresponding to the only character of $A^2$ that is invariant under both $A^\diag$ and $A^\adiag$: the trivial one. To be precise, this describes the restriction to $A\times A$-orbits (times the appropriate volume factor) of two functions that live on $[H]_B\times [H]_B$. The fact that such functions arise in the spectral decomposition can be understood using the language of scattering theory: the constant term of a function on $[G]$ is invariant under an action of the Weyl group $\mathbb Z/2 \times \mathbb Z/2$ of $G$ expressed in terms of certain ``scattering operators''. (Those will not appear explicitly, but as standard intertwining operators acting on Mellin transforms, s.\ Corollary \ref{constantterm-s-G}.) In particular, the constant term of a function on $[G]$ which has exponents in the $A^\diag$-direction, also has the same exponents along the $A^\adiag$-direction. Of course, in the case of kernel functions, constant terms are very explicit, s.\ Proposition \ref{constantterm-kernel}. 

The integrand $\tr(M(-z)M'(z)\pi_{z}(\Phi))$ directly reflects the definition of the trace formula as the constant coefficient of the Laurent expansion of \eqref{Kformal-s} -- equivalently, as the derivative of the Taylor expansion that we obtain by canceling the pole. More precisely, the (relatively arbitary) process chosen to deform the kernel $K_{\Phi_2}$ to a function $(K_{\Phi_2})_s$ with different exponent involves multiplication on the second variable (close to the cusp) by $\delta^\frac{s}{2}$, where $\delta$ is the modular character of the Borel viewed as a function on $[H]$ by fixing an Iwasawa decomposition. (The trace formula depends on the choice of such a decomposition.) Comparing the constant term of $K_{\Phi_2}$ with the constant term of its modification involves replacing the operator $M_1(z) M_2(z)$, where $M_1$, $M_2$ are the aforementioned intertwining operators applied to the first, resp.\ second variable of $[G]_B = [H]_B \times [H]_B$, by $M_1(z) M_2(z+s)$, and taking derivatives one obtains the operator $M'(z)$. This is just an impressionistic explanation of the phenomenon, as the actual argument involves a variation of the input of a pseudo-Eisenstein series, s.\ \S \ref{ssdiagonal}.

On a technical note, for the contour shift I rely on a result of Harish-Chandra, which states that intertwining operators are of polynomial growth away from their poles, in any bounded vertical strip in the region $\Re(s)\ge 0$. This may be the biggest technical obstacle in generalizing this method to higher rank without hard analysis -- but it is conceivable that the result of Harish-Chandra can be generalized. Fortunately, no such estimates on Eisenstein series are needed, which are signifantly more complicated, as they are related to logarithmic derivatives of intertwining operators and, hence, zeroes of $L$-functions. The trace formula that we obtain by this method is absolutely convergent from the beginning -- no hard estimates are needed to establish this.

Incidentally, such ``asymptotically finite'' behavior appears for all ``theta series'' on smooth affine varieties with an action of a reductive group, see \cite[\S 5]{SaStacks}. Thus, developing a general Plancherel formula for asymptotically finite functions would give rise to a general relative trace formula, which does not exist yet. This was part of the motivation for the present paper -- another one being that I never understood the spectral decomposition of the Selberg trace formula.

\subsection{Notational habits}

In this paper, the symbol $\mathcal S$ is reserved for various spaces of \emph{Schwartz functions}, that is, smooth functions which are of rapid decay together with their polynomial derivatives. Extensions of such spaces are denoted by $\mathcal S^+$, possibly with the appropriate indices. The definitions are given in \S \ref{ssSchwartz}.

At several points, it would be more appropriate to be working with measures, instead of functions. For example, for the group $H=\PGL_2$, setting $G=\PGL_2\times\PGL_2$ and considering $H$ as a space with a right $G$-action via 
\begin{equation}h\cdot (h_1,h_2) = h_1^{-1} h h_2,
\end{equation}
we have an isomorphism of stacks $H/H-\operatorname{conj} = (H\times H)/G^\diag$, and this induces a canonical isomorphism of the $H(F)$-conjugacy-coinvariant space of Schwartz measures on $H(F)$ (for any local field $F$) with the $G^\diag(F)$-coinvariant space of Schwartz measures on $H(F)\times H(F)$, s.\ \cite{SaStacks}. Since here we are working with Schwartz functions, we need to fix measures, and some compatibility conditions on the choices of measures are described in \S \ref{sspreparation}.

For the convenience of the reader, I systematically use similar symbols for functions on similar spaces. That includes:
\begin{itemize}
 \item $\Phi$ for functions on $H(\adele)$, $G(\adele)$;
 \item $\varphi$ for functions on $[H]$, $[G]$;
 \item $f$ for functions on $[H]_B$, $[G]_B$.
\end{itemize}

\subsection{Acknowledgements} 

This proof of the trace formula was presented as part of a course that I taught at the University of Chicago in the Spring Quarter of 2017. I am grateful to the University of Chicago for the invitation, and to the audience of this course for patiently enduring a bad first draft of this paper and the technicalities that come with the trace formula. I would also like to thank Bill Casselman and Erez Lapid for discussing some of these technicalities with me, and providing references that helped me simplify some arguments. This work was supported by 
NSF grant DMS-1502270.

\section{Cusps, asymptotically finite functions, and approximation by constant term}

What follows is an adaptation of the formalism of asymptotically finite functions of \cite{SaStacks} and of the well-known theory of the constant term to the special, simple, case of $[H]=[\PGL_2]$ and $[G]=[\PGL_2]\times [\PGL_2]$.

\subsection{Compactifications} \label{sscompactifications}

We start with $[H]=[\PGL_2]$. Let $A\simeq \Gm$ be the universal Cartan of $H$, that is: the reductive quotient $B/N$ of any Borel subgroup. By abuse of notation, we will also be using $A$ for $\RR^\times_+=$ the identity component of $\Gm(\RR) \subset \Res_{k/\mathbb Q} A (\RR)$, in other words the diagonal copy of $\RR^\times_+$ in $A(k_\infty)$. There should be no confusion, since it is $A(k_\infty)$, not $A$ as an algebraic group, which acts on quotients such as:
$$[H]_B:= A(k)N(\adele)\backslash H(\adele),$$
the ``boundary degeneration'' of $[H]$. In the function field case, the reader should choose a place corresponding to a point defined over the finite base field, and make the necessary translations, replacing $\RR$ by the corresponding localization of $k$.

The \emph{reductive Borel--Serre compactification} $\overline{[H]}$ of $[H]$ is the union of $[H]$ with the space $\infty_B := A\backslash [H]_B$. To see how they are glued, consider the diagram:
\begin{equation}\label{leftright-H} \xymatrix{
& B(k)\backslash H(\adele) \ar[dl]_{\pi_H}\ar[dr]^{\pi_B}& \\
[H] && [H]_B.
}\end{equation}
The classical picture is:
$$\xymatrix{
& \Gamma_\infty\backslash \mathcal H \ar[dl]_{\pi_H}\ar[dr]^{\pi_B}& \\
\Gamma\backslash \mathcal H && N\backslash \mathcal H.
}$$
Here $\mathcal H$ is the complex upper half plane, $\Gamma$ an arithmetic lattice in $\SL_2(\RR)$, $N\subset \SL_2(\RR)$ the subgroup of upper triangular matrices, acting by horizontal translation, and $\Gamma_\infty = \Gamma \cap N$.

In this diagram, the left arrow is an isomorphism in a neighborhood of the cusp. (A rigorous definition of what ``cusp'' means is coming up.) We will call such a neighborhood a ``Siegel neighborhood''. (Notice that this is \emph{different} from a ``Siegel set'', which is supposed to cover all of $[H]$.) On the other hand, the space $\infty_B=A\backslash [H]_B$ is naturally glued to $[H]_B$ ``at the cusp'': simply consider
\begin{equation}\label{embeddingHB} \RR_{\ge 0}\times^{A} [H]_B,\end{equation}
where we have identified $A$ with $\RR^\times_+$ by an \emph{anti-dominant} cocharacter (with respect to $B$).\footnote{This identification $A\simeq \RR^\times_+$ is convenient for the description of compatifications (because it is more convenient to embed $\Rplus$ into $\RR$, instead of $\Rplus \cup\{\infty\}$, but it contradicts standard choices for the spectral decomposition. Therefore, later in this paper we will be working with the identification $A\simeq \Rplus$ defined by the \emph{dominant} character $\delta^\frac{1}{2}$.} Then the cusp is $\infty_B = \{0\}\times^A [H]_B$. (The opposite ``infinity'' in $[H]_B$, the one we would get if we were glueing $\infty$ to $\Rplus$, instead of $0$, will be called the ``funnel''.) We can then glue $A\backslash [H]_B$ to $B(k)\backslash H(\adele) $ obtaining an embedding $\overline{B(k)\backslash H(\adele)}$, with the topology generated by the open subsets of the latter, and the preimage of open subsets of $\RR_{\ge 0}\times^{A} [H]_B$, and use the identification of $B(k)\backslash H(\adele) $ with $[H]$ close to the cusp to construct the reductive Borel--Serre compactification $\overline{[H]}$. As mentioned before, a \emph{Siegel neighborhood} $\mathscr S_B$ of the cusp will be a neighborhood of $\infty_B$ in $\overline{[H]}$, intersected with $[H]$,\footnote{Throughout the paper, when we mention neighborhoods of the cusp in various compactifications, we will be identifying them with intersections with the open orbit ($[H],[G],[H]_B$ etc.); thus, for example, a ``compact neighborhood of the cusp in $\overline{[H]}$'' is the intersection with $[H]$ of a compact neighborhood of the cusp in $\overline{[H]}$.} with the property that it maps isomorphically onto its image under the map $\pi_G$ of \eqref{leftright-H}.

Now we move to $[G]=[H]\times[H]$. Of course, it has a reductive Borel--Serre compactification $\overline{[G]} := \overline{[H]}\times \overline{[H]}$, but this is not the one that will be relevant for our purposes. A general theory of ``equivariant toroidal'' compactifications was presented in \cite{SaStacks}; here, I describe only one of them, denoted $\overline{[G]}^D$ (for ``diagonal''), which will be useful for analyzing kernel functions. It is obtained by blowing up $\overline{[G]}$ at the product $\infty_B\times \infty_B$ of the two cusps, and then removing $\infty_B  \times [H] \cup [H]\times \infty_B$. Another way to describe it is to consider the partial compactification $Y$ of $\RR^\times_+ \times \RR^\times_+$ obtained by adding a divisor in the direction of the cocharacter $t\mapsto (t,t)$ (as $t\to 0$); it is isomorphic to $\RR^\times_+ \times \RR_{\ge 0}$ after the non-canonical change of variables $(u,v) = (\frac{x}{y}, x)$ or $(u,v)=(\frac{x}{y}, y)$. We then let
\begin{equation}\label{embeddingGB}
 \overline{[G]_B}^D := Y \times^{A\times A} [G]_B, 
\end{equation}
where\footnote{We should really be writing $[G]_{B\times B}$ here, since $B$ denotes a Borel subgroup of $H$, but for notational convenience we write just $[G]_B$. The same convention applies to constant terms.} $[G]_B = [H]_B \times [H]_B$, and glue the orbit at infinity to $[G]$ as before, using the diagram \eqref{leftright-H} for $[H]\times[H]$. The orbit at infinity will be denoted by $\infty_D$; it is isomorphic to $A^\diag\backslash [G]_B$.

\begin{remark}[{``Semi-algebraic'' topology on $[H]$ and $[G]$}]

Since we will be working with Schwartz functions, it makes sense to work with a restricted ``semi-algebraic'' topology, where neighborhoods are described by polynomial inequalities. Without further mention, when talking about neighborhoods of various orbits at infinity on $\overline{[G]}$, $\overline{[G]}^D$, and the corresponding embeddings of $[G]_B$, we will mean this kind of ``semi-algebraic'' neighborhoods.

One way to describe what this means is to choose a Borel subgroup $B$ of $H$ and a good compact subgroup $K$ of $G(\adele)$ so that $G(\adele) = (B\times B)(\adele) \cdot K$; hence, $[G]_B = [A\times A]\cdot K = (A\times A) \times  ([A]^1\times [A]^1) \cdot K$ (where $[A\times A]$ is considered as a subspace of $[G]_B$ by the choice of $B$, and $[A]^1$ denotes idele classes of norm $1$). Then, our neighborhoods in $[G]_B$ should be of the form (or contain sub-neighborhoods of the form) $U=V\times ([A]^1\times [A]^1)K$, where $V$ is a subset of the real semi-algebraic torus $A\times A$  described by polynomial inequalities. By taking images via (the analog of) \eqref{leftright-H}, this induces a notion of semi-algebraic neighborhoods in the vicinity of $\infty_B\times\infty_B$ on $[G]$, and the analogous construction can be used for any class of parabolic subgroups.

This remark was not important in the rank-one case, because the orbit at infinity $\infty_B=A\backslash [H]_B$ is compact in that case, thus any neighborhood of the cusp contains a semi-algebraic neighborhood. 
\end{remark}

\subsection{Schwartz and asymptotically finite functions}\label{ssSchwartz}

We define a height function $g\mapsto \Vert g\Vert$ on $[H]$, whose precise definition is actually not important -- only its ``polynomial equivalence class'' matters. For example, we can fix a Borel subgroup $B\subset H$ and a good maximal compact subgroup $K$ of $H(\adele)$ so that $H(\adele)=B(\adele)K$, and a Siegel neighborhood $U$ of the cusp, and set:
$$ \Vert g \Vert = \begin{cases}
                    1, \mbox{ if } g\notin U;\\
                    \delta(b), \mbox{ if } g\in U\mbox{ and }g = bk,\,\, b\in B(\adele), k\in K.
                   \end{cases}$$
                   
We let $\mathcal S([H])$, the Schwartz space of $[H]$, be the space of smooth functions on $[H]$ which, together with their derivatives under the action of the Lie algebra of $H(k_\infty)$, are of rapid decay, that is:
$$ \sup_{[H]}(|X\varphi(g)| \cdot \Vert g \Vert^r) <\infty$$
for any $r$, where $X$ denotes any element of  $U(\mathfrak h)$, the universal enveloping algebra of $\mathfrak h(k_\infty) \otimes_{\RR} \CC$. Its $K$-invariants, for any open compact subgroup $K$ of the finite adeles of $H$, naturally form a nuclear Fr\'echet space; thus, $\mathcal S([H])$ is an ``$LF$-space''. In the function field case, one should replace ``rapid decay'' by compact support, and ignore references to the topology. (There is a way to define a completion of the space that has a similar topology as in the Archimedean case -- cf.\ \cite[Appendix A]{SaBE1} -- but we won't get into that.)

Similar definitions hold for the space $[H]_B$, and are left to the reader; the Schwartz functions here should be of rapid decay in both the direction of the cusp and the direction of the ``funnel''. Similar definitions also hold for the spaces $[G]$, $[G]_B$, etc., and we have: $\mathcal S([G])=\mathcal S([H])\hat\otimes\mathcal S([H])$.
For a \emph{closed} subset $U$ of either of these spaces (say, $[H]$), we denote by $\mathcal S(U)$ the space of \emph{restrictions} of Schwartz functions on the ambient space, endowed with the quotient topology for $\mathcal S([H])\to \mathcal S(U)$.

\begin{remark}\label{remarkrapid}
Throughout the paper, ``rapid decay'' will always refer to a function together with its polynomial derivatives.
\end{remark}

Let 
\begin{equation}
 \pi_s:= C^\infty(A\backslash [H]_B, \delta^{\frac{s}{2}}),
\end{equation}
the space of smooth functions on $[H]_B$ which are $\delta^{\frac{s}{2}}$-eigenfunctions with respect to the normalized action \eqref{action-normalized-H} of $A\simeq\Rplus$.

\begin{definition}\label{asymptfinite-H}
A smooth function $f$ on $[H]_B$ will be called ``asymptotically finite with exponent $s \in \CC$'' (implicitly: in the direction of the cusp) if there is an element $f^\dagger\in \pi_s$ such that $f-f^\dagger$ is of rapid decay in a compact Siegel neighborhood of the cusp, while $f$ coincides with an element of $\mathcal S([H]_B)$ away from a compact neighborhood of the cusp. The space of asymptotically finite functions on $[H]_B$ with exponent $s$ will be denoted by $\mathcal S_s^+([H]_B)$ -- its $K$-invariants, for every open compact subgroup $K$ of the finite adeles, naturally form a nuclear Fr\'echet space.

A smooth function $\varphi$ on $[H]$ will be called ``asymptotically finite with exponent $s \in \CC$'' if there is an element $\varphi^\dagger\in \pi_s$ such that
$$ \varphi|_{\mathscr S_B} - \varphi^\dagger|_{\mathscr S_B} \in \mathcal S(\mathscr S_B),$$
where $\mathscr S_B$ is any closed Siegel neighborhood of the cusp, and the functions $\varphi$, $\varphi^\dagger$ are identified as functions on this set by pulling back through the maps $\pi_H$, resp.\ $\pi_B$ of \eqref{leftright-H}. The space of asymptotically finite functions on $[H]$ with exponent $s$ will be denoted by $\mathcal S_s^+([H])$ -- its $K$-invariants, for every open compact subgroup of the finite adeles, naturally form a nuclear Fr\'echet space.
\end{definition}

Recall that the action of $A$ on functions on $[H]_B$ has been normalized as in \eqref{action-normalized-H}, so $f^\dagger$, in the above definition, satisfies $f^\dagger(ax) = \delta^{\frac{1+s}{2}}(a) f^\dagger(x)$ for $a\in A$ (and same for  $\varphi^\dagger$). The classical analog for the space $\Gamma\backslash \mathcal H$ is that the function is equal to a multiple of $y^{\frac{1+s}{2}}$, up to a function of rapid decay.

The space $\mathcal S^+_s([H])$ is naturally the space of Schwartz sections of a line bundle over the reductive Borel--Serre compactification $\overline{[H]}$. (Equivalently: smooth sections, since the space is compact.) The natural Fr\'echet space structure on $\mathcal S^+_s([H])^{K}$, mentioned in the definition, is given by the $C^\infty$-seminorms of $\varphi$ away from the cusp, of $\varphi^\dagger \in \pi_s:= C^\infty(A\backslash [H]_B, \delta^{\frac{s}{2}})$, and of $\varphi|_{\mathscr S_B} - \varphi^\dagger|_{\mathscr S_B} \in \mathcal S(\mathscr S_B)$. (See the definition preceding Remark \ref{remarkrapid} for what the Schwartz space of a closed set means! It is a quotient of the Schwartz space of an open neighborhood.) A similar description holds for the topology on $\mathcal S^+_s([H]_B)$. By the definition, we have short exact sequences:
\begin{equation}\label{short-SsHB} 0 \to \mathcal S([H]_B) \to \mathcal S^+_s([H]_B) \to \pi_s \to 0
\end{equation}
and 
\begin{equation}\label{short-SsH} 0 \to \mathcal S([H]) \to \mathcal S^+_s([H]) \to \pi_s \to 0.
\end{equation}

Now we come to $[G]=[H]\times [H]$. We will only define asymptotically finite functions with respect to the partial compactification $\overline{[G]}^D$ that we saw in the previous section. (For more general definitions, s.\ \cite{SaStacks}.) Recall that the orbit at infinity for this compactification is denoted by $\infty_D$($\simeq A^\diag\backslash [G]_B$).

 We normalize the action of $A\times A$ on $C^\infty([G]_B)=C^\infty([H]_B \times [H]_B)$ as in \eqref{action-normalized-H} (on each factor). In particular, the action of the \emph{diagonal} copy $\{a=(x,x)\}$ of $A$ is as follows:
$$ a\cdot f(x) = \delta^{-1}(a) f(ax),$$
while the action of the \emph{anti-diagonal} copy $\{a=(x,x^{-1})\}$ of $A$ is equal to the unnormalized action:
$$ a\cdot f(x) = f(ax).$$

We first consider, given an $s\in\CC$, the complex line bundle over $\infty_D$ whose sections are $(A^\diag, \delta^\frac{s}{2})$-equivariant functions on $[G]_B$ (under the normalized action); its smooth sections will simply be denoted by $C^\infty(\infty_D, \delta^\frac{s}{2})$, or by $C^\infty(A^\diag\backslash [G], \delta^\frac{s}{2})$.

The following is a very important observation:
\begin{observation}\label{observation-Schwartz}
Choose any section of the homomorphism $A\times A \to A^\diag$, and any Borel subgroup $B$ of $H$ and compact subgroup $K$ of $G$ such that $G(\adele)=(B\times B)(\adele)\cdot K$. These data give rise to an identification $[G]_B = (A\times A) \times ([A]^1\times [A]^1)K$, and a section $\sigma$ of the quotient map $[G]_B \to A^\diag\backslash [G]_B$, which we can use to trivialize the bundle defined above, i.e., for $f\in C^\infty(A^\diag\backslash [G], \delta^\frac{s}{2})$, we have $\sigma^* f \in C^\infty(A^\diag\backslash [G])$. Then, whether $\sigma^* f \in \mathcal S(A^\diag\backslash [G])$ (i.e., is of rapid decay) does not depend on the chosen section; more precisely, pull-back via $\sigma$ defines a subspace
$$ \mathcal S(A^\diag\backslash [G], \delta^\frac{s}{2}) \subset C^\infty(A^\diag\backslash [G], \delta^\frac{s}{2}),$$
with a Fr\'echet space structure that does not depend on the choices made to define $\sigma$.
\end{observation}

Indeed, this is easily seen since the character $\delta^\frac{s}{2}$ is of polynomial growth, while elements of $ S(A^\diag\backslash [G])$ decay (together with their derivatives) faster than polynomials.

\begin{definition}\label{asymptfinite-G}
A (smooth) function $f$ on $[G]_B=[H]_B\times [H]_B$ will be called ``asymptotically finite with exponent $s \in \CC$'' (with respect to the diagonal direction $D$, which will not explicitly appear in the notation) if there is an element $f^\dagger\in \Pi_s:= \mathcal S(A^\diag\backslash [G]_B, \delta^{\frac{s}{2}})$, and a compact Siegel neighborhood $U$ of the product $\infty_B\times\infty_B$ of the two cusps in $\overline{[H]}\times \overline{[H]}$, such that 
\begin{enumerate}
\item 
$f|_{U\cap [G]_B} - f^\dagger|_{U\cap [G]_B} \in \mathcal S(U\cap [G]_B)$, and
\item $f|_{[G]_B\smallsetminus \mathring U} \in  \mathcal S([G]_B\smallsetminus \mathring U).$
\end{enumerate} 
The space of asymptotically finite functions on $[G]_B$ with exponent $s$ will be denoted by $\mathcal S_s^+([G]_B)$. Its $K\times K$-invariants, for any compact open subgroup $K$ of the finite adeles of $H$, have a natural nuclear Fr\'echet space structure. 

A (smooth) function $\varphi$ on $[G]=[H]\times [H]$ will be called ``asymptotically finite with exponent $s \in \CC$'' if there is an element $\varphi^\dagger\in \Pi_s$, and a compact Siegel neighborhood $U$ of the product $\infty_B\times\infty_B$ of the two cusps in $\overline{[H]}\times \overline{[H]}$) such that:
\begin{enumerate}
\item 
$$ \varphi|_{U\cap [G]} - \varphi^\dagger|_{U \cap [G]} \in \mathcal S(U\cap [G]),$$
and
\item 
$$ \varphi|_{[G]\smallsetminus \mathring U} \in  \mathcal S([G]\smallsetminus \mathring U).$$
\end{enumerate}
The space of asymptotically finite functions on $[G]$ with exponent $s$ will be denoted by $\mathcal S_s^+([G])$. Its $K\times K$-invariants, for any compact open subgroup $K$ of the finite adeles of $H$, have a natural nuclear Fr\'echet space structure. 
\end{definition}

By the definition, we have short exact sequences
\begin{equation}\label{short-SsGB} 0 \to \mathcal S([G]_B) \to \mathcal S^+_s([G]_B) \to \Pi_s \to 0,
\end{equation}
and
\begin{equation}\label{short-SsG} 0 \to \mathcal S([G]) \to \mathcal S^+_s([G]) \to \Pi_s \to 0.
\end{equation}

\subsection{Regularized integrals} \label{ssregints}

In \cite[\S 5.6]{SaStacks}, I defined a ``regularized integral'' of asymptotically finite functions, provided the exponents are not ``critical''. In the special cases under consideration here, it translates to the following:

\begin{proposition}
 Let $s_1, s_2\in \CC$, and $\varphi_i \in \mathcal S^+_{s_i}([H])$ ($i=1,2$). Consider any holomorphic section $\CC\ni t\mapsto f_t\in \mathcal S^+_t([H])$ with $f_{s_1+s_2} =\varphi_1 \varphi_2$. Then the integral
 $$ \int_{[H]} f_t(h) dh$$
 converges for $\Re(t)\ll 0$, and has meromorphic continuation with a simple pole at $t=s_1+s_2$. If, in particular, $s_1+s_2\ne 0$, its value at $t=0$ depends only on $\varphi_1 \varphi_2$, and defines an invariant bilinear pairing 
 $$ \mathcal S^+_{s_1}([H]) \hat\otimes \mathcal S^+_{s_2}([H]), \to \CC$$
 extending the pairing $\left<\varphi_1 , \varphi_2 \right>_{[H]}  = \int_{[H]} \varphi_1 \varphi_2$ on Schwartz spaces.
 
 The same holds if we replace $[H]$ by $[H]_B$, $[G]$, or $[G]_B$, with the definitions of asympotically finite functions given previously.
\end{proposition}

The holomorphic structure on the family of spaces $\mathcal S^+_t([H])$, mentioned in the definition, can be defined by fixing an Iwasawa decomposition $H(\adele) = B(\adele) K$, setting $\Delta(bk) = \delta^\frac{1}{2}(b)$, picking a smooth cutoff function $u$ on $[H]$ which is equal to $0$ away from a Siegel neighborhood of the cusp and $1$ in a smaller neighborhood, and identifying the spaces via multiplication by $(1-u) + u\cdot \Delta^t$, where the restriction of $\Delta$ to the Siegel neighborhood is understood as a function on $[H]$.

\textbf{Notice that the pairings $\left<\,\, , \,\,\right>$ are always bilinear in this paper. By abuse of language, we will be using the term ``inner product'' to refer to them.}

\begin{definition}
When $s_1+s_2\ne 0$, the invariniant bilinear pairing
$$ \mathcal S^+_{s_1}([H]) \hat\otimes \mathcal S^+_{s_2}([H]) \to \CC$$
of the previous proposition will be denoted
$$ \varphi_1\otimes\varphi_2 \mapsto \left<\varphi_1 , \varphi_2 \right>^*_{[H]}  = \int_{[H]}^* \varphi_1 \varphi_2.$$
Completely analogous notation will be used for the regularized pairings for the spaces $[H]_B$, $[G]$, and $[G]_B$. 
\end{definition}

\subsection{Approximation by constant term}

The theorems presented in this section are very classical and well-known; they generalize to any reductive group as in {\cite[Corollaries I.2.8, I.2.11]{MW}}.

Again, we start with the rank-one case $[H]=[\PGL_2]$. 
For $\varphi\in C([H])$, we denote by $\varphi_B \in C([H]_B)$ its constant term:
\begin{equation}\label{constantterm-B}
\varphi_B( g) = \int_{[N]} \varphi(ng) dn.
\end{equation}

\begin{theorem} \label{approxconstantterm-H}
Fix $K=$ a compact open subgroup of $H(\adele_f)$, where $\adele_f=$ the finite adeles. (In the function field case, $\adele_f=\adele$.) 

In the function field case, there is a Siegel neighborhood $\mathscr S_B$ such that $\varphi|_{\mathscr S_B} = \varphi_B|_{\mathscr S_B}$ for all $\varphi\in C^\infty([H])^K$. 

In the number field case, fix also a number $r>0$. Let $V_r\subset C([H])^K$ be the Banach space defined by the norm: $\Vert \varphi\Vert_r = \sup_{g\in [H]} (|\varphi(g)| \cdot \Vert g\Vert^{-r})$, and let  $V_r^\infty\subset C^\infty([H])^K$ be the Fr\'echet space of its smooth vectors. (Hence, a complete system of seminorms on $V_r^\infty$ consists of the seminorms 
$$\Vert \varphi\Vert_{X,r} := \Vert X\varphi \Vert_r,$$
for $X\in U(\mathfrak h)$, the universal enveloping algebra of $\mathfrak h(k_\infty) \otimes_{\RR} \CC$.)

Then, the map $\varphi \mapsto \varphi|_{\mathscr S_B} - \varphi_B|_{\mathscr S_B}$ represents a continuous morphism: 
$$V_r^\infty \to \mathcal S(\mathscr S_B),$$
where $\mathscr S_B$ denotes any closed Siegel neighborhood of the cusp, considered as a subset of $B(k)\backslash H(\adele)$ (and the functions $\varphi$, $\varphi_B$ are ``restricted'' to this set by pulling back through the maps $\pi_H$, resp.\ $\pi_B$ of \eqref{leftright-H}). 

\end{theorem}

\begin{example}\label{HmodN}
Let $X= N\backslash H$, where $N\subset B$ is the unipotent radical of a chosen Borel subgroup. Consider the morphism
$$\Sigma: \mathcal S(X(\adele)) \to C^\infty([H])$$
given by $\Sigma\Phi(g) := \sum_{\gamma\in X(k)} \Phi(\gamma g)$.

\begin{proposition*}\label{HmodNprop}
This map defines a continuous morphism
$$ \mathcal S(X(\adele))\to \mathcal S([H]).$$
\end{proposition*}

\begin{proof}
Obviously, the map to $C^\infty([H])$ is continuous. By generalities about representations of moderate growth (s.\ \cite{BeKr}), it will actually be continuous as a map from $\mathcal S(X(\adele))$ to some $V_r^\infty$. By Theorem \ref{approxconstantterm-H}, it is enough to show that the restriction of the constant term $(\Sigma\Phi)_B$ to a neighborhood of the cusp coincides with a restriction of an element of $\mathcal S([H]_B)$.

We compute this constant term:
$$ (\Sigma\Phi)_B(g) = \int_{[N]} \sum_{\gamma\in X(k)}\Phi(\gamma ng) dn. $$
We break up the sum into Borel orbits (Bruhat cells, in this case). It becomes:
$$ (\Sigma\Phi)_B(g) =  \sum_{\alpha\in A(k)}\Phi(\alpha g) + \sum_{\alpha\in A(k)} \int_{N(\adele)} \Phi(\alpha w g), $$
where $w$ is the non-trivial element of the Weyl group.

The first term is an element of $\mathcal S([H]_B)$. The second is rapidly decaying at the cusp, because in the quotient $X\sslash N = \spec k[X]^N \simeq \Ga$ (the last isomorphism by identifying the image of the closed Bruhat cell with $0$), the action of $A$ is via the character $\delta$ (which goes to $\infty$ at the cusp).

\end{proof}
\end{example}

\begin{example}\label{torus}
 Let $T=\Gm\hookrightarrow H=\PGL_2$, $X=T\backslash H$. We consider the morphism 
$$\Sigma: \mathcal S(X(\adele)) \to C^\infty([H])$$
given by $\Sigma\Phi(g) := \sum_{\gamma\in X(k)} \Phi(\gamma g)$.

\begin{proposition*}
This map defines a continuous morphism
$$ \mathcal S(X(\adele))\to \mathcal S^+_{1}([H]).$$
\end{proposition*}

\begin{proof}
Again, it is enough to show that the restriction of the constant term $(\Sigma\Phi)_B$ to a neighborhood of the cusp coincides with a restriction of an element of $\mathcal S^+_{1}([H]_B)$.

We compute this constant term:
$$ (\Sigma\Phi)_B(g) = \int_{[N]} \sum_{\gamma\in X(k)}\Phi(\gamma ng) dn. $$

Assume that $B$ has been chosen to contain $T$ (thus furnishing an isomorphism $T\simeq A$, the universal Cartan). We split the above sum into Borel orbits on $X(k)$; there are three of them:
$$ T 1 B,\,\,\, T w B,\,\,\, \mbox{ and } T\xi B,$$
where $w$ denotes the non-trivial element of the Weyl group of $T$ in $H$ and $\xi$ is a representative of the open Borel orbit. 

The sum becomes:
$$ \int_{N(\adele)} \Phi(T 1 n g) dn + \int_{N(\adele)} \Phi(T w n g) dn + \sum_{\alpha\in A(k)} \int_{N(\adele)} \Phi(T \xi \alpha n g) dn.$$

The affine quotient $X\sslash N = \spec k[X]^N$ can be identified with $\Ga$ by letting $0$ be the (common) image of the orbits $T1B$ and $TwB$, and then $A$ acts on $\Ga$ via the character $\delta$. Therefore, as $\delta(a)\to \infty$ the image of the point $\xi\cdot a$ in $X\sslash N$ goes to $\infty$, and from this it is easy to see that the third term is rapidly decaying at the cusp. Each of the other two terms is $\delta$-equivariant under the unnormalized $A(\adele)$-action, and hence $\delta^\frac{1}{2}$-equivariant under the normalized $A(\adele)$-action. 

Hence, from the above formula and Theorem \ref{approxconstantterm-H}, we deduce that $\Sigma$ is actually a continuous map into $ \mathcal S^+_{1}([H])$.
\end{proof}

\end{example}

\begin{remark}
We see in the above example that the function $\Sigma\Phi$ is asymptotic not only to an eigenfunction for the group $A\subset A(k_\infty)$, but also for all of $[A]$. This suggests that our compactifications are too large: we could have glued $[A]\backslash [H]_B$ at infinity. This would have been  achieved by replacing \eqref{embeddingGB} by
$$ \overline{[A]} \times^{[A]} [H]_B,$$
where $\overline{[A]}$ is the union of $[A]$ with a point as $\delta(a)\to \infty$. However, for the analysis it will not make any difference which compactification we are using, and the real torus $A$ is easier to handle, notationally.
\end{remark}

Now we move to rank two. The group $G=H\times H$ has four classes of parabolics, namely: $G$, $H\times B$, $B\times H$ and $B\times B$. For any conjugacy class $P=MU$ of $k$-parabolics, we denote by $[G]_P$ the ``boundary degeneration''  $M(k)U(\adele)\backslash G(\adele)$. (Recall that, by abuse of notation, we denote $[G]_{B\times B}$ simply by $[G]_B$.) It carries an action of $[A_P]$, commuting with the $G(\adele)$-action, where $A_P$ is the maximal split torus in the center of the Levi quotient $M$.

For $\varphi\in C([G])$, we denote by $\varphi_P \in C([G]_P)$ its constant term:
\begin{equation}\label{constantterm-P}
\varphi_P( g) = \int_{[U]} \varphi(ug) du.
\end{equation}

Consider first the reductive Borel--Serre compactification of $[G]$:
$$ \overline{[G]} = \overline{[H]} \times \overline{[H]}.$$

It contains four $G(\adele)$-orbits, parametrized by the classes of parabolics $P$. We will call ``$P$-cusp'' the closure of the orbit corresponding to $P$; for example, the $B\times H$-cusp is the product $\infty_B\times \overline{[H]}$. A ``Siegel neighborhood'' of a cusp is a semi-algebraic neighborhood, intersected with $[G]$, on which the left arrow of the diagram 
\begin{equation}\label{leftright-G} \xymatrix{
& P(k)\backslash G(\adele) \ar[dl]_{\pi_G}\ar[dr]^{\pi_P}& \\
[G] && [G]_P
}\end{equation}
is an isomorphism onto the image.

The theorem that follows is the theorem of approximation by the constant term, essentially as it appears in \cite[Corollaries I.2.8, I.2.11]{MW}. In it, a \emph{Siegel set} $\mathscr S$ is a compact neighborhood of the cusp in the ``reductive Borel--Serre compactification''  $\overline{(B\times B)(k)\backslash G(\adele)}$, which is large enough that it surjects onto $[G]$. We will afterwards discuss what it means in terms of Siegel neighborhoods, relating it to the statement for the function field case. We endow $[G]$ with the product of the norms $\Vert \cdot \Vert$ on $[H]$, also to be denoted by $\Vert \cdot \Vert$.

\begin{theorem}  \label{approxconstantterm}
Fix $K=$ a compact open subgroup of $G(\adele_f)$.

In the function field case, for every $P$, there is a Siegel neighborhood ${\mathscr S}_P$ of the $P$-cusp such that $\varphi|_{{\mathscr S}_P} = \varphi_P|_{{\mathscr S}_P}$ for all $\varphi \in C^\infty([G])^K$. 

In the number field case, fix also a number $r>0$ and a Siegel set $\mathscr S$. Let $V_r\subset C([G])^K$ be the Banach space defined by the norm: $\Vert \varphi\Vert_r = \sup_{g\in [H]} |\varphi(g)| \cdot \Vert g\Vert^{-r}$, and let  $V_r^\infty\subset C^\infty([G])$ be the Fr\'echet space of its smooth vectors.

Then, the map 
\begin{equation}
\varphi \mapsto \sum_{P} (-1)^{\operatorname{rk}(G)-\operatorname{rk}(P)} \varphi_P|_{\mathscr S} 
\end{equation}
represents a continuous morphism: 
$$V_r^\infty \to \mathcal S(\mathscr S).$$
\end{theorem}

\begin{remark}
The theorem has been formulated taking into account that the center of $G$ is trivial/compact. If the center of $G$ was not compact, in the number field case we would need to specify that the rapid decay is in the direction ``transverse to the center'' -- cf.\ \cite[Corollary  I.2.11]{MW} for the complete formulation. It is these central directions that prevent the formulation of the function field case from being true in the number field case -- the ``leftovers'' from the difference between a function and its constant term do not need to be of rapid decay in ``central'' directions.

More specifically, on a Siegel neighborhood $\mathscr S_B$ (we use the same shorthand notation for $\mathscr S_{B\times B}$, $\varphi_{B\times B}$ as for $[G]_B$) of $\infty_B\times \infty_B$ one does not only see the contributions of $\varphi_B$, but also terms that decay rapidly as we approach $\infty_B\times \infty_B$, but not, for example, as we approach $\infty_B\times [H]$ inside of $\mathscr S_B$. The following picture will help us visualize what is happening:

\begin{center}
\includegraphics[scale=0.4]{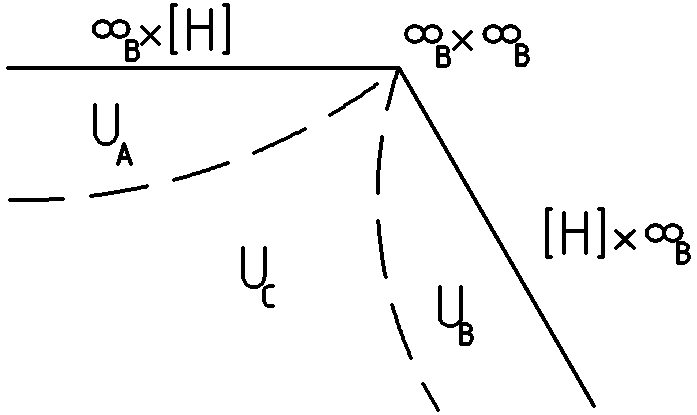}
\end{center}

Here the $U_A$ and $U_B$ are (semi-algebraic) neighborhoods of the open orbits in the corresponding cusps, and $U_C$ is their complement, all within a Siegel neighborhood $\mathscr S_B$. The difference $\varphi - \varphi_B$ coincides with the restriction of an element of $\mathcal S([G])$ when restricted to $U_C$, but not necessarily on $U_A$ or $U_B$. Similarly, $(\varphi - \varphi_{B\times H})|_{U_B}$ and  $(\varphi - \varphi_{H\times B})|_{U_A}$ coincide with restrictions of elements of $\mathcal S([G])$. Hence, $\varphi$ is approximated by $\varphi_B$ on $U_C$, by $\varphi_{B\times H}$  on $U_B$ and by $\varphi_{H\times B}$ on $U_A$; moreover, for each $P=H\times B$ or $B\times H$, the difference $\varphi_P-\varphi_B$ is of rapid decay as we approach the cusp in a direction transverse to the $A_P$-orbits, but \emph{not} necessarily in the ``central'' direction of $A_P$-orbits (i.e., those orbits whose limit is a point of $\infty_B\times [H]$ or $[H]\times\infty_B$ not belonging to $\infty_B\times\infty_B$).

\end{remark}

\section{Spectral decomposition in rank one}

The goal of this section is to develop a ``Plancherel formula'' for the (regularized) inner product of asymptotically finite functions on $[H]=[\PGL_2]$. It will be our warm-up for the case of $[G]=[H]\times[H]$, on which the Selberg trace formula needs to be understood. But it is also necessary for understanding simpler instances of the relative trace formula, such as the relative trace formula for $(X\times X)/G$, where $X=\Gm\backslash\PGL_2$ (s.\ Example \ref{torus}).

Let $s_1,s_2$ be two complex numbers, and let $\varphi_i\in \mathcal S^+_{s_i}([H])$ be two asymptotically finite functions. There is actually no reason to restrict ourselves to a single exponent at infinity -- we could consider functions with several exponents, possibly with multiplicity. This is indeed what we will do in our ``model'' space $\RR^\times_+$, but for notational simplicity we avoid it for $[H]$.

As we saw in \S \ref{ssregints}, the regularized ``inner product''
$$\left< \varphi_1 ,\varphi_2\right>^* = \int^*_{[H]} \varphi_1(h) \varphi_2(h) dh$$
makes sense and is an $H(\adele)$-invariant bilinear pairing
$$ \mathcal S^+_{s_1}([H]) \hat\otimes \mathcal S^+_{s_2}([H]) \to \CC,$$
whenever $s_1+s_2\ne 0$. Recall that we are using the name and notation of inner product for a bilinear pairing, by abuse of language. Recall also that the exponents refer to the normalized action of $A$ on $[H]_B$, cf.\ 
\eqref{action-normalized-H}. Whenever necessary to remind the space on which this bilinear pairing is taken, we will write $\left< \,\, ,\,\,\right>_{[H]}^*$. Completely analogous definitions and symbols will be used for the boundary degenerations.

Our model for analyzing the above inner product on $[H]$ is the real group $\RR^\times_+$, and we start from that.

\subsection{Paley--Wiener theorem for Schwartz functions on $\Rplus$}

\begin{theorem}\label{PW-torus1D}
Mellin transform:\footnote{Mellin transform is traditionally defined by replacing $s$ by $-s$ in the above definition; this is unfortunate, because it means that the functional $f\mapsto \check f(s)$ would be equivariant with respect to the character parametrized by $-s$. We choose this definition, which is compatible with equivariance properties of the map.}
\begin{equation}\label{Mellin} \check f (s) = \int_{\Rplus} f(x) x^{-s} d^\times x\end{equation}
is an isomorphism between $\mathcal S(\Rplus)$ and the Fr\'echet space $\mathbb H^\PW(\CC)$ of entire functions on $\CC$ which, on every bounded vertical strip $V$ and for every $N$ satisfy 
\begin{equation}\label{vertstrip} \sup_{s\in V} |f(s)| (1+|\Im(s)|)^N <\infty.\end{equation}

The inverse map is given by 
\begin{equation}\label{Mellininv} f(x) = \frac{1}{2\pi i}\int_{\sigma-i\infty}^{\sigma+i\infty} \check f(s) x^s ds\end{equation}
for any $\sigma \in \RR$.
\end{theorem}

The Fr\'echet space structure is, of course, determined by the norms of the estimate. In the literature, the Paley-Wiener theorem is usually formulated for compactly supported smooth functions, in which case there is also an exponential growth condition in the $\Re(s)$-direction, so let us briefly prove the above version. We denote throughout by $\partial$ the generator $x\frac{d}{dx}$ of the Lie algebra of $\Rplus$.

\begin{proof}
The Mellin transform is clearly entire, and by standard Fourier theory, 
$$\int_{\Rplus} x^s \partial^N f(x) d^\times x = (-s)^N \check f(s),$$
so for $m\le \Re(s)\le M$ and any $N$, we have an estimate
$$ |\check f(s)| |\Im(s)|^N \le \int_{\Rplus} \left| x^{-s} \partial^N f(x)  \right|d^\times x \le $$
$$ \le \int_{\Rplus} \left| x^{-m} \partial^N f(x)  \right|d^\times x + \int_{\Rplus} \left| x^{-M} \partial^N f(x)  \right|d^\times x, $$
and both integrals in the last expression are seminorms on $\mathcal S(\Rplus)$.

Vice versa, Mellin inversion \eqref{Mellininv}
satisfies, because of the estimate on vertical strips for $\check f$, 
$$ x^M \partial^N f(x)  = \frac{1}{2\pi i} \int_{\sigma-i\infty}^{\sigma+i\infty} s^N \check f(s) x^{s+M} ds. $$ Since the value of $\sigma$ is arbitrary (by a contour shift argument which is allowed, given the estimate on vertical strips for $\check f$), we can take $\sigma+M=0$, and then we have
$$ \left | x^M \partial^N f(x)   \right| \le \frac{1}{2\pi} \int_{-M-i\infty}^{-M+i\infty} |s^N \check f(s) | ds$$
for all $x$, finishing the proof.
\end{proof}

\subsection{Mellin transform of asymptotically finite functions on the torus}

Let $E^+$, $E^-$ denote two finite multisets of complex numbers, identified with characters (``exponents'') of $\Rplus$ by associating $s\in \CC$ with the character $t\mapsto t^s$. The space $\mathcal S^+_{(E^+,E^-)}(\Rplus)$ is the space of smooth functions which coincide with generalized eigenfunctions with the given characters, in neighborhoods of $0$ ($E^+$) and $\infty$ ($E^-$), up to Schwartz functions. Notice that if $s\in \CC$ appears with multiplicity, it refers to \emph{generalized} eigenfunctions of degree same as the multiplicity, e.g., multiplicity $3$ means: $t\mapsto t^s(c_0+c_1 \log t + c_2 (\log t)^2)$. The space $\mathcal S^+_{(E^+,E^-)}(\Rplus)$ has a natural Fr\'echet space structure, and $\mathcal S^+_{(\emptyset,\emptyset)}(\Rplus) = \mathcal S(\Rplus)$.

Let us start with the case $E^- = \emptyset$. Extending the Paley--Wiener theorem of the previous subsection, we have: 

\begin{theorem}\label{theorem-Mellin-asfinite}
Mellin transform \eqref{Mellin} on $\mathcal S^+_{(E^+,\emptyset)}(\Rplus)$ converges for $\Re(s)\ll 0$, admits meromorphic continuation to $\CC$ with poles at the points of $E^+$ and multiplicities equal to the multiplicity of the exponents, and defines an isomorphism between $\mathcal S_{(E^+,\emptyset )}(\Rplus)$ and the Fr\'echet space $\mathbb H^\PW_{(E^+, \emptyset)}$ of meromorphic functions on $\CC$, with poles at $E^+$ (with the given multiplicities), and satisfying the same estimate \eqref{vertstrip} on bounded vertical strips away from neighborhoods of the poles. The inverse map is given by \eqref{Mellininv} for any $\sigma\ll 0$.
\end{theorem}

\begin{proof}
Notice that, if $\sigma_0 \in E^+$, then  multiplication by $(s-\sigma_0)$ defines an isomorphism
$$\mathbb H^\PW_{(E^+, \emptyset)} \xrightarrow\sim \mathbb H^\PW_{(E^+\smallsetminus \{\sigma_0\}, \emptyset)},$$
where only one copy of $\sigma_0$ is removed from the multiset $E^+$. We will prove, analogously, that the operator $(\partial -s_0)$ defines an isomorphism
$$ \mathcal S_{(E^+,\emptyset)}(\Rplus)  \xrightarrow\sim \mathcal S_{(E^+\smallsetminus \{\sigma_0\},\emptyset)}(\Rplus), $$
thus reducing the statement by induction, and by the formula 
$$ \widecheck{((\partial-s_0) f)}(s) = (s-s_0) \check f(s),$$
to the case of Theorem \ref{PW-torus1D}. 

(These statements only use the Mellin transform in the convergent region, so they also prove its meromorphic continuation.)

To show that the operator $(\partial -s_0)$ defines the aforementioned isomorphism, assume that 
$$ f(x) \sim \sum_{s'} P_{s'}(\log x) x^{s'}$$
as $x\to 0$, where the sum is finite and the $P_{s'}$ are polynomials. 
Then 
$$(\partial -s_0) f (x)\sim \sum_{s'} ((s'-s_0)P_{s'}(\log x)+P'_{s'}(\log x)) x^{s'},$$
so it is immediately clear that $(\partial -s_0) f \in \mathcal S_{(E^+\smallsetminus \{\sigma_0\},\emptyset)}(\Rplus)$, and the coefficients of its asymptotic expansion depend continuously on the coefficients of the asymptotic expansion of $f$. Moreover, for every $N$ and $M$ we have 
$$ \sup_{x\ge 1} \left| x^M\partial^N (\partial -s_0) f(x)\right| \le \sup_{x\ge 1} \left| x^M\partial^{N+1} f(x)\right| + |s_0|\sup_{x\ge 1} \left| x^M\partial^N  f(x)\right|,$$
and
$$ \sup_{x\le 1} \left| x^M\partial^N \left((\partial -s_0) f(x) - \sum_{s'} ((s'-s_0)P_{s'}(\log x)+P'_{s'}(\log x)) x^{s'}\right)\right| = $$
$$  \sup_{x\le 1} \left| x^M\partial^N (\partial -s_0) \left( f(x) - \sum_{s'} P_{s'}(\log x) x^{s'}\right)\right| \le  $$
$$\sup_{x\ge 1} \left| x^M\partial^{N+1} \left(f(x)- \sum_{s'} P_{s'}(\log x) x^{s'}\right)\right| + |s_0|\sup_{x\ge 1} \left| x^M\partial^N  \left(f(x)- \sum_{s'} P_{s'}(\log x) x^{s'}\right)\right|,$$
which shows that the map 
$$(\partial -s_0):  \mathcal S_{(E^+,\emptyset)}(\Rplus)  \to \mathcal S_{(E^+\smallsetminus \{\sigma_0\},\emptyset)}(\Rplus)$$
is continuous. Vice versa, starting from an element $F \in \mathcal S_{(E^+\smallsetminus \{\sigma_0\},\emptyset)}(\Rplus)$, the inverse map is
$$ (\partial -s_0)^{-1} F (x) = -x^{s_0} \int_x^\infty t^{-s_0} F(t) d^\times t,$$
and it is easy to repeat estimates as above to show continuity in the opposite direction. 
\end{proof}

Obviously, the analogous theorem will hold for $\mathcal S^+_{(\emptyset, E^-)}(\Rplus)$, except that Mellin transform converges, and Mellin inversion is given by a continuous integral, for $\sigma \gg 0$. To put both $E^+$ and $E^-$ together, we need a formalism to distinguish whether poles of the Mellin transform belong to $E^+$ or $E^-$.

\begin{definition}\label{charge-definition}
A \emph{charged Laurent expansion} at a point $s_0\in \CC$ is an expression of the form
$$L= (a_{-n}^+ + a_{-n}^-) (s-s_0)^{-n} + \dots + (a_{-1}^+ + a_{-1}^-) (s-s_0)^{-1} + \sum_{i=0}^\infty a_i (s-s_0)^i.$$

The positive and negative residues of a charged Laurent expansion at $s_0$ are defined to be $\Res_{s=s_0}^\pm L = a_{-1}^\pm$.
\end{definition}

In other words, it is a usual Laurent expansion, but with the polar part split into a ``positive'' and a ``negative'' part. Notice that even the zero function can have non-trivial charged Laurent expansions, with $a_{-i}^+ = - a_{-i}^-$.

Let $h$ be a meromorphic function on $\CC$, and $E=(E^+,E^-)$ two multisets of complex numbers. We will say that $h$ is endowed with charged Laurent expansions at $E$ if it comes equipped with charged Laurent expansions at the points of $E$, where the order of the $+$ and the $-$ polar parts is bounded by the corresponding multiplicities on $E^+, E^-$. (In particular, this places a restriction on the order of pole of $h$ at points of $E$.)

We let $\mathbb H^\PW_{(E^+, E^-)}$ be the Fr\'echet space of meromorphic functions on $\CC$, with poles and charged Laurent expansions determined by $(E^+,E^-)$ (in particular, no poles outside of $(E^+,E^-)$), and satisfying the same estimate \eqref{vertstrip} on bounded vertical strips away from neighborhoods of the poles. The implicit Fr\'echet space seminorms here include, separately, the absolute values of the coefficients $a_{-i}^\pm $. 

\underline{\emph{Notational conventions:}}

\begin{itemize}
\item We will be writing 
$$ \Laur_{s_0}^\pm(h)(s) = \sum_{i=-\infty}^{\infty} a_i^\pm (s-s_0)^i,$$
with notation as above, where $a_i^\pm = a_i$ when $i\ge 0$ and $a_i^\pm =0$ for $i\ll 0$. \emph{Notice that the Laurent expansion of $h$ at $s_0$ is not $\Laur_{s_0}^+(h) + \Laur_{s_0}^-(h)$, since the non-polar parts are repeated in both}. Nonetheless, this is a convenient notational convention.

\item For the function $h^\vee: s\mapsto h(-s)$, we switch the charges in the Laurent expansion, i.e.:
$$\Laur_{s_0}^\pm(h^\vee)(s) = \Laur_{-s_0}^\mp(h)(-s).$$
The idea is, as we shall see, that $\pm$ denotes on which side of the integration of Mellin inversion the pole should lie.

\item For a product $h_1\cdot h_2$ of two functions as above, we set
$$ \Laur_{s_0}^+(h_1\cdot h_2)(s) = \Laur_{s_0}^+(h_1)(s)\cdot \Laur_{s_0}^+(h_2)(s),$$
and similarly for negative charges.
\end{itemize}

In our application of this formalism, we will never allow $\Laur_{s_0}^+(h_1)$ to have a pole simultaneously with $\Laur_{s_0}^-(h_2)$ -- if $h_i = \check f_i$, this would be the case when the regularized convolution of $f_1$ and $f_2$ doesn't make sense. The reader can check the following:

\begin{lemma}
Under the assumption that $\Laur_{s_0}^+(h_i)$ does not have a pole if $\Laur_{s_0}^-(h_j)$ does ($i, j= 1,2$, $i\ne j$), the polar part of $\Laur_{s_0}^+(h_1\cdot h_2)+\Laur_{s_0}^-(h_1\cdot h_2)$ coincides with the polar part of $h_1\cdot h_2$.
\end{lemma}

Now let $(E^+, E^-)$ be multisets of exponents for $\Rplus$, where $E^+$ stands for exponents that will appear as $x\to 0$, and $E^-$ for $x\to \infty$. On elements of the space $\mathcal S_{(E^+,E^-)}(\Rplus)$, Mellin transform \eqref{Mellin} can then be understood as a regularized integral. 
\begin{lemma}
 The integral \eqref{Mellin} with $f\in \mathcal S_{(E^+,E^-)}(\Rplus)$, understood as a regularized integral, coincides with the meromorphic continuation of the convergent integral of Theorem \ref{theorem-Mellin-asfinite} when $f\in \mathcal S_{(E^+,\emptyset)}(\Rplus)$ (or $f\in \mathcal S_{(\emptyset,E^-)}(\Rplus)$).
\end{lemma}

This is immediate from the definitions, and left to the reader. As a corollary, for $f\in \mathcal S_{(E^+,E^-)}(\Rplus)$, Mellin transform $\check f$ is a well-defined meromorphic function, with poles determined by the multiset $E^+ \cup E^-$, where the union of multisets is here taken with the \emph{maximum} multiplicity. We ``embellish'' $\check f$ by assigning charges to its Laurent expansion: Write $f = f_+ + f_-$ with $f_+ \in \mathcal S_{(E^+,\emptyset)}(\Rplus)$ and $f_- \in \mathcal S_{(\emptyset, E^-)}(\Rplus)$, and let the ``$+$''-polar part of $\check f$ be the polar part of $\check f_+$, and the ``$-$''-polar part be the polar part of $\check f_-$. Clearly, these definitions do not depend on choices. \emph{From now on, the Mellin transform $\check f$ will be understood not just as a meromorphic function on $\CC$, but as a meromorphic function with charged Laurent expansions.}

\begin{theorem}\label{PWext-torus1D}
If $(E^+, E^-)$ are multisets of exponents for $\Rplus$ ($E^+$ as $x\to 0$, $E^-$ as $x\to \infty$), then Mellin transform as above, defines an isomorphism between $\mathcal S_{(E^+,E^-)}(\Rplus)$ and the Fr\'echet space $\mathbb H^\PW_{(E^+, E^-)}$ of meromorphic functions on $\CC$, with poles and \emph{charged Laurent expansions} determined by $(E^+,E^-)$, and satisfying the same estimate \eqref{vertstrip} on bounded vertical strips away from neighborhoods of the poles.
\end{theorem}

\begin{proof}
We have 
$$\mathcal S_{(E^+,E^-)}(\Rplus) = \left(\mathcal S_{(E^+,\emptyset)}(\Rplus) \oplus \mathcal S_{(\emptyset,E^-)}(\Rplus) \right) / \mathcal S(\Rplus)^\diag,$$
and 
$$ \mathbb H^\PW_{(E^+, E^-)} = \left( \mathbb H^\PW_{(E^+, \emptyset)} \oplus \mathbb H^\PW_{(\emptyset, E^-)}\right) / \mathbb H^{\PW,\diag},$$
which reduces us to the previous theorem.
\end{proof}

\begin{corollary}
For $f\in \mathcal S_{(E^+,E^-)}(\Rplus)$ we have the Mellin inversion formula:
\begin{equation}\label{Mellininv-ext} f(x) = \frac{1}{2\pi i} \int_{\sigma-i\infty}^{\sigma+i\infty} \check f(s) x^s ds - \sum_{\Re(s')<\sigma} x^{s'} \Res_{s=s'}^+ \check f(s) + \sum_{\Re(s')>\sigma} x^{s'} \Res_{s=s'}^- \check f(s),\end{equation} 
where $\sigma \in \RR$ is such that no element of $E^+, E^-$ has real part $\sigma$. 

If $\sigma$ is allowed to be arbitrary, the corresponding expression is:
\begin{eqnarray}\nonumber  f(x) &=& \frac{1}{2\pi i} PV\int_{\sigma-i\infty}^{\sigma+i\infty} \check f(s) x^s ds - \sum_{\Re(s')<\sigma} x^{s'} \Res_{s=s'}^+ \check f(s) + \sum_{\Re(s')>\sigma} x^{s'} \Res_{s=s'}^- \check f(s) \\ 
&+& \frac{1}{2}\sum_{\Re(s')=\sigma} x^{s'} (\Res_{s=s'}^- \check f(s)-\Res_{s=s'}^+ \check f(s)), \label{Mellininv-ext-PV}
\end{eqnarray} 
where $PV$ denotes the Cauchy principal value.

\end{corollary}

\begin{remark}
 It is easy to remember the formulas above, when $\sigma = 0$, as stating that the Plancherel decomposition includes the usual integral over the unitary line, plus discrete contributions (residues) for every exponent that is \emph{superunitary}; that is, the absolute value of the corresponding character is $>1$ in the direction where the exponent appears. If there are unitary exponents (= absolute value 1), they produce poles on the unitary line, and half of the corresponding residues are counted (for each direction).
\end{remark}

\subsection{Inner product of asymptotically finite functions on the torus}

Let $E_1, E_2$ be two exponent arrangements, where now $E_i$ stands for a pair $(E_i^+, E_i^-)$.  Let $f_i \in \mathcal S_{E_i}(\Rplus)$. Assume that no two exponents $s_1 \in E_1^+$, $s_2\in  E_2^+$ satisfy $s_1+s_2=0$, and similarly for $E_1^-, E_2^-$. Then the regularized ``inner product'' 
$$ \left< f_1, f_2\right>^* = \int_{\Rplus}^* f_1(x) f_2(x) d^\times x$$
is defined, in a completely analogous way as in \S \ref{ssregints}. (Again, we take all our pairings to be bilinear, but  by abuse of language we will be calling them ``inner products''.)

The condition on the exponents means the following: Whenever the $+$-part $\Laur^+_{s_0}(\check f_1(s))$ of the Laurent expansion of $\check f_1$ around $s_0$ has a pole, the $+$-part of $\check f_2$ at $-s_0$, $\Laur^+_{-s_0}(\check f_2(s))$, has no pole; and similarly by interchanging $f_1$ with $f_2$, and $+$ with $-$. Notice that in the conventions presented in the previous subsection, $\Laur^+_{-s_0}\check f_2$ corresponds to $\Laur^-_{s_0}$ of the function $s\mapsto \check f_2(-s)$.

\begin{theorem} \label{Plancherel-torus-s}
We have 
\begin{eqnarray} \nonumber
\left< f_1, f_2\right>^*  &=& \frac{1}{2\pi i}\int_{\sigma-i\infty}^{\sigma+i\infty} \check f_1(s) \check f_2(-s) ds \\ &-& \sum_{\Re(s')<\sigma} \Res_{s=s'}^+ \check f_1(s)\check f_2(-s) + \sum_{\Re(s')>\sigma} \Res_{s=s'}^- \check f_1(s) \check f_2(-s). \nonumber \\\label{spectral-torus}
\end{eqnarray}
Here $\sigma \in \RR$ is such that no element of $E_1, E_2$ has real part $\sigma$. If we want to allow such values of $\sigma$, the analogous expression to \eqref{Mellininv-ext-PV} holds.
\end{theorem}

\begin{remark}
This can be seen as the same statement as \eqref{Mellininv-ext}, applied to the regularized convolution $f(x) = \int_{\Rplus}^* f_1(xy) f_2(y) d^\times y$
of the functions $f_1$ and $x\mapsto f_2(x^{-1})$.
\end{remark}

\begin{proof}
There are various ways to prove this theorem, including:
\begin{enumerate}
\item by using the remark above;
\item by meromorphic continuation from subunitary characters (i.e., modifying the sets of exponents $E^\pm$);
\item by subtracting from each function its superunitary asymptotic part, thus reducing the problem to the case of $L^2$-functions. (This method does not generalize to tori of higher rank.) 
\end{enumerate}

\end{proof}

\subsubsection{Extension to functions that are not asymptotically finite} 

For our applications to the trace formula, we will need to extend the validity of the above inner product formulas. The reader can skip this technical subsection at first reading, and only return to it when needed. The issue in the case of automorphic functions is that we are going to have to apply our formulas to constant terms, which cannot be controlled very well away from neighborhoods of the cusps. (They involve poles of intertwining operators/Eisenstein series, i.e., zeroes of $L$-functions.) Thus, we need to weaken our assumptions on the functions, to include functions which in one direction are only asympotically finite ``up to an almost-$L^2$ smooth function''.

More specifically, one of our functions, say $f_1$, will indeed be asymptotically finite close to $\infty$, but of rapid decay close to zero. The other, $f_2$, will be asymptotically finite close to $\infty$, but only asympotically finite ``up to an almost-$L^2$ smooth function'' close to zero, by which we mean the properties stated in the proposition that follows: \marginpar{CONTINUE HERE}

\begin{theorem} \label{asfinitealmostL2}
 Assume that $f_2 \in C^\infty(\Rplus)$ is asymptotically finite close to $\infty$, and of moderate growth in a neighborhood of $0$. Define the Mellin transform $\check f_2(z) = \check f_{2,0}(z) + \check f_{2,\infty}(z)$, where $f_2 = f_{2,0}+f_{2,\infty}$ is a decomposition with $f_{2,\infty}\in\mathcal S^+_{\emptyset,E_2^-}(\Rplus)$ and $f_{2,0}$ of rapid decay close to infty, $\check f_{2,\infty}$ is the (meromorphically continued) Mellin transform of $f_{2,\infty}$ as before, and $\check f_{2,0}(z)$ is the Mellin transform of $f_{2,0}$, defined by a convergent integral for $\Re(z)\ll 0$. Assume the following about $\check f_2$:
 \begin{itemize}
  \item it admits meromorphic continuation, with a finite number of poles, to $\Re(z)\le 0$;

 \item it satisfies the same estimate \eqref{vertstrip} on any vertical strip $V$ of the form $a\le \Re(z) \le 0$, away from the poles. 
 \end{itemize}

 Then:
 \begin{enumerate}
  \item $f_2$ can be written as $f_2^{\operatorname{fin}} + f_2^{\operatorname{rest}}$, with $f_2^{\operatorname{fin}}$ an asymptotically finite smooth function and $f_2^{\operatorname{rest}}$ a smooth function which is of rapid decay (together with its derivatives) close to $\infty$, and lies in $L^2(\RR^\times)$, together with its multiplicative derivatives. 
  \item If $f_1 \in \mathcal S^+_{(\emptyset, E_1^-)}(\Rplus)$, and no two exponents $s_1 \in E_1^-$, $s_2\in E_2^-$ satisfy $s_1+s_2 = 0$, then the regularized inner product 
  $$ \int^* f_1 f_2 = \int f_1 f_{2,0} + \int^* f_1 f_{2,\infty} = \int^* f_1 f_2^{\operatorname{fin}} + \int f_1 f_2^{\operatorname{rest}}$$
  admits the same decomposition as \eqref{Plancherel-torus-s}, with any $\sigma \ge 0$ (and the appropriate modifications as in \eqref{Mellininv-ext-PV} if there are poles on $Re(z)=\sigma$). This includes any terms of the form 
  $$\Res^+_{s=s'} \check f_1(s) \check f_2(-s) = - \check f_1(s') \Res^-_{s=-s'} \check f_2 (s)$$ due to exponents $-s'$ of $f_2$ at $\infty$ with $\Re(s')<\sigma$; although $\check f_2(s)$ may not be defined for $\Re(s) >0$, its negatively charged residue is defined, as the negatively charged residue of $\check f_2^{\operatorname{fin}}(s)$ (equivalently, of $\check f_{2,\infty}$).
 \end{enumerate}

 \begin{remark}
  The estimate on vertical strips can be further relaxed to $$|\check f(\sigma+it)|\le \frac{C_m}{(1+|t|^m) |\sigma|^r}$$ plus $L^2$ on $\sigma=0$, with weaker conclusions on $f_2^{\operatorname{rest}}$, if one is interested in applying such considerations to Eisenstein series and ``Fourier--Eisenstein transforms'', cf.\ \cite[\S 14, Corollary]{CasPW}. However, the present version of the proposition will be enough for our purposes, and is much easier to prove.
 \end{remark}

\end{theorem}

\begin{proof}
 By Theorem \ref{Plancherel-torus-s}, we may clearly assume that $f_{2,\infty}=0$. Then $f_1 |x|^{-\sigma}$ and $f_2 |x|^\sigma$ are both $L^2$, for $\sigma \gg 0$, and hence we can write the Plancherel formula for their product:
 $$ \int f_1 f_2 = \frac{1}{2\pi i} \int_{\Re(s) = \sigma \gg 0} \check f_1(s) \check f_2(-s) ds,$$
 or the Mellin inversion for $f_2$:
 $$ f_2(x) = \frac{1}{2\pi i} \int_{\Re(s) = \sigma \gg 0} \check f_2(-s) x^{-s} ds.$$
 Our assumptions on $\check f_2$ allow us to shift the contour to any $\sigma \ge 0$, picking up residues which (in the expression for $f_2$) are generalized eigenfunctions on $\Rplus$. If we take $\sigma =0$, the remaining integral represents an $L^2$-function, if there are no poles on $\Re(s)=0$. If there are poles on $\Re(s) = 0$, we will again have a principal value integral and $\frac{1}{2}$ of the corresponding residue; to see that this still admits the same description, subtract an element of $\mathcal S^+_{(E,\emptyset)}(\Rplus)$ for some multiset $E$, whose Mellin transform has the same poles on $\Re(s)=0$.

\end{proof}

\subsection{Preparation for $[PGL_2]$: Cuspidal and Eisenstein parts, pseudo-Eisenstein series.} \label{sspreparation}

We now return to $H=\PGL_2$.

We are assuming throughout invariant measures on the various spaces under consideration. Tamagawa measures are a good choice, usually, but the precise choices do not matter, as long as they are made compactibly. This means, mainly, two things:

\begin{enumerate}
\item The measures on $H(\adele)$, $[H]$ and $B(k)\backslash H(\adele)$ are chosen compatibly with the obvious local homeomorphisms between these spaces. The measure on $B(k)\backslash H(\adele)$ induces a measure on $[H]_B$ via the quotient map of \eqref{leftright-H}, by taking the measure of $[N]$ to be $1$.

\item We will fix the isomorphism $A\simeq \Rplus$ via the \emph{dominant} character $\delta^{\frac{1}{2}}$ (so that the cusp corresponds to $\infty$) -- this is \emph{opposite} to our description of the reductive Borel--Serre compactification in \S \ref{sscompactifications}, where we used anti-dominant characters, but corresponds to the classical coordinate $y^\frac{1}{2}$ on the upper half plane.

\item Recall that $\pi_s$ denotes the space $C^\infty(A\backslash [H]_B, \delta^{\frac{s}{2}})$, where the eigencharacter $\delta^\frac{s}{2}$ of $A$ corresponds to the normalized action \eqref{action-normalized-H}. Having identified $A$ with $\Rplus$, we have the usual multiplicative Haar measure $d^\times x = x^{-1} dx$ on $\Rplus$, and can define the Mellin transform of $f\in \mathcal S([H]_B)$:
\begin{equation}\label{Mellin-HB}
 \check f(s) (g)= \int_{A} f(ag) \delta^{-\frac{1+s}{2}}(a) da \in \pi_s.
\end{equation}
The morphism $\mathcal S([H]_B)\ni f\mapsto \check f(s)\in \pi_s$ is compatible with the Mellin transform along the group $A= \Rplus$, the way it was previously defined.

\item We define a duality
$$\left<\,\, , \,\,\right>_{\pi_s}: \pi_s \otimes \pi_{-s} \to \CC$$ 
in such a way that that, for $f\in \mathcal S([H]_B)$ we have:
$$\left< \check f(s),\check f(-s) \right>_{\pi_s} = \int_{[H]_B} f(g) \check f(s)(g) dg.$$
Equivalently, the choice of measures on $[H]_B$ and $A =\Rplus$ induces a measure on the appropriate line bundle over $A\backslash [H]_B$, and the product of $\check f(s)$ and $\check f(-s)$ is integrated against that measure.
\end{enumerate}

Now we recall some generalities about the Schwartz space of $[H]$. I refer the reader to Casselman's paper \cite{Cas-Schwartz} for proofs:

There is a decomposition
\begin{equation}
L^2([H]) = L^2([H])_{\cusp} \oplus L^2([H]_\Eis,
\end{equation}
where $L^2([H])_\cusp$ consists of functions whose constant term vanishes, and $L^2([H]_\Eis$ is, by definition, the orthogonal complement of this space. 

It is a slightly non-trivial fact that, if we set $\mathcal S([H])_\cusp = L^2([H])_{\cusp}  \cap \mathcal S([H])$ and $\mathcal S([H])_\Eis = L^2([H])_\Eis \cap \mathcal S([H])$ then we have:
\begin{equation}\label{cuspdecomp}
\mathcal S([H]) = \mathcal S([H])_\cusp \oplus \mathcal S([H])_\Eis.
\end{equation}
(See \cite[Proposition 3.13]{Cas-Schwartz}.)

Now consider the pseudo-Eisenstein series morphism 
$$ \Psi\Phi(g) = \sum_{\gamma\in B(k)\backslash H(k)} \Phi(\gamma g).$$

It is the analog of the morphism $\Sigma$ of Example \ref{HmodN}, except that our Schwartz function is not taken to be on $N\backslash H(\adele)$, but directly on $[H]_B$. It is usually easier to think of the space $N\backslash H(\adele)$, and we can reduce things to it via the following lemma.

\begin{lemma}\label{lemmaAk}
The morphism $\mathcal S(N\backslash H(\adele)) \to \mathcal S([H]_B)$ given by the sum
$$ \Phi\mapsto \sum_{\alpha\in A(k)} \Phi(ag) $$
is continuous and surjective.
\end{lemma}

\begin{proof}
Restricting to $[A]$-orbits on $[H]_B$, this easily becomes a statement about Schwartz functions on $\adele^\times$, and their sum over $k^\times$ to produce Schwartz functions on $[\Gm]$. Recalling that we denote by $A$ a subgroup of $A(\adele) = \adele^\times$ isomorphic to $\Rplus$, we have $A(\adele) = A\times A(\adele)^1$, hence $\mathcal S(A(\adele)) = \mathcal S(A) \hat\otimes \mathcal S(A(\adele)^1)$, $\mathcal S([\Gm]) = \mathcal S(A)\hat\otimes \mathcal S([\Gm]^1)$, and the statement follows easily from the surjectivity of the same map from $\mathcal S(A(\adele)^1)$ to $\mathcal S([\Gm]^1) = C^\infty([\Gm]^1)$.
\end{proof}

\begin{remark}\label{remarkclassical}
The space $\mathcal S(N\backslash H(\adele))$ is closer to the ``classical'', rather than the  adelic, formalism of Eisenstein series. Assume, for simplicity, that $k=\QQ$, so that $k_\infty=\RR$, and let us replace the group $\PGL_2$ by $\SL_2$, so that $N\backslash \SL_2 = \mathbb A^2\smallsetminus\{(0,0)\}$. Take $\Phi = \otimes_{p\le \infty} \Phi_p$ in this space, with $\Phi_p =  1_{N\backslash \SL_2(\mathbb Z_p)}$ when $p<\infty$. Then, the morphism $\Sigma$ of Example \ref{HmodN}, for $g\in \SL_2(\RR)$, is equal to the classical Eisenstein sum:
\begin{equation}\label{classsum} \Sigma\Phi(g) = \sum_{(m,n)\in \Z^2, (m,n)=1} \Phi_\infty((m,n)\cdot g).\end{equation}
\end{remark}

Composing the morphism of the previous lemma with the pseudo-Eisenstein morphism, and arguing as in Example \ref{HmodN}, we obtain:
\begin{proposition}
The pseudo-Eisenstein morphism is a continuous morphism
$$\Psi:\mathcal S([H]_B)\to \mathcal S([H]),$$
and its composition with the constant terms is given by the formula 
\begin{equation}\label{constantterm-Psi} (\Psi f)_B = f + Rf,\end{equation}
where $R$ is the Radon transform:
\begin{equation}\label{Radon} \mathcal S([H]_B)\ni f\mapsto \int_{N(\adele)} f(wng) dn \in C^\infty([H]_B),\end{equation}
where $w = \begin{pmatrix}
 & 1 \\
 -1 &
\end{pmatrix}.$
\end{proposition}

The argument in the proof of Proposition \ref{HmodNprop} says that the image of the Radon transform consists of functions which are of rapid decay in a neighborhood of the cusp. Below we will give a bound for their growth in the opposite direction (the ``funnel'').

The pseudo-Eisenstein morphism is adjoint to the constant term (under the inner product pairings on $[H]$ and $[H]_B$), and therefore its image is contained in $\mathcal S([H])_\Eis$. The following holds:

\begin{proposition}[{\cite[Lemma 4.2]{Cas-Schwartz}}]\label{Psidense}
The image of $\Psi$ is dense in $\mathcal S([H])_\Eis$.
\end{proposition}

This statement is not needed for the $L^2$-decomposition (the density in $L^2([H])_\Eis$ is all that is needed, and easier to establish), but it will be useful for our extension of the Plancherel formula to asymptotically finite functions. Before we extend it to those, let us also recall the notion of Eisenstein series, which are the same as pseudo-Eisenstein series when instead of Schwartz functions on $[H]_B$ we consider functions induced from idele class characters $\chi$, i.e., elements of the principal series 
$$I(\chi) := \Ind_{B(\adele)}^{H(\adele)} (\chi\delta^\frac{1}{2}) \subset C^\infty([H]_B).$$

Elements of $I(\chi)$ are obtained when we decompose the space $\mathcal S([H]_B)$ in terms of characters of $[A]$. We will instead, for notational and conceptual simplicity, decompose in terms of characters of $A\simeq\Rplus$, and denote (as before) the induced space by $\pi_s$:
$$ \pi_s = C^\infty(A\backslash [H]_B, \delta^\frac{s}{2}).$$

Recall that $\delta^\frac{s}{s}$ refers to the \emph{normalized} action of $A$, \eqref{action-normalized-H}. The $J$-invariants of the spaces $\pi_s$, for any compact open subgroup $J$ of the finite adeles of $H$, can be thought of as a holomorphic bundle of Fr\'echet spaces over $\CC$. (For example, choose a suitable compact subgroup $K$ of $H(\adele)$ to identify all of them through their restriction to the compact set $[A]^1\cdot K\subset [H]_B$ -- but the structure of Fr\'echet bundle does not depend on this choice.) Thus, it makes sense to talk about holomorphic sections, (strongly) meromorphic functionals and morphisms, etc. When not fixing the subgroup $J$, these notions will refer to the corresponding inductive limit -- on these spaces -- or projective limit -- on their duals. For example, a strongly meromorphic family of functionals on the $\pi_s$ will be a family of functionals which is strongly meromorphic when restricted to the $J$-invariants, for each $J$.
 
By abuse of language, we will refer as ``Eisenstein series'' to the strongly meromorphic family of morphisms (s.\ Theorem \ref{meromcont} below):
$$ \mathcal E: \pi_s \to C^\infty([H]),$$
which for $\Re(s)>1$ are given by the convergent series
\begin{equation}\label{Eisseries} f\mapsto \sum_{\gamma\in B\backslash H(k)} f(\gamma g).\end{equation}

The Eisenstein series can be thought of as the spectral components of pseudo-Eisenstein series, and, similarly, the spectral components of Radon transforms are the \emph{standard intertwining operators} 
$$M(s): \pi_s \to \pi_{-s},$$
given, when $\Re(s)>1$, by the convergent integral:
\begin{equation}\label{Ms}
M(s) F_s(g) = \int_{N(\adele)} F_s(wng) dn.
\end{equation}

The analog of \eqref{constantterm-Psi} is the formula for the constant term of Eisenstein series: 
\begin{equation}\label{constantterm-Eisenstein}
 \mathcal E(f_s)_B = f_s + M(s) f_{-s}.\end{equation}

The basic properties of meromorphic continuation of the Eisenstein series and their constant terms are summarized in the following; some of them (such as the simplicity of poles) are established in the process of the $L^2$-decomposition, that we will present and generalize in the next subsection. However, it is not among the goals of the present paper to revisit the proofs of these properties, and therefore we will take them for granted:

\begin{theorem}\label{meromcont}
The Eisenstein series and the standard intertwining operator, given by the convergent expressions \eqref{Eisseries}, resp.\ \eqref{Ms}, for $\Re(s)>1$, admit strongly meromorphic continuation to $\CC$, with no poles on $\Re(s)=0$ and only a finite number of simple poles for $\Re(s)>0$. Moreover, the intertwining operators satisfy 
\begin{equation}\label{FE}M(s)M(-s)=\Id.\end{equation}
\end{theorem}

We will refer to the poles of the Eisenstein morphism for $\Re(s)>0$ as ``the Eisenstein poles''.  The last identity of the theorem is only relevant for us on the line $\Re(s)=0$, since we do not use at any point the meromorphic continuation to $\Re(s)<0$. The only exception, where the meromorphic continuation to $\Re(s)<0$ will be used, will be to obtain an explicit formula in Theorem \ref{explicit-rankone}, which is not used in the rest of the paper.

\begin{proof}
This can be reconstructed for the analogous statement for Eisenstein series applied to principal series $I(\chi)$, s.\ \cite[Chapter IV]{MW}. Notice that for any fixed exponent $s$ of $A$, and any compact open subgroup $J$ of the finite ideles $A(\adele_f)$, there is only a finite number of $J$-invariant characters $\chi$ of $[A]$ which coincide with the given one on the image of the real torus $A$. Thus, the $J$-invariants of $\pi_s$ decompose into a finite sum of principal series induced from characters of $[A]$.
\end{proof}

\begin{corollary}\label{corRadonH}
The constant term morphism:
$$ \mathcal S([H])\to C^\infty([H]_B)$$
and the 
Radon transform \eqref{Radon}
$$ R: \mathcal S([H]_B) \to C^\infty([H]_B)$$
have image in the space of smooth functions $f$ which close to the cusp are of rapid decay, and close to the funnel are bounded, together with their derivatives $Xf$ with $X\in U(\mathfrak h)$, by elements of $\pi_{-1-\epsilon}$, for any $\epsilon>0$.

Moreover, if $F$ denotes either the constant term $\varphi_B$ of some $\varphi\in\mathcal S([H])$ or the Radon transform $Rf$ of some $f \in \mathcal S([H]_B)$, the Mellin transform of $F$, defined as in \eqref{Mellin-HB}:
\begin{equation}\label{Mellin-constant}  \check F(s) (g)= \int_{A} F(ag) \delta^{-\frac{1+s}{2}}(a) da \in \pi_s\end{equation}
converges absolutely for $\Re(s)<-1$, and has meromorphic continuation to $\Re(s)\le 0$, with no poles on $\Re(s)=0$ and at most a finite number of simple poles in the region $\Re(s)<0$.
\end{corollary}

\begin{remark}
If we fix an Iwasawa decomposition $H(\adele) = B(\adele) K$, the bound on the funnel means that the $Rf$ satisfies $|Rf(bk)|\ll \delta^{-1-\frac{\epsilon}{2}}(b)$.  
\end{remark}

\begin{proof}
The rapid decay the constant term at the cusp follows from the approximation theorem \ref{approxconstantterm-H}.
The rapid decay of $Rf$ at the cusp is clear from the definition: when $g$ in \eqref{Radon} approaches the cusp, $wng$ approaches the funnel, where $f$ is of rapid decay. (The details are left to the reader.) 

The bound by $\pi_{-1-\epsilon}$ is equivalent to the convergence of Eisenstein series, resp.\ the intertwining operator $M(s)$, for $\Re(s)>1$. We will revisit the proof of this bound in the proof of Proposition \ref{Radon-s}.

For the Mellin transforms, immediately from the definitions we have the formulas
\begin{equation}
 \widecheck{\varphi_B}(s) = \mathcal E_{-s}^* \varphi,
\end{equation}
where $\mathcal E_{-s}^*$ denotes the adjoint of the Eisenstein morphism $\pi_s\to C^\infty([H])$, and
\begin{equation}
 \widecheck{Rf}(s) = M(-s) \check f(-s),
\end{equation}
so the last statement from the properties of Eisenstein series and intertwining operators, as in the preceding theorem.
\end{proof}

\subsection{Bounds on vertical strips}

We will also need a bound for the growth of the intertwining operators in vertical strips, in order to be able to shift contours of integration. This, of course, is an issue pertinent to number fields only, as for function fields ``vertical strips'' are compact. The present subsection could be part of the previous one, but I am isolating it because it is the part of the argument for which in higher rank the required results are not available yet -- and because of its technical nature, due to which the reader may wish to skip it at first reading, especially if they are primarily interested in function fields.

The technicalities involved in shifting contours are responsible for a lot of the complications that arise in the spectral decomposition and the trace formula. There are two issues here: the growth of intertwining operators, and the growth of Eisenstein series in vertical strips. The former is part of the latter (since Eisenstein series are approximated by their constant terms in a neighborhood of the cusp), but the latter also involves estimates on Eisenstein series away from the cusps, e.g., on truncated Eisenstein series. In rank one, both can be obtained by existing methods (at least for $K_\infty$-finite functions, in the case of Eisenstein series).
However, in view of the situation in higher rank, it is preferable to avoid using bounds for the Eisenstein series themselves, which are much harder to obtain than bounds for the intertwining operators. (The standard argument for bounds on Eisenstein series involves an application of the Maa\ss--Selberg relations which requires estimates on \emph{derivatives} of intertwining operators, see \cite{Mueller,FLM}. Those derivatives are very sensitive to possible zeroes of $L$-functions near the edge of their criitcal strip, making it hard to control them with purely harmonic-analytic arguments.) Therefore, I only formulate a statement on bounds for intertwining operators:

\begin{theorem}
Choose an Iwasawa decomposition $H(\adele)=B(\adele) K$, to identify all topological spaces $\pi_s$ with $C^\infty(K\cap B(\adele)\backslash K)$. The intertwining operator $M(s): \pi_s \to \pi_{-s}$ is of polynomial growth (away from its poles), on any vertical strip of the form $0\le \Re(s)\le a$; that is, for every open compact subgroup $J$ of $H(\adele_f)$ and any seminorm $\rho$ on $C^\infty(K\cap B(\adele)\backslash K)^J$ there is a finite number of seminorms $\rho_i$ on the same space, and polynomials $P_i$ on $\CC$, such that
\begin{equation}
 \rho(M(s) f) \le \sum_i |P_i(s)| \rho_i(f),
\end{equation}
for all $f\in C^\infty(K\cap B(\adele)\backslash K)^J$ (thought of as an element of $\pi_s$).
\end{theorem}

As we shall see in the proof, one can get much more precise estimates about specific seminorms.

\begin{proof}
 A complete system of seminorms on $C^\infty(K\cap B(\adele)\backslash K)^J$ is given by the seminorms $\rho(f) = \Vert Xf \Vert_{L^2(K\cap B(\adele)\backslash K)}$, for $X \in U(\mathfrak h)$. Harish-Chandra has proven \cite[Chapter IV, \S 8]{HC} that, for $s=\sigma + it$, $\sigma \ge 0$,
 \begin{equation}\label{HCestimate}\Vert M(s) f\Vert \le e^{2T\sigma} (1+ 2 \left| \frac{\sigma}{t}\right|) \Vert f \Vert\end{equation}
in the $L^2$-norm, for some positive number $T$. Replacing $f$ by any $Xf$, this is easily seen to imply the claim of the theorem.
  
Inequality \eqref{HCestimate} is stated for $K_\infty$-invariant functions at the end of the proof of Lemma 101 in \cite{HC}, but the proof applies without the assumption of $K_\infty$-finiteness. For the convenience of the reader, I briefly outline the main steps: 

Let us fix a Siegel neighborhood of the cusp of the form $U = U_T = \{g = bk \in [H]_B| \delta^{\frac{1}{2}}(g) > e^T\}$, for large enough $T$; here, $g=bk$ refers to the Iwasawa decomposition fixed in the theorem. Let $v= v_T$ be its indicator function (on $[H]_B$), and denote by $\tilde v$ the indicator function of the corresponding neighborhood in $[H]$.

Consider the truncated Eisenstein series $\wedge^T \mathcal E(f_s) = \mathcal E(f_s) - \Psi(v\cdot f_s ) - \Psi(v\cdot M(s) f_s)$. In the convergent region $\Re(s)>1$ for the Eisenstein sum, this can be written as a convergent series:
$$ \wedge^T\mathcal E(f_s) = \Psi(F_s),$$
where $F_s = f_s \cdot (1-v) - M(s) f_s\cdot v$. From this, one deduces the Maa\ss--Selberg relations about the inner product of two truncated Eisenstein series; in our bilinear notation:

\begin{eqnarray}\label{MaassSelberg}
 \left < \wedge^T\mathcal E(f_{s_1}) , \wedge^T\mathcal E(f_{s_2}) \right >_{[H]} = \frac{e^{T(s_1+s_2)}}{s_1+s_2} \left<f_{s_1}, f_{s_2}\right>  + \frac{e^{T(-s_1+s_2)}}{-s_1+s_2} \left<M(s_1) f_{s_1}, f_{s_2}\right> + \nonumber \\ + \frac{e^{T(s_1-s_2)}}{s_1-s_2} \left<f_{s_1}, M(s_2) f_{s_2}\right> + \frac{e^{T(-s_1-s_2)}}{-s_1-s_2} \left<M(s_1) f_{s_1}, M(s_2) f_{s_2}\right>
\end{eqnarray}

Taking $s_1=s =\sigma + it \ne 0, s_2 = \bar s$, $f_{\bar s} = \overline{f_s}$, the left hand side is $\ge 0$, which gives:
$$ \frac{e^{2T \sigma}}{2\sigma} \left<f_s, \overline{f_s}\right>  + 2 \Re\left( \frac{e^{2iTt}}{2it} \left<f_s, \overline{M(s) f_s}\right> \right)- \frac{e^{-2T\sigma}}{2\sigma} \left<M(s) f_s, \overline{M(s) f_s}\right> \ge 0.$$

Hence, if $\Vert f_s \Vert = 1$ in the norm of $L^2(K\cap B(\adele)\backslash K)$, we get:
$$ \frac{e^{2T \sigma} - e^{-2T\sigma} \Vert M(s) f_s\Vert^2}{2\sigma} +  \frac{\Vert M(s) f_s \Vert}{|t|}  \ge 0,$$
which implies that $\Vert M(s) f_s \Vert \le e^{2T\sigma} \left( \frac{\sigma}{|t|} + \sqrt{1+ \frac{\sigma^2}{t^2}}\right)$ for $\sigma\ge 0$, whence \eqref{HCestimate}.

\end{proof}

Combining this with Corollary \ref{corRadonH}, we get: 

\begin{corollary}\label{corrapiddecay}
 Let $f\in \mathcal S([H]_B)$, then $\widecheck{Rf}(s) = M(-s)\check f(-s)$ is of rapid decay on any vertical strip of the form $a\le \Re(s)\le 0$, away from the negatives of the Eisenstein poles.
\end{corollary}

\begin{proof}
 This is a combination of the polynomial growth of $M(-s)$ with the rapid decay of $\check f(-s)$.
\end{proof}

\subsection{Decomposition of the space of asymptotically finite functions}

We now move to asymptotically finite functions, i.e.\ a space of the form  $\mathcal S^+_{s} ([H])$. Notice that it makes sense to say that an element of $\mathcal S^+_{s} ([H])$ is ``orthogonal'' to a given subspace of $\mathcal S([H])$, since the former are of polynomial growth and the latter are of rapid decay (so their ``inner product'' converges). We let 
$$ \mathcal S^+_{s} ([H])_\Eis = \mathcal S^+_{s} ([H]) \cap \mathcal S([H])_\cusp^\perp.$$

Consider the space $\mathcal S^+_{s} ([H]_B)$ of asymptotically finite functions on $[H]_B$, where the exponent $s$ is understood to be appearing in the neighborhood of the cusp, as in Definition \ref{asymptfinite-H}. 
(In particular, elements of $\mathcal S^+_{s} ([H]_B)$ are by definition of rapid decay in the direction of the ``funnel''.) Extending the results of \S \ref{sspreparation}, we have:

\begin{proposition}\label{Psiprop-s}
We have a decomposition
\begin{equation}\label{cuspdecomp-s}\mathcal S^+_{s} ([H]) = \mathcal S([H])_\cusp \oplus \mathcal S^+_{s} ([H])_\Eis.\end{equation}
The pseudo-Eisenstein sum $\Psi f(g) = \sum_{\gamma\in B\backslash H(k)} f(\gamma g)$ is convergent on $\mathcal S^+_{s} ([H]_B)$, for any $s\in \CC$, and maps continuously, with dense image, into $\mathcal S^+_{s} ([H])_\Eis$. Its composition with the constant term is given by the same formula \eqref{constantterm-Psi}.
\end{proposition}

\begin{proof}

Convergence of the Eisenstein sum is easier to see by an analog of Lemma \ref{lemmaAk}, using an auxilliary space 
$$\mathcal S'_{s}(N\backslash H(\adele)):=\hat\bigotimes_{v\ne v_0}'\mathcal S(N\backslash H(k_v))\hat\otimes \mathcal S^+_s(N\backslash H(k_{v_0})),$$
where $v_0$ is  a fixed Archimedean place, and $\mathcal S^+_s(N\backslash H(k_{v_0}))$ denotes smooth functions on $N\backslash H(k_{v_0})$ which are asymptotically $\delta^{\frac{s}{2}}$-eigenfunctions close to the cusp (and of rapid decay in the funnel).

It can be proven exactly as in Lemma \ref{lemmaAk} that the morphism 
$$\mathcal S'_{s}(N\backslash H(\adele)) \to \mathcal S^+_s([H]_B)$$
given by summation over $A(k)$-orbits is (continuous and) surjective.\footnote{As in Remark \ref{remarkclassical}, the auxilliary space $\mathcal S'_{s}(N\backslash H(\adele))$ is of ``classical'', rather than adelic, nature; it corresponds to the sum \eqref{classsum}, but with $\Phi_\infty$ a function which is asymptotically finite close to the origin (and, again, of rapid decay away from it).}  But the elements of $\mathcal S'_{s}(N\backslash H(\adele))$ are of rapid decay outside of a compact set in the adelic points of the affine closure $\spec k[H]^N$ of $N\backslash H$, which proves convergence of the sum. 

This argument also proves absolute convergence of the Radon transform \eqref{Radon} (because $N$-orbits on the affine closure of $N\backslash H$ are closed, hence the restrictions of elements of $\mathcal S'_{s}(N\backslash H(\adele))$), and the formula for the constant term \eqref{constantterm-Psi} of the pseudo-Eisenstein sum. As before, the Radon transform of an element of $\mathcal S'_{s}(N\backslash H(\adele))$ is of rapid decay close to the cusp, and therefore, by \eqref{constantterm-Psi}, the image of $\Psi$ lies in $\mathcal S^+_{s} ([H])$. By the definitions, it is contained in $\mathcal S^+_s([H])_\Eis$.

This argument (more precisely, formula \eqref{constantterm-Psi}) actually proves a stronger statement, which will be crucial at various points, so we formulate it as a lemma:
\begin{lemma}\label{germsinfty}
Consider the short exact sequences:
$$  0 \to \mathcal S([H]_B) \to \mathcal S^+_s([H]_B) \to \pi_s \to 0,$$
$$ 0 \to \mathcal S([H]) \to \mathcal S^+_s([H]) \to \pi_s \to 0. $$
Since $\Psi$ preserves Schwartz spaces, it induces an endomorphism:
$$ \pi_s \to \pi_s$$
by passing to the quotients of the above sequences. This map is the identity.
\end{lemma}

This proves the remaining assertions, since now for every element  $\varphi\in \mathcal S^+_s([H])$ we can subtract $\Psi f$ for some $f\in \mathcal S^+_s([H]_B)$ with the same image in $\pi_s$, and reduce ourselves to the case of $\mathcal S([H])$ (cf.\ \eqref{cuspdecomp} and Proposition \ref{Psidense}).
\end{proof}

We now study the image of Radon transform.

\begin{proposition}\label{Radon-s}
Radon transform \eqref{Radon}, applied to elements of $\mathcal S^+_{s_0}([H]_B)$, 
has image in the space of smooth functions which close to the cusp are of rapid decay, and close to the funnel are bounded (together with their derivatives under $U(\mathfrak h)$) by elements of $\pi_{-\sigma-\epsilon}$, for any $\epsilon>0$, where $\sigma = \max\{\Re(s_0), 1\}$.

Therefore, if $f\in \mathcal S^+_{s_0}([H]_B)$, the Mellin transform of its Radon transform, defined as in \eqref{Mellin-constant}:
\begin{equation}\label{Mellin-Rf}\widecheck{(Rf)}(s) (g)= \int_{A} Rf(ag) \delta^{-\frac{1+s}{2}}(a) da \in \pi_s\end{equation}
converges for $\Re(s)<-\sigma$. This Mellin transform admits meromorphic continuation to the region $\Re(s)\le 0$, with poles only the opposites of the (simple) poles of Eisenstein series for $\Re(s)>0$, and the pole $-s_0$, and uniform rapid decay away from the poles in vertical strips of the form $a\le \Re(s) \le 0$. (If the pole $-s_0$ coincides with the negative of a pole of Eisenstein series, it may appear with multiplicity $2$.)
\end{proposition}

This is an important proposition, and although very classical in one form or another, I present its proof carefully (modulo some uniformity of the estimates over compact sets, which is easy and left to the reader). The proof includes the proof of convergence of intertwining operators for $\Re(s)>1$, which we have stated and used in Corollary \ref{corRadonH}. For notational convenience, I will present the proof for $\SL_2$ (also denoted by $H$), instead of $\PGL_2$, because the affine closure of $N\backslash \SL_2$ is $\mathbb A^2$, so functions on $N\backslash \SL_2$ can be represented as functions of two variables $(x,y)$. 

\begin{proof}
The rapid decay at the cusp of the image of Radon transform is proven as for rapidly decaying functions. 

To prove the stated growth in the funnel, it suffices to show that the Mellin transform $\widecheck{Rf}(s)$ converges for $\Re(s)<-\sigma$. 

We may again use the auxilliary space $\mathcal S'_{s_0}(N\backslash H(\adele))$, as in the previous proposition, and assume that $f$ is the image of an element $\Phi$ in this space. Moreover, at this point it is more convenient to perform Mellin transform \eqref{Mellin-Rf} over all of $[A]$, not just the real torus $A$. (Obviously, they converge simultaneously, because $[A]=A\times [A]^1$ and $[A]^1$ is compact.) In terms of $\Phi$, that Mellin transform becomes Eulerian:
$$ \int_{[A]} \int_{N(\adele)} f(wna) dn \delta^{-\frac{1+s}{2}}(a) da = \int_{A(\adele)} \int_{N(\adele)} \Phi(wna) dn \delta^{-\frac{1+s}{2}}(a) da =$$
\begin{equation}\label{adelicint}
= \int_{\adele^\times} \int_{\adele} \Phi(-t, t^{-1}x) dx |t|^{-(1+s)} d^\times t,
\end{equation}
where we have made the substitutions $a = \begin{pmatrix} t \\ & t^{-1}\end{pmatrix}$, $w=\begin{pmatrix}
& 1\\ -1\end{pmatrix}$, and have represented $\Phi$ as a function on $\adele^2$.

I will first explain the case $s_0= -1$ (so that the elements of $\mathcal S^+_{s_0}(N\backslash H(k_{v_0}))$ are asymptotically constant close to the cusp); this, as we will see, is essentially the analysis that one would do for the Radon transform on $\mathcal S([H]_B)$. The modifications for the general case will be presented afterwards.

In that case, the result follows \emph{a fortiori} if we show it when the auxilliary space $\mathcal S'_{s_0}(N\backslash H(\adele))$ is replaced by $\mathcal S(\adele^2)$. Then one easily checks that we have
\begin{equation}\label{adelicNint} \int_{N(\adele)} \Phi(wn a) dn= \int_{\adele} \Phi(-t, t^{-1}x)  dx  \ll |t|= \delta^\frac{1}{2}(a)\end{equation}
as $|t|=\delta^\frac{1}{2}(a)\le 1$. Moreover, for almost every place $v$, the corresponding local integral is zero if $|a_v|>1$.

Therefore, the integral \eqref{adelicint} is bounded by a Tate integral of a Schwartz function on $\adele$ against the measure $|t|^{-s} d^\times t$, and hence converges when $-s >1$. This proves the desired estimate when $s_0 = -1$. Notice that the divergence of the integral when $(\sigma+\epsilon)=1$ is of global nature, i.e.\ is due to the divergence of the Euler product, not of its local factors.

In the general case, whenever $\Re(s_0)\le -1$ we obviously have the same estimate \eqref{adelicNint}, while for $\Re(s_0)>-1$ the estimate \eqref{adelicNint} remains the same for all places $v\ne v_0$. For $v=v_0$ it becomes (exercise!) a bound of the form $\ll \max\{|t|, |t|^{1-\Re(s_0)}\}$. Thus, the integral \eqref{adelicint} converges as before, except for the fact that for the local factor at $v_0$ to converge we need the extra condition $-\Re(s_0+s)>0$.

A better argument, without explicit calculation, which will reduce the general case to the case of $\mathcal S(N\backslash H(\adele))$, and moreover will prove the statement about poles of the Mellin transform, is to observe that for every $a\in A(k_{v_0})$, and $\Phi \in \mathcal S'_{s_0}(N\backslash H(\adele))$, the function 
$$ ((a\cdot) - \delta^{\frac{s_0}{2}}(a)) \Phi \in \mathcal S(N\backslash H(\adele)),$$
where $(a\cdot)$ denotes the operator of the action of $a$, normalized as in \eqref{action-normalized-H}. The Radon transform 
$$R: \Phi\mapsto \int_{N(\adele)} \Phi(wn\bullet) dn$$
is anti-equivariant with respect to the normalized action of $A(\adele)$, thus 
$$ ((a\cdot) - \delta^{-\frac{s_0}{2}}(a)) R\Phi \in R\left(\mathcal S(N\backslash H(\adele))\right).$$
Taking Mellin transforms, this gives that:
$$ (\delta^{\frac{s}{2}}(a) - \delta^{-\frac{s_0}{2}}(a)) \widecheck{R\Phi}(s) $$ 
is the Mellin transform of the Radon transform of an element of $\mathcal S(N\backslash H(\adele))$, which reduces us to the previous case, with an extra pole at $s=s_0$. (Of course, the ``first'' pole of the Mellin transform as we let $\Re(s)$ grow also determines its domain of convergence.)

Finally, the statements about meromorphic continuation follow from the identity $\widecheck{Rf} (s)= M(-s) \check f(-s)$, and the analogous properties of intertwining operators (Theorem \ref{meromcont} and Corollary \ref{corrapiddecay}).
\end{proof}

\begin{remark}\label{remarkRadon}
The meromorphic continuation of $\widecheck{Rf}(s)$  to $\Re(s)\le 0$ with only a finite number of poles, together with the decay on vertical strips, imply, by Theorem  \ref{asfinitealmostL2}, that $Rf$ is the sum of an asymptotically finite function with an $L^2$ function (both supported away from the cusp, up to a function of rapid decay). 
\end{remark}

This implies the following on constant terms:
\begin{corollary}\label{constantterm-s}
The constant term morphism: 
$$ \mathcal S^+_{s_0}([H])\ni \varphi \mapsto \varphi_B\in C^\infty([H]_B)$$
has image in the space of smooth functions which close to the cusp are asymptotically finite with exponent $s_0$, and close to the funnel are bounded (together with their derivatives under $U(\mathfrak h)$) by elements of $\pi_{-\sigma-\epsilon}$, for any $\epsilon>0$, where $\sigma = \max\{\Re(s_0), 1\}$.

Moreover, if $\varphi_B$ is written as $\varphi_{B,0} + \varphi_{B,\infty}$, where $\varphi_{B,0}\in \mathcal S^+_{s_0}([H]_B)$ and $\varphi_{B,\infty}$ is of rapid decay in a neighborhood of the cusp, and we define the Mellin transform:
$$ \widecheck{\varphi_B} (s) := \widecheck{\varphi_{B,0}}(s) + \widecheck{\varphi_{B,\infty}}(s) \in \pi_s,$$
where the second summand is given by the convergent integral \eqref{Mellin-constant} for $\Re(s)<-\sigma$ and the first is understood as the regularization of the corresponding integral, then $\widecheck{\varphi_B} (s)$ admits meromorphic continuation to the region $\Re(s)\le 0$, with poles only the opposites of the (simple) poles of Eisenstein series for $\Re(s)>0$, and the pole $-s_0$. (If the latter coincides with the opposite of a pole of Eisenstein series, it may appear with multiplicity $2$.)
\end{corollary}

\begin{proof}
For $\varphi$ of the form $\Psi f$, with $f\in \mathcal S^+_{s_0}([H]_B)$, this follows from the formula \eqref{constantterm-Psi} for the constant term of pseudo-Eisenstein series, together with Proposition \ref{Radon-s}. By Lemma \ref{germsinfty}, we can subtract from every $\varphi$ an element of the form $\Psi f$ and reduce ourselves to the case of $\mathcal S([H])$, which was addressed in Corollary \ref{corRadonH}.
\end{proof}

\begin{remark}
 It is clear that the image space described in the corollary (as well as in Corollary \ref{corRadonH} and Proposition \ref{Radon-s}) have natural topologies as countable strict limits of Fr\'echet spaces, and it is clear from the proofs that the morphisms appearing are continuous.
\end{remark}

\begin{remark}
 One can prove enough about the decay of Eisenstein series in vertical strips in order to deduce the analog of Remark \ref{remarkRadon} for constant terms of general elements of $\mathcal S^+_{s_0}([H])$; however, we do not need this and will not state it; s.\ \cite{CasPW}.
\end{remark}

\subsection{Plancherel decomposition for asymptotically finite functions on $[PGL_2]$.} \label{ssPlrankone}

We will need to reproduce the argument of Selberg (and Langlands, in  general rank) for the spectral decomposition of $L^2([H])$, in order to extend it to asymptotically finite functions. It will eventually boil down to regularized inner products on $[H]_B$, which can be analyzed using the model of $\Rplus$.  

Let $\varphi_i\in \mathcal S^+_{s_i}([H]))$, where $i=1,2$ and $s_i\in \CC$ with $s_1+s_2\ne 0$. Hence, the regularized ``inner product''
\begin{equation}\label{ip-H}\left< \varphi_1 , \varphi_2\right>_{[H]}^*= \int_{[H]}^* \varphi_1(g) \varphi_2(g) dg\end{equation}
is defined.

Here is the strategy for developing the ``Plancherel decomposition'' of the ``inner product'' \eqref{ip-H}: First, we assume that $\varphi_i = \Psi f_i$, for some $f_i\in \mathcal S^+_{s_i}([H]_B)$ ($i=1,2$). We develop a spectral expansion for that case, which \emph{should depend only on the $\varphi_i$'s, not on the choice of $f_i$}. (The map $\Psi$ has non-trivial kernel.) The expansion will be an integral of bilinear pairings which are continuous on $\mathcal S_{s_1}^+([H]) \otimes \mathcal S_{s_2}^+([H])$; we will prove that this expression remains convergent for arbitrary elements of $\mathcal S_{s_1}^+([H]) \otimes \mathcal S_{s_2}^+([H])$ and hence, by the density of pseudo-Eisenstein series (Proposition \ref{Psiprop-s}), decomposes the regularized inner product on the whole space.

We start with the following:

\begin{proposition}
If $\varphi_1 = \Psi f_1$ for some $f_1\in \mathcal S^+_{s_1}([H]_B)$, we have:
\begin{eqnarray}\label{decomprankone}
\left< \varphi_1 , \varphi_2\right>_{[H]}^* = \frac{1}{2\pi i}\int_{\sigma-i\infty}^{\sigma+i\infty} \left< \check f_1(s), \widecheck{\varphi_{2,B}}(-s)\right>_{\pi_s} ds \nonumber \\
- \sum_{\Re(s')<\sigma} \Res_{s=s'}^+ \left< \check f_1(s), \widecheck{\varphi_{2,B}}(-s)\right>_{\pi_s} + \sum_{\Re(s')>\sigma} \Res_{s=s'}^- \left< \check f_1(s), \widecheck{\varphi_{2,B}}(-s)\right>_{\pi_s}, \label{Plancherel-offaxis}
\end{eqnarray}
for any $\sigma\gg 0$ not coinciding with the real part of a pole of $ \check f_1(s)$ or $\widecheck{\varphi_{2,B}}(-s)$. (If it coincides with the real part of the pole, use the same modification as in \eqref{Mellininv-ext-PV}.)
\end{proposition}

Here $\widecheck{(\varphi_{2,B})}(s)$ denotes the Mellin transform of the constant term $\varphi_{2,B}$, as described in Corollary \ref{constantterm-s}.
The charges of the Laurent series are taken as $-$ for exponents at the cusp, and as $+$ for exponents at the funnel. More precisely:

\begin{itemize}
\item The Mellin transform $\check f_1(s)$ only has a \emph{negative, simple} pole at $s=s_1$. 

\item The Mellin transform $\widecheck{\varphi_{2,B}}(s)$ has only one negative, simple pole at $s=s_2$. The rest of its poles are \emph{positive}, and for $\Re(s)\le 0$ they consist of the negatives of the Eisenstein poles (which do not lie on the line $\Re(s)=0$), and the pole $-s_2$ (if it has real part $\le 0$). Notice that if $\Re(s_2) < 0$, then it may coincide with the negative of an Eisenstein pole, in which case we will have a non-trivial decomposition of the Laurent expansion at $s_2$ into positive and negative ``charges''. On the other hand, if $\Re(s_2)>0$, then $-s_2$ may coincide with the negative of an Eisenstein pole, in which case $\widecheck{\varphi_{2,B}}(s)$ may have a double pole at that point. All other poles are simple. Poles with $\Re(s)>0$, other than the possible ``negative'' pole at $s_2$, will not be of interest to us.

By our conventions, the function $s\mapsto \widecheck{\varphi_{2,B}}(-s)$ will have a \emph{positive} pole at $-s_2$, and all other poles \emph{negative}.

\item The poles of $\widecheck{\varphi_{2,B}}(s)$ for $\Re(s)>0$ (if we assume its meromorphic continuation, which we haven't stated and are not using) are irrelevant. One could take their charges to be negative, but it will play no role. Notice, however, that, as in Theorem \ref{asfinitealmostL2}, $\Res_{s=s_2}^- \widecheck{\varphi_{2,B}}(s_2)$ makes sense, even if $\Re(s_2)>0$. Indeed, by definition, it only depends on the behavior of $\varphi_{2,B}$ (equivalently: of $\varphi$) close to the cusp, more precisely on the image in $\pi_{s_2}$ under the equivalent maps of Lemma \ref{germsinfty}.
\end{itemize}

Thus, the sum $\sum_{\Re(s')<\sigma}  \Res_{s=s'}^+ \left< \check f_1(s), \widecheck{\varphi_{2,B}}(-s)\right>_{\pi_s}$ in the above formula can only contain the term $$\left<\check f_1(-s_2), \Res_{s=-s_2}^+ \widecheck{\varphi_{2,B}}(-s)\right> = - \left<\check f_1(-s_2), \Res_{s=s_2}^- \widecheck{\varphi_{2,B}}(s)\right> ,$$
and only if $\Re(-s_2)<\sigma$. On the other hand, the sum of ``negative'' residues can contain many terms, corresponding to poles of Eisenstein series (with $\Re(s)>\sigma\ge 0$) and the pole $s_1$ of $\check f_1$ (which may coincide with a pole of an Eisenstein series).

\begin{proof}
By adjunction between the pseudo-Eisenstein morphism and the constant term (which continues to hold for the regularized inner products that we are considering), we have:
$$\left< \Psi f_1 , \varphi_2\right>_{[H]}^* = \left< f_1, \varphi_{2,B} \right>_{[H]_B}^*.$$

The constant term $\varphi_{2,B}$ has been described in Corollary \ref{constantterm-s}. Although it is not asymptotically finite in both directions, it is easy to see that the Plancherel decomposition of Theorem \ref{Plancherel-torus-s} continues to hold as long as we take $\Re(\sigma)\ge 0$; this amounts to the statement of the proposition.

\end{proof}

The formula \eqref{Plancherel-offaxis} is not yet our ``Plancherel'' formula, not even in the case of the actual $L^2([H])$ decomposition (e.g., when the $\varphi_i$'s belong to $\mathcal S([H])$), because its integrand depends on $f_1$ and cannot be understood in terms of $\varphi_1$ alone. To see what must be done, recall the expression \eqref{constantterm-Psi} for $\varphi_{1,B} = (\Psi f_1)_B$:
$$ (\Psi f_1)_B = f_1 + R f_1.$$

Decomposing its Mellin transform correspondingly, we get (at least for $\Re(s)\le 0$, cf.\ Corollary \ref{constantterm-s}):
\begin{equation}\label{decompct}
\widecheck{(\Psi f_1)_B}(s)  = \check f_1(s) + M(-s) \check f_1(-s),
\end{equation}
where $M(s)$ denotes the standard intertwining operator \eqref{Ms}. 
Thus, to have an expression that only depends on $\varphi_1$, the integrand should (generically) depend on $\check f_1(s) + M(-s) \check f_1(-s)$, not just $\check f(s)$. This is why we have to move the contour of integration in \eqref{Plancherel-offaxis} to the line $\sigma=0$. To avoid using estimates for Eisenstein series on vertical strips, now we assume that $\varphi_2 = \Psi(f_2)$, as well (with $f_2\in \mathcal S^+_{s_2}([H]_B)$. Thus, we have $\varphi_{2,B} = f_2  + Rf_2$, and $\widecheck{\varphi_{2,B}}(-s) = \check f_2(-s) + M(s) \check f_2(s)$.  
What follows is written under the assumption that $\Re(s_1)\ne 0 \ne \Re(s_2)$; if that is not the case, one needs to adjust with principal value integrals as in \eqref{Mellininv-ext-PV}.

The following is a corollary of Theorem \ref{asfinitealmostL2} and Corollary \ref{corrapiddecay}.
\begin{corollary}
 For $\varphi_2 = \Psi(f_2)$, the formula \eqref{decomprankone} continues to hold for all $\sigma\ge 0$ (with the modifications as in \eqref{Mellininv-ext-PV} when poles lie on the line of integration).
\end{corollary}

Now we use the functional equation \eqref{FE} of intertwining operators, as well as the (easy) adjunction formula:
\begin{equation}\label{Msadjunction} 
\left< M(-s) F_{1,-s}, F_{2,-s}\right>_{\pi_s} = \left<F_{1,-s}, M(-s) F_{2,-s}\right>_{\pi_{-s}}.
\end{equation}

This gives:
\begin{proposition}\label{constantterm-symmetry}
For $\varphi\in \mathcal S^+_{s_0}([H])$, the constant term $\widecheck{\varphi_B}(s)$ satisfies:
\begin{equation}\label{constantterm-symmetry-eq}
\widecheck{\varphi_B}(s) = M(-s) \widecheck{\varphi_B}(-s).
\end{equation}
\end{proposition}

\begin{remark}
We will only use this relation on the line $\Re(s)=0$, so the continuation of $\widecheck{\varphi_B}(s)$ to $\Re(s)>0$ is not important.
\end{remark}

\begin{proof}
This follows from \eqref{decompct} and \eqref{FE} when $\varphi = \Psi f$. By continuity of the map $\varphi\mapsto \widecheck{\varphi_B}(s)$ (equivalently, of the Eisenstein series), and of the intertwining operators $M(s)$, it extends to the whole space $\mathcal S([H])$.
\end{proof}

We will actually only need this statement for pseudo-Eisenstein series. Applying it to $\varphi = \varphi_2 = \Psi f_2$,  
writing 
$\widecheck{\varphi_{2,B}}(s)$ as $$\frac{1}{2}\left(\widecheck{\varphi_{2,B}}(s) +  M(-s) \widecheck{\varphi_{2,B}}(-s))\right),$$ 
combining the terms $\pm s$ in the integral, and using the adjunction \eqref{Msadjunction}, we get:
\begin{eqnarray}
\left< \varphi_1 , \varphi_2\right>_{[H]}^* =  \frac{1}{2\pi i}\int_{0}^{+i\infty} \left< \check f_1(s)+ M(-s) \check f_1(-s), \widecheck{\varphi_{2,B}}(-s)\right>_{\pi_s} ds \nonumber \\
+ \sum_{\Re(s_2)>0}  \left(\left<\check f_1(-s_2), \Res_{s=s_2}^- \widecheck{\varphi_{2,B}}(s)\right>_{\pi_{-s_2}} + \Res^-_{s=s_2} \left< \check f_1(s), \widecheck{\varphi_{2,B}}(-s)\right>_{\pi_s}\right) \nonumber \\ 
+ \sum_{\Re(s')>0, s'\ne s_2} \Res_{s=s'}^- \left< \check f_1(s), \widecheck{\varphi_{2,B}}(-s)\right>_{\pi_s}. \label{Plancherel-onaxis-prelim}
\end{eqnarray}

The notation $\sum_{\Re(s_2)>0}$ means, of course, that this term appears only if $\Re(s_2)>0$. (Again, if $\Re(s_2)=0$ one needs to modify as in \eqref{Mellininv-ext-PV}.) Notice that the pole in the second term on the second line comes from both $\check f_1(s)$ and $\widecheck{\varphi_{2,B}}(-s)$ if $s_1=s_2$; otherwise, the term can be written $\left< \check f_1(s_2), \Res^-_{s=s_2} \widecheck{\varphi_{2,B}}(-s)\right>_{\pi_{s_2}}$.
Now we claim:

\begin{theorem}\label{Plancherelrankone}
All integrands in \eqref{Plancherel-onaxis-prelim}, including the discrete summands,\footnote{This refers to the summands as they appear in the equation, e.g., the term $\left(\left<\check f_1(-s_2), \Res_{s=s_2}^- \widecheck{\varphi_{2,B}}(s)\right>_{\pi_{-s_2}} + \Res^-_{s=s_2} \left< \check f_1(s), \widecheck{\varphi_{2,B}}(-s)\right>_{\pi_s}\right)$ as a whole.} depend on $\varphi_1$ and $\varphi_2$ alone, not $f_1$ or $f_2$, and extend to continuous functionals on $\mathcal S^+_{s_1}([H])\hat\otimes \mathcal S^+_{s_2}([H])$. 

For each discrete summand $s'$ (or $s_2$) on the right hand side of \eqref{Plancherel-onaxis-prelim}, let $J_{s'}(\varphi_1,\varphi_2)$ denote the corresponding bilinear form. Notice that $s'$ runs over the subset of $\{s_1, s_2, \mbox{(poles of Eisenstein series)}\}$ with $\Re(s)\ge 0$. The expression
\begin{eqnarray}
\frac{1}{2\pi i}\int_{0}^{+i\infty} \left<\widecheck{\varphi_{2,B}}(s), \widecheck{\varphi_{2,B}}(-s)\right>_{\pi_s} ds + \sum_{\Re(s')>0} J_{s'}(\varphi_1,\varphi_2). \label{Plancherel-onaxis}
\end{eqnarray}
converges absolutely for all $\varphi_1\otimes \varphi_2\in \mathcal S^+_{s_1}([H])\hat\otimes \mathcal S^+_{s_2}([H])$,\footnote{We use $\varphi_1, \varphi_2$ for notational simplicity, but the expression and statement of the theorem is understood to hold for arbitrary elements of the completed tensor product.} and is equal to the regularized inner product $\left< \varphi_1 , \varphi_2\right>_{[H]}^*$. (If if either of $\Re(s_i)$ happens to be zero, one needs to introduce the appropriate modifications as in \eqref{Mellininv-ext-PV}.)
\end{theorem}

\begin{proof}
The integrand in the continuous integral can be written $$\left< \widecheck{\varphi_{1,B}}(s), \widecheck{\varphi_{2,B}}(-s)\right>_{\pi_s},$$ and hence is clearly a continuous function of $\varphi_1$, $\varphi_2$, by continuity of the map $\varphi\mapsto \widecheck{\varphi_{B}}(s)$ (or, equivalently, of the Eisenstein morphism). We first prove the absolute convergence of this integral for an arbitrary element of $\mathcal S^+_{s_1}([H])\hat\otimes \mathcal S^+_{s_2}([H])$. For notational simplicity, we still assume that the element is of the form $\varphi_1 \otimes \varphi_2$, but the reader can easily reformulate for an arbitrary element of the completed tensor product.

Fix an Archimedean place $v$. (In the function field case, one can choose a place $v$ corresponding to a point that is defined over the finite base field, and work with the Bernstein center of $H(k_v)$.) The Harish-Chandra isomorphism identifies the center $\mathfrak z(\mathfrak h_v)$ of the complexified Lie algebra with the polynomial ring $\CC[\mathfrak t_\CC]^W$ (we will denote the image of $Z$ by $\check Z$), where $T$ is the restriction of scalars to $\RR$ of the universal Cartan of $H(k_v)$ and $\mathfrak t_\CC$ is the complexification of its Lie algebra. The character $\delta^\frac{1}{2}: T\to \Rplus \simeq A$ gives rise to a map $\mathfrak t_\CC \to \mathfrak a_\CC$, and hence a pull-back map
$$\CC[\mathfrak a_\CC]^W \to \CC[\mathfrak t_\CC]^W.$$
Denote by $\mathfrak z_0(\mathfrak h_v)$ the subring corresponding to the image of this map. Any $Z\in \mathfrak z_0(\mathfrak h_v)$ acts on $\pi_s$ by the scalar $\check Z(s)$, where we have used the identification $\mathfrak a_\CC \simeq \CC$. It acts on $\pi_{-s}$ by the same scalar, since $W$ acts on $\mathfrak a_\CC$ by $s\mapsto -s$.

Let $Z_i\in \mathfrak z_0(\mathfrak h_v)$ be the element with Harish-Chandra transform $\check Z_i(s) = (s^2-s_i^2)$ ($i=1,2$). Then $Z_i$ acts as zero on $\pi_{s_i}$, and hence by \eqref{short-SsH} we have $Z_i\varphi_i \in \mathcal S([H])$. On the other hand, $|\check Z_i|$ is bounded below away from $s_i$, and we have:
$$ \left< \widecheck{(Z_1\varphi_{1})_B}(s), \widecheck{(Z_2\varphi_{2})_B}(-s)\right>_{\pi_s} = \check Z_1(s)\check Z_2(-s) \left< \widecheck{\varphi_{1,B}}(s), \widecheck{\varphi_{2,B}}(-s)\right>_{\pi_s},$$
so the statement of convergence reduces to the case $\varphi_i\in \mathcal S([H])$. In this setting it is of course known, but I remind that the argument uses a non-trivial element $Z \in \mathfrak z_0(\mathfrak h_v)$ that annihilates the discrete summands on the right hand side of \eqref{Plancherel-onaxis-prelim}, to arrive at
$$ \left<Z\varphi, \overline{Z\varphi}\right> = \frac{1}{2\pi i}\int_{0}^{+i\infty} |\check Z(s)|^2 \left< \widecheck{\varphi_B}(s), \overline{\widecheck{\varphi_B}(s)}\right>_{\pi_s} ds.$$
Together with the positivity of the hermitian forms in the integrand, this proves absolute convergence.

Now, the left hand side of the equality is also a continuous function of $\varphi_1\otimes \varphi_2$. Subtracting the continuous integral, we deduce that the same is true for the sum of the discrete terms -- call this sum $L(f_1, f_2)$. 

Notice that each of the discrete summands of \eqref{Plancherel-onaxis-prelim} -- denote them by $L_1, L_2, \dots$ -- is supported on a different set of points of the spectrum of $\mathfrak z_0(\mathfrak h_v)$; that is, if we let $\mathfrak z_0(\mathfrak h_v)$ act on the space of functionals by $(ZL_i)(f_1, f_2) = L_i (Zf_1, f_2)$, the annihilators $I_i$ of the $L_i$'s are relatively prime. On the other hand, for fixed $f_1, f_2$, it is clear from the expressions of these functionals that the map
$$ Z\mapsto L_i(Zf_1, f_2)$$ is a linear combination of the first few derivatives of $\check Z$ at the corresponding point $s'$ where the residue in the expression of $L_i$ is taken, with coefficients which a priori depend on $f_1, f_2$. Taking suitable elements $Z_j$ in $\bigcap_{l\ne i} I_l$, we can express these coefficients as a linear combination of the terms $L_i(Z_j f_1, f_2)$ which, by our choice of $Z_j$, is equal to $L(Z_j f_1, f_2)$. Thus, these coefficients are continuous functionals of $\varphi_1\otimes \varphi_2$; taking now $\check Z=1$, we obtain the same conclusion for the functionals $L_i$.

\end{proof}

\begin{example}\label{torus-RTF}
 Applying this theorem to the inner product of two asymptotically finite functions of the form $\Sigma\Phi$ as in Example \ref{torus}, we obtain the spectral decomposition of the relative trace formula of \cite[Theorem 5.3.1]{SaBE2}.
\end{example}

The bilinear forms $J_{s'}$ can be explicated; the reader can skip this part, which is (quite surprisingly!) involved, and not used elsewhere. Of course, if $s'\ne \pm s_1, \pm s_2$, they are equal to the (bilinear) ``inner product'' of the (regularized) projections of $\varphi_1, \varphi_2$ to the discrete automorphic representation that is spanned by residues of Eisenstein series at those points. To see that $J_{s'}$ factors through these projections is easy, by the equality 
$$ \Res_{s=s'}^- \left< \check f_1(s), \widecheck{\varphi_{2,B}}(-s)\right>_{\pi_s} = \left<\Res_{s=s'} \mathcal E f_1(s), \varphi_2\right>^*.$$
To show that it is equal to the inner product of these projections requires an inductive argument as in the proof of the theorem, s.\ the proof of \cite[V.3.2.(3).(ii)]{MW}. Since we are not developing $L^2$-theory here, we will avoid talking about orthogonal projections, and define a bilinear pairing on the residual representation as follows:

\begin{definition}\label{defdiscrete}
 Let $s'$ be a pole for $M(s)$ (equivalently, for the Eisenstein operator at $\pi_s$). Let $\pi_{s'}^\disc$ be the image of $\Res_{s=s'} M(s)$, so we have maps
 \begin{equation}\label{discretequot}
  \pi_{s'} \overset{\Res_{s=s'} M(s)}\twoheadrightarrow \pi_{s'}^\disc \hookrightarrow \pi_{-s'}.
 \end{equation}
 We define:
 \begin{itemize}
  \item A non-degenerate pairing $\pi_{s'}^\disc \hat\otimes \pi_{s'}^\disc \to \CC$ by 
  \begin{equation}\label{pairingdiscrete}\left<v_1, v_2\right>_{\pi_{s'}^\disc} = \left<\tilde v_1, v_2\right>_{\pi_{s'}} = \left<v_1, \tilde v_2\right>_{\pi_{-s'}},
  \end{equation}
  were $\tilde v_i$ is a preimage of $v_i$ under \eqref{discretequot}.
  \item When $\Phi\in \mathcal S(H(\adele))$, an identification of the operator $\pi_{s'}^\disc(\Phi dg)$ as an element of $\pi_{s'}^\disc\hat\otimes \pi_{s'}^\disc$ by considering the image of $\pi_{s'}(\Phi dg) \in \pi_{s'}\hat\otimes \pi_{-s'}$ under $\Res_{s=s'} M_1(s)$ or, equivalently, the image of $\pi_{-s'}(\Phi dg) \in \pi_{-s'}\hat\otimes \pi_{s'}$ under $\Res_{s=s'} M_2(s)$.
 \end{itemize}
\end{definition}

It turns out that explicating the forms $J_{s'}$ in terms of $\varphi_1, \varphi_2$ is not quite straightforward, and does not admit a uniform, general answer, because of possible coincidences between the exponents $s_1, s_2$ and the Eisenstein poles (or their opposites). The following theorem describes the answer, but before we formulate it, we must decompose the representations $\pi_s$ in terms of idele class characters $\chi$ of $[A]$. All our notation will be adjusted by replacing $s$ by $\chi$, and referring to the corresponding eigenspaces of $\pi_s$. The forms $J_s$ decompose accordingly. Residues are taken with respect to the same variable, i.e., $\Res_{\chi = \chi'}$ means the residue at $s=0$ as $\chi$ varies in the family $\chi'\delta^{\frac{s}{2}}$.

Here is the only point where we will use the meromorphic continuation of Mellin transforms of constant terms to $\Re(s)>0$; the result will not be used anywhere else. Moreover, the case where $\Re(s_1)=0$ or $\Re(s_2)=0$ is left to the reader. I remind that $\varphi\mapsto \varphi^\dagger$ denotes the asymptotics morphism of \eqref{short-SsH}.

\begin{theorem}\label{explicit-rankone}
 \begin{enumerate}
  \item Assume that $\chi_1\ne \chi_2$ and $\chi_1\ne$ a positive Eisenstein pole, $\Re(\chi_1)>0$. In that case, for $\chi' = \chi_1$ we have 
  \begin{equation}  J_{\chi'} (\varphi_1, \varphi_2) = \left< \mathcal E(\varphi_1^\dagger), \varphi_2 \right>.\end{equation} 
  The analogous formula holds for $\chi'=\chi_2$, if we interchange the indices $1$ and $2$ in both the hypotheses and the conclusion.
  \item Assume that $\chi_1 = \chi_2$, but $\ne$ a positive Eisenstein pole, with $\Re(\chi_i)>0$. Then the contribution of $\chi' = \chi_1 = \chi_2$ reads
  \begin{equation}
    J_{\chi'} (\varphi_1, \varphi_2) =  \Res_{\chi = \chi'} \left< \widecheck{\varphi_{1,B}}(\chi), \widecheck{\varphi_{2,B}}(\chi^{-1})\right>,
  \end{equation}
unless $\chi'^{-1}$ is a (negative) Eisenstein pole, in which case the intertwining operator $M(\chi)$ vanishes at $\chi=\chi'$, and, considering $M'(\chi') =$ the derivative of $M(\chi)$ at $\chi'$, we have
  \begin{eqnarray} J_{\chi'} (\varphi_1, \varphi_2) =  \Res_{\chi = \chi'} \left< \widecheck{\varphi_{1,B}}(\chi), \widecheck{\varphi_{2,B}}(\chi^{-1})\right> \nonumber \\ -\left<\Res^+_{\chi = \chi'} \widecheck{\varphi_{1,B}}(\chi), M'(\chi') \Res^+_{\chi = \chi'} \widecheck{\varphi_{2,B}}(\chi)\right>.
  \end{eqnarray}
 \item Assume that $\chi'=$ a positive Eisenstein pole, $\chi' \ne \chi_1, \chi_2$. Then 
   \begin{equation}\label{ipdisc}   J_{\chi'} (\varphi_1, \varphi_2) = \left<\Res^-_{\chi = \chi'} \widecheck{\varphi_{1,B}}(\chi^{-1}), \Res^-_{\chi = \chi'} \widecheck{\varphi_{2,B}}(\chi^{-1})\right>_{\pi'},
   \end{equation}
   where $\pi'\subset I(\chi'^{-1})$ is the image of $\Res_{\chi = \chi'}M(\chi)$, and $\left<\,\, , \,\, \right>_{\pi'}$ is the ``inner product'' on $\pi'$ given by Definition \ref{defdiscrete}. 
 \item Finally, assume that $\chi' = \chi_1$ or $\chi' = \chi_2$ is a positive Eisenstein pole. Then 
 \begin{equation}
  J_\chi' (\varphi_1, \varphi_2) = \Res_{\chi = \chi'} \left< \widecheck{\varphi_{1,B}}(\chi), \widecheck{\varphi_{2,B}}(\chi^{-1})\right>.
 \end{equation}
 \end{enumerate}

\end{theorem}

\begin{proof}
It is enough to establish these claims for $\varphi_1$, $\varphi_2$ in the dense subspace spanned by pseudo-Eisenstein series, so let $\varphi_i = \Psi f_i$, $i=1,2$. 

First, it will be helpful to compute the expression 
$$ \Res_{\chi = \chi'} \left< \widecheck{\varphi_{1,B}}(\chi), \widecheck{\varphi_{2,B}}(\chi^{-1})\right>.$$ 
By \eqref{constantterm-Psi}, it is equal to 
$$ \Res_{\chi = \chi'} \left(\left< \check f_1(\chi), \check f_2(\chi^{-1})\right> + \left<M(\chi^{-1})\check f_1(\chi^{-1}), \check f_2(\chi^{-1})\right> + \right.$$
$$ \left. + \left< \check f_1(\chi), M(\chi) \check f_2(\chi)\right> + \left< M(\chi^{-1})\check f_1(\chi^{-1}), M(\chi)\check f_2(\chi)\right>     \right).$$

We number the terms, as they appear, by (I)--(IV). 

Assume first that $\chi'=\chi_1 \ne \chi_2$, $\chi_1>0$ and $\ne$ to a positive Eisenstein pole. Then the third line of \eqref{Plancherel-onaxis-prelim} is the residue of the sum of terms (I) \& (III), and can be written as $\left<\varphi_1^\dagger, \widecheck{\varphi_{2, B}}(\chi_1^{-1})\right> = \left<\mathcal E\varphi_1^\dagger, \varphi_2\right>$. By symmetry (or direct reasoning, using the second line of \eqref{Plancherel-onaxis-prelim}), the analogous result holds if the indices $1$ and $2$ are interchanged.

Next, assume that $\chi'=\chi_1=\chi_2 >0$ and $\ne$ to a positive Eisenstein pole. Then the second line of \eqref{Plancherel-onaxis-prelim} is the residue of the sum of terms (I), (III), and (IV). We will express the residue of the second term in terms of $\varphi_1$ and $\varphi_2$: we claim that it is equal to 
$$\left<\Res^+_{\chi = \chi_2} \widecheck{\varphi_{1,B}}(\chi), M'(\chi_2)\Res^+_{\chi = \chi_2} \widecheck{\varphi_{2,B}}(\chi)\right>$$
(an expression which is symmetric in $\varphi_1$ and $\varphi_2$). Notice, first of all, that the term (II) is holomorphic at $\chi = \chi_2$, unless $\chi_2^{-1}$ is a (negative) pole of the intertwining operator (equivalently, the Eisenstein series). We can write its residue as 
$$\left<(\Res_{\chi = \chi_2} M(\chi^{-1}))\check f_1(\chi_2^{-1}), \check f_2(\chi_2^{-1})\right>  = $$
$$ = \left<(\Res_{\chi = \chi_2} M(\chi^{-1}))\check f_1(\chi_2^{-1}), M'(\chi_2) (\Res_{\chi = \chi_2} M(\chi^{-1}))\check f_2(\chi_2^{-1})\right>,$$
and the claim follows.

Next, assume that $\chi'=$ a positive Eisenstein pole, $\chi'\ne \chi_1, \chi_2$. 
Then the third line of \eqref{Plancherel-onaxis-prelim} is equal to the residue of the term (III), and can also be written as 
$$\left< \check f_1(\chi'), (\Res_{\chi = \chi'} M(\chi)) \check f_2(\chi')\right>.$$
The second term in this bilinear pairing is equal to $\Res^-_{\chi = \chi'} \widecheck{\varphi_{2,B}}(\chi^{-1})$, and the pairing depends only on the image of $\check f_1(\chi')$ via $(\Res_{\chi = \chi'} M(\chi))$, which is equal to $\Res^-_{\chi = \chi'} \widecheck{\varphi_{1,B}}(\chi^{-1})$. It can thus be written as 
$$ \left<\Res^-_{\chi = \chi'} \widecheck{\varphi_{1,B}}(\chi^{-1}), \Res^-_{\chi = \chi'} \widecheck{\varphi_{2,B}}(\chi^{-1})\right>_{\pi'},$$
where $\pi'$ is the image of $\Res^-_{\chi = \chi'}M(\chi)$.

Next, assume that $\chi' =\chi_1 \ne \chi_2$ is a positive Eisenstein pole. Then $\chi_2 \ne \chi'^{-1}$ (because $\chi_2 \ne \chi_1^{-1}$, by assumption), and 
the third line of \eqref{Plancherel-onaxis-prelim} is equal to the residue of the sum of terms (I) and (III). On the other hand, 
 since $\chi'$ is a pole of the Eisenstein series (and the intertwining operator), $\chi'^{-1}$ is not. Moreover, since $\chi' \ne \chi_1^{-1}, \chi_2, \chi_2^{-1}$, the terms (II), (IV) are holomorphic at $\chi'$, thus the contribution of $\chi'$ is equal to $\Res_{\chi = \chi'} \left< \widecheck{\varphi_{1,B}}(\chi), \widecheck{\varphi_{2,B}}(\chi^{-1})\right>$. By symmetry (or direct reasoning, using the second line of \eqref{Plancherel-onaxis-prelim}), the same holds if $\chi'=\chi_2 \ne \chi_1$ is a positive Eisenstein pole.
 
Finally, consider the case $\chi' = \chi_1 = \chi_2$, a positive Eisenstein pole. Then the second line of \eqref{Plancherel-onaxis-prelim} is equal to the residue of the sum of the terms (I), (III), and (IV), but the term (II) has no pole, so again it can be written in the same way.

\end{proof}

\section{Rank two}

\subsection{Kernel functions and their constant terms}

Recall that by ``asympotically finite functions'' on $[G]$ we mean asymptotically finite with respect to the ``diagonal'' partial compatification $\overline{[G]}^D$, i.e.\ elements of the space $\mathcal S^+_s([G])$ of Definition \ref{asymptfinite-G} -- and similarly for $[G]_B$. The theory can be developed more generally, but we will not need it for the purposes of the trace formula. 

Recall that we have a short exact sequence
$$ 0 \to \mathcal S([G]) \to \mathcal S^+_s([G]) \to \Pi_s \to 0.$$
where $\Pi_s = \mathcal S(A^\diag\backslash [G]_B,\delta^{\frac{s}{2}})$.

Let now $G=\PGL_2\times \PGL_2$, acting by left and right multiplication on $H=\PGL_2$. For $\Phi\in\mathcal S(H(\adele))$ we set:
$$ K_{\Phi}(g_1,g_2):= \Sigma\Phi(g_1,g_2) := \sum_{\gamma\in H(k)} \Phi(g_1^{-1}\gamma g_2) \in C^\infty([G]).$$
This is of course the kernel of the convolution operator defined by $\Phi$ on automorphic forms; however, since we will not directly compute/regularize its trace, we prefer to think of it as a theta series, an automorphic function on $[G]$, and try to compute/regularize the inner product of two such (which of course is the trace of the convolution of two operators), 

\begin{proposition}\label{propasfinite}
The above map $\Phi\mapsto K_\Phi$ defines a continuous map
$$\mathcal S(H(\adele)) \to \mathcal S^+_0([G]),$$
where on the right we have the space of asymptotically finite functions with respect to the ``diagonal'' embedding $\overline{[G]}^D$ introduced in \S \ref{ssSchwartz}, with trivial (normalized) exponent for the $A^\diag$-action.
\end{proposition}

\begin{proof}
We compute constant terms. Consider a proper parabolic $P=P_1\times P_2$ of $G$. Of course, each of $P_1$ and $P_2$ is either $H$, or the Borel $B$ of $H$. 
Consider first the case $P=B\times H$. Then, for $a\in B/N$,
$$ (\Sigma\Phi)_P(a,g) = \sum_{\gamma\in N\backslash H(k)} \varphi(a,\gamma g),$$
where $\varphi(a, g) = \int_{N(\adele)} \Phi(a^{-1}ng) dn$. 
When $a$ approaches the cusp, i.e., $\delta(a)\to \infty$, this is of rapid decay. 

Proposition \ref{constantterm-kernel}, that follows, calculates the constant term of $K_\Phi$ with respect to $P=B\times B$, as a sum of two terms $K_{\Phi,B}$ and $R_{1/2}K_{\Phi,B}$, with $K_{\Phi,B}\in \Pi_0=\mathcal S(A^\diag\backslash [G], \delta^0)$.  In a neighborhood of the $\infty_B\times \infty_B$-cusp, the term $R_{1/2}K_{\Phi,B}$ is of rapid decay. By Theorem \ref{approxconstantterm}, the stated description of $K_\Phi$ follows.
\end{proof}

Before we state the description of the $B\times B$-constant term that was used in the last proof, we need fix our notation for Mellin transforms on $[G]_B$. We have an action of $A\times A$ on $C^\infty([G]_B)$ (normalized as in \eqref{action-normalized-H}), and its eigenspaces are the spaces $\pi_{s_1} \hat\otimes \pi_{s_2}$, where $\pi_s = C^\infty(A\backslash[H],\delta^\frac{s}{2})$, as before. Identifying $A\simeq \RR^\times$ as in rank one (namely, through the character $\delta^\frac{1}{2}$), the parameters $s_1, s_2$ live in $\mathfrak a^*_\CC = \CC$. The Mellin transform of an element $f\in \mathcal S([G]_B)$ is the entire section $(s_1,s_2)\mapsto \check f(s_1,s_2) \in \pi_{s_1} \hat\otimes \pi_{s_2}$, given by 
\begin{equation}\label{Mellin-GB}
 \check f(s_1,s_2) (g)= \int_{A\times A} f((a_1,a_2)g) \delta^{-\frac{1+s_1}{2}}(a_1) \delta^{-\frac{1+s_2}{2}}(a_2) d(a_1,a_2).
\end{equation}

We will also apply the transform to functions that are invariant under $A^\diag$ or $A^\adiag$ (under the normalized action). In those cases, it obviously doesn't make sense as a function of two variables $(s_1,s_2)$, but it makes sense as a section over $(\mathfrak a^\diag)^\perp_\CC = \mathfrak a_\CC^{*,\adiag}$, resp.\ $(\mathfrak a^\adiag)^\perp_\CC = \mathfrak a_\CC^{*,\diag}$; for example, if the function $f$ is invariant under $A^\adiag$, we will have
\begin{equation}\label{Mellin-GB-adiag}
 \check f(s) (g)= \int_{A} f((a,1)g) \delta^{-\frac{1+s}{2}}(a) da.
\end{equation}
Strictly speaking, this should be interpreted (and denoted) as a generalized function on $\mathfrak a^*_\CC\times \mathfrak a^*_\CC$, supported on the diagonal; however, it will always be clear from the context whether we are considering the Mellin transform of a general function on $[G]_B$ (which is a function on $\mathfrak a^*_\CC\times \mathfrak a^*_\CC$) or of an $A^\adiag$- (or $A^\diag$-)invariant function (in which case it will be denoted as a function of one variable).

Finally, we fix identifications $(\mathfrak a^\diag)^\perp_\CC = \mathfrak a_\CC^*$ and $(\mathfrak a^\diag)^\perp_\CC = \mathfrak a_\CC^*$. For the former, there is a natural choice, namely the diagonal embedding $\mathfrak a_\CC^* = \mathfrak a^{*,\diag}_\CC = (\mathfrak a^\adiag)^\perp_\CC \subset \mathfrak a_\CC^* \times \mathfrak a_\CC^*$. For the latter, we use the restriction of a character to the \emph{first} copy in order to identify with $\mathfrak a^*_\CC$, that is: 
$\mathfrak a_\CC^* \ni s \mapsto (s, -s) \in \mathfrak a^{*,\adiag}_\CC = (\mathfrak a^\adiag)^\perp_\CC \subset \mathfrak a_\CC^* \times \mathfrak a_\CC^*.$
Hence, for a function $F$ on $[G]_B$ which is invariant under $A^\adiag$, its Mellin transform $\check F(s)$ belongs to $\pi_s \hat\otimes \pi_s$, while if the function is invariant under (the normalized action of) $A^\diag$, its Mellin transform $\check F(s)$ belongs to $\pi_s \hat\otimes \pi_{-s}$.

\begin{proposition}\label{constantterm-kernel}
For the constant term of $K_\Phi$ we have
\begin{equation}\label{diagplusadiag}
(K_\Phi)_B = K_{\Phi,B} + R_{1/2}K_{\Phi,B},
\end{equation}
where 
$$ K_{\Phi,B}(g_1,g_2)  = \sum_{\alpha\in A(k)} \int_{N(\adele)} \Phi(g_1^{-1} n \alpha g_2 ) dn$$
is the kernel for the convolution action of $\Phi dh$ on $L^2([H]_B)$, while $R_{1/2} K_{\Phi,B}$ is its Radon transform \eqref{Radon} \emph{in either of the two variables}:
$$ R_{1/2}K_{\Phi,B}(g_1,g_2) = \sum_{\alpha\in A(k)} \int_{N^2(\adele)} \Phi(g_1^{-1} n_1^{-1} w \alpha n_2 g_2) d(n_1,n_2).$$
The map $\Phi\mapsto K_{\Phi,B}$ represents a continuous morphism:
$$\mathcal S(H(\adele))\to \Pi_0=\mathcal S(A^\diag\backslash [G], \delta^0),$$
and the map $\Phi\mapsto R_{1/2}K_{\Phi,B}$ represents a continuous morphism into the subspace of 
$$ C^\infty(A^\adiag\backslash [G])$$
consisting of those functions which in a neighborhood of the cusp in $A^\adiag\backslash [G]$ are of rapid decay (together with their derivatives), and are bounded by elements of $\pi_{-1-\epsilon}\otimes\pi_{-1-\epsilon}$, for any $\epsilon>0$ (endowed with its natural topology as a countable strict limit of Fr\'echet spaces).

Therefore, the Mellin transform of $K_{\Phi,B}$ is an entire section over $(\mathfrak a^\diag)^\perp_\CC = \mathfrak a_\CC^{*,\adiag}$, while the Mellin transform of $R_{1/2}K_{\Phi,B}$, as a section over $(\mathfrak a^\adiag)^\perp_\CC = \mathfrak a_\CC^{*,\diag}$,  is defined by the convergent integral \eqref{Mellin-GB-adiag} (and is holomorphic) when $\Re(s)<-1$ ($s\in \mathfrak a_\CC^{*,\diag}$). The latter admits meromorphic continuation to the region $\Re(s)\le 0$, with poles only the opposites of the (simple) poles of Eisenstein series for $\Re(s)>0$.

Finally, the Mellin transforms of the above functions have the following invariant properties with respect to intertwining operators:
\begin{equation}\label{diagtodiag}
\widecheck{K_{\Phi,B}}(s) = M_1(-s)M_2(s) \widecheck{K_{\Phi,B}}(-s),
\end{equation}
\begin{equation}\label{adiagtoadiag}
\widecheck{R_{1/2}K_{\Phi,B}}(s) = M_1(-s)M_2(-s) \widecheck{R_{1/2}K_{\Phi,B}}(-s),
\end{equation}
and
\begin{equation}\label{diagtoadiag}
\widecheck{R_{1/2}K_{\Phi,B}}(s) = M_1(-s) \widecheck{K_{\Phi,B}}(-s) = M_2(-s) \widecheck{K_{\Phi,B}}(s),
\end{equation}
where $M_1$ and $M_2$ denote the standard intertwining operators in the first, resp.\ second variable.
\end{proposition}

\begin{remark}
Although $\delta^0$ is the trivial character, the notation $ \mathcal S(A^\diag\backslash [G], \delta^0)$ is meant to remind that we are talking about $A^\diag$-invariant functions under the \emph{normalized} action \eqref{action-normalized-H}; under the unnormalized action, those would be eigenfunctions with eigencharacter $\delta$. (On the other hand, for $A^\adiag$ the normalized and unnormalized actions coincide, so we can unambiguously write $C^\infty(A^\adiag\backslash [G])$.)
\end{remark}

\begin{remark}
The embedding \begin{equation}\label{HSsubspace} \pi_{s} \hat\otimes \pi_{-s}\to \End(\pi_s)\end{equation}
identifies 
\begin{equation}\label{Mellinisconv}
\widecheck{K_{\Phi,B}}(s) = \pi_s(\Phi dg).
\end{equation}
\end{remark}

The proposition above should be compared with Proposition \ref{Radon-s} and Corollary \ref{constantterm-s}. In the next susbsections we will generalize some of its statements to arbitrary asymptotically finite functions on $[G]$.

\begin{proof}

By the Bruhat decomposition, we have:
\begin{equation} (\Sigma\Phi)_P(a_1,a_2) = \sum_{w\in W} \sum_{\alpha\in A(k)} \int_{N\cap {^wN}\backslash N\times N (\adele)} \Phi(a_1^{-1} n_1^{-1} w \alpha n_2 a_2) d(n_1,n_2),\end{equation}
where $W$ denotes the Weyl group of $H$.

The summand corresponding to $w=1$ varies under $\delta(a)$ with respect to the unnormalized diagonal action of $A$; with respect to the normalized action, this is the trivial character. Otherwise, it is of rapid decay on $[A^\diag\backslash A\times A]$. The summand corresponding to $w=$the non-trivial element of the Weyl group is invariant under ${^{(w,1)}A^\diag}=$ the anti-diagonal action of $A$ (the normalized and unnormalized action coincide for this group), and as $\delta(a)\to \infty$ they it is of rapid decay on $(a,1)$ (or, equivalently, on $(1,a)$).

The rapid decay of $K_{\Phi,B}$ and $R_{1/2}K_{\Phi,B}$ in the stated directions is immediate. Notice that 
\begin{equation}\label{Radon2var}R_{1/2}K_{\Phi,B} = R_1 K_{\Phi,B} = R_2 K_{\Phi,B},\end{equation} where $R_1, R_2$ denote Radon transform in the first, resp.\ second variable. Also, each $H(\adele)\times\{1\}$ and $\{1\}\times H(\adele)$-orbit on $A^\diag\backslash[G]_B$ is isomorphic to $[H]_B$, so that by choosing a ``semi-algebraic'' section of the map $[H]_B\to A\backslash[H]_B$ (as in Observation \ref{observation-Schwartz}), we can think of elements of $\mathcal S(A^\diag\backslash[G]_B)$ as smooth functions on $A\backslash [H]_B$ valued in $\mathcal S([H]_B)$. The properties of $R_{1/2}K_{\Phi,B}$ (and its Mellin transform) now follow from Proposition \ref{Radon-s}. The relation \eqref{diagtoadiag} follows from \eqref{Radon2var}, while \eqref{diagtodiag}, when $\widecheck{K_{\Phi,B}}(s)$ is interpreted as $\pi_s(\Phi dg)$ (s.\ \eqref{Mellinisconv}), is just the statement that $\pi_s(\Phi dg) \circ M(-s) = M(-s) \circ \pi_{-s}(\Phi dg)$, together with the Maa\ss--Selberg relation \eqref{FE}. Finally, \eqref{adiagtoadiag} follows from the other two.
\end{proof}

\subsection{Pseudo-Eisenstein series of asymptotically finite functions}

Let $\mathcal S([G])_{\npr}$ denote the subspace of elements of $\mathcal S([G])$ where some non-trivial constant term (with respect to $B\times H$ or $H\times B$) vanishes, and $\mathcal S^+_s([G])_\pr =  \mathcal S^+_s([G])\cap \mathcal S([G])_{\npr}^\perp$. The index ``$\pr$'' stands for ``principal Eisenstein'', and the index ``$\npr$'' for ``non-principal''.

Notice that, as in rank one, due to Theorem \ref{approxconstantterm} on approximation by the constant term, the integral of an element of $\mathcal S^+_s([G])$ against an element of $\mathcal S([G])_{\npr}$ is absolutely convergent, it makes sense to say that an element of $\mathcal S^+_s([G])$ is orthogonal to $\mathcal S([G])_{\npr}$.

In analogy to Proposition \ref{Psiprop-s} and Lemma \ref{germsinfty}, we have:

\begin{proposition}\label{Psiprop-sG}
 For every $s\in \CC$, there is a direct sum decomposition:
\begin{equation}\label{cuspdecomp-sG}\mathcal S^+_{s}([G]) = \mathcal S^+_s([G])_\pr \oplus \mathcal S([G])_{\npr}.\end{equation}
The pseudo-Eisenstein sum $\Psi f(g) = \sum_{\gamma\in (B\times B)\backslash G(k)} f(\gamma g)$ is convergent on $\mathcal S^+_{s} ([G]_B)$, for any $s\in \CC$, and maps continuously, with dense image, into $\mathcal S^+_{s} ([G])_\pr$. Its composition with the constant term is given by the formula:
\begin{equation}\label{constantterm-Psi-G}
 (\Psi f)_B = f+ R_1 f+ R_2 f + R_1 R_2 f,
\end{equation}
where, again, $R_1$ and $R_2$ denote Radon transforms in the two variables.

Consider the short exact sequences:
$$  0 \to \mathcal S([G]_B) \to \mathcal S^+_s([G]_B) \to \Pi_s \to 0,$$
$$ 0 \to \mathcal S([G]) \to \mathcal S^+_s([G]) \to \Pi_s \to 0. $$
Since $\Psi$ preserves Schwartz spaces, it induces an endomorphism:
$$ \Pi_s \to \Pi_s$$
by passing to the quotients of the above sequences. This map is the identity.
\end{proposition}

The proof is completely analogous to that of Proposition \ref{Psiprop-s}, and will be omitted.

Now we examine constant terms of pseudo-Eisenstein series, and more generally of asymptotically finite functions on $[G]$. We have the following analog of  
Corollary \ref{constantterm-s}:

\begin{corollary}\label{constantterm-s-G}
The constant term morphism:
$$ \mathcal S^+_{s_0}([G])\to C^\infty ([G]_B)$$
has image in the space of smooth functions $f$ which have the following properties:

\begin{itemize}
 \item In a neighborhood of the cusp $\infty_B\times\infty_B$, they concide with elements of $\mathcal S^+_{s_0}([G]_B)$.
 \item Away from the cusp, they are bounded (together with their derivatives under $U(\mathfrak g)$) by elements of $\pi_{-\sigma-\epsilon}\otimes \pi_{-\sigma-\epsilon}$, for any $\epsilon>0$, where $\sigma = \max\{\Re(s_0), 1\}$; in particular, their Mellin transform \eqref{Mellin-GB} makes sense as a regularized integral, for $\Re(s_1), \Re(s_2)<-\sigma$.
 \item Their Mellin transform has meromorphic continuation to 
 $$\{(s_1,s_2)| \Re(s_1), \Re(s_2)\le 0\},$$ with simple poles of the form $s_i = -s'$ ($i=1,2$), where $s'$ is an Eisenstein pole for $\PGL_2$, as well as simple poles on the divisors $\{(s_1, s_2)| \pm s_1 \pm s_2 = s_0\}$ (for any combination of the signs).
\end{itemize}
\end{corollary}

\begin{proof}
 The proof is completely analogous to that of Corollary \ref{constantterm-s}, including an analog of Proposition \ref{Radon-s}, and is left to the reader. I only note that the pole on the divisor $\{(s_1, s_2)| s_1 + s_2 = s_0\}$ appears already for the Mellin transform of an element of $\mathcal S^+_{s_0}([G]_B)$, and the other poles will appear as one calculates the Mellin transform for the constant term of a pseudo-Eisenstein series, starting from \eqref{constantterm-Psi-G}.
\end{proof}

\begin{remark}
 Again, the target space of this corollary has a natural topology, and the proof shows that the constant term map is continuous with respect to that topology.
\end{remark}

Note that, as in rank one or even our model space $\Rplus$, the Mellin transform, thought of as a meromorphic function of two variables, does not completely determine the constant term; one needs a theory of ``charged Laurent expansions'' to invert the transform. There is an elegant way to do this, but it will be postponed to future work that will treat the case of arbitrary rank. For now, suffice it to observe that the Mellin transform of a non-trivial constant term may be zero -- this, in fact, will be the case for kernel functions. All the information for those functions is contained in their ``charged Laurent expansions'' along the singular hyperplanes, of which we will not make any explicit mention.

\subsection{Variation of the inner product with the exponent}\label{ssvariation}

The trace formula, as a distribution on $H(\adele)\times H(\adele)$, should be a regularized inner product
\begin{equation}\label{TF-ideal} \left< K_{\Phi_1}, K_{\Phi_2}\right>_{[G]}^*=\int^*_{[G]} K_{\Phi_1}(g_1,g_2) K_{\Phi_2}(g_1,g_2) d(g_1,g_2)
\end{equation}
between two kernel functions. However, because these are asymptotically finite functions with \emph{trivial} regularized exponents (cf.\ Proposition \ref{propasfinite}), the inner product does not exist, not even in the regularized sense of \S \ref{ssregints}. Therefore, what the trace formula really computes is a Laurent expansion, as the exponent of one of the asymptotically finite functions above varies. 

More precisely, we will replace $K_{\Phi_2}$ by a section $s\mapsto \varphi_s \in \mathcal S^+_s([G])$ with $\varphi = \varphi_0= K_{\Phi_2}$. The regularized inner product will then be defined for $s\ne 0$ and will have a simple pole at $s=0$:
$$\left< K_{\Phi_1}, \varphi_s\right>_{[G]}^* = \frac{a_{-1}}{s} + a_0 + \dots.$$
The coefficient $a_{-1}$ is an invariant bilinear form on $(\Phi_1, \Phi_2)$, but it is not very interesting: spectrally, in only depends on the traces $\tr(\pi_s(\Phi_i dg))$ of the convolution operators defined by the $\Phi_i$'s on parabolically induced automorphic representations, s.\ Theorem \ref{STF-1}. The traces on cuspidal representations are ``contained'' in the coefficient $a_0$, which however is not an invariant distribution, and depends not only on $\Phi_1, \Phi_2$, but also on the derivative of the section $s\mapsto \varphi_s$ at $s=0$.

Since that derivative lives in an infinite-dimensional space, to reduce the number of variables \textbf{we now make a choice:} We choose a  maximal compact subgroup $K\subset H(\adele)$, and a Borel subgroup $B$, so that $H(\adele) = B(\adele) K$. These choices are \emph{standard} in the theory of the trace formula, and the non-invariant trace formula, thought of as a distribution on $(H\times H)(\adele)$, depends on them. On the other hand, it will not depend on other choices made below, such as that of a neighborhood of the cusp.

We let $\Delta$ denote the function on $[H]_B$, or on $B(k)\backslash H(\adele)$, which pulls back to the function 
$$ b k \mapsto \delta^\frac{1}{2}(b) \,\, \, (b\in B(\adele), k\in K)$$
on $H(\adele)$. Since $B(k)\backslash H(\adele)\to [H]$ is an isomorphism in a neighborhood $U$ of the cusp, we will denote by the same letter the restriction of $\Delta$ to such a neighborhood $U$, considered as a subset of $[H]$.

We impose another condition on the pair $(B,K)$, that will only be used to obtain convenient geometric expressions for orbital integrals in Proposition \ref{reghyper}: 
\begin{equation}\label{condBK}
 \mbox{For every $n\in N$, $w\in W$, we have $\Delta(w^{-1} n)\le 1$.}
\end{equation}
Here, $N$ is the unipotent radical of $B$, and $w$ any element in the normalizer of a Cartan subgroup of $B$. (It obviously does not matter which.) The standard choices $B= $ upper triangular matrices, $K = $ the product of $H(\mathfrak o_v)$ at non-Archimedean places and of the image of the orthogonal/unitary group for the quadratic/hermitian form represented by the identity matrix, will do.

\begin{lemma}\label{constantcoeffs}
\begin{itemize}
\item The map $f\mapsto f_s(g_1,g_2):= f(g_1,g_2)\cdot \Delta^s(g_2)$ is a continuous isomorphism from $\mathcal S^+_0([G]_B)$ to $\mathcal S^+_s([G]_B)$, and hence can be used to trivialize the family of those Fr\'echet spaces over $\CC$ (the parameter space for $s$), in order to talk about holomorphic/meromorphic sections. These notions of holomorphic/meromorphic sections do not depend on the choice of $K$.
\item For $f$, $f_s$ as above, and any $F\in C^\infty([G]_B)$ which is of moderate growth and coincides with an element of $\mathcal S^+_0([G]_B)$ in a neighborhood of the cusp, the regularized integral 
$$\int^*_{[G]_B} F\cdot f_s$$
represents a continuous morphism from $\mathcal S^+_0([G]_B)$ to the space of meromorphic functions on $\CC$ with at most a simple pole at $0$.
\item The same statement is true if $f_s$ is replaced by $f_{U,s}$, the function which is equal to $f$ away from $[H]\times U$, where $U$ is a neighborhood of the cusp in $[H]$, and equal to $f_s$ in $[H]\times U$. Moreover, 
the class 
$$\left[ \int^*_{[G]_B} F\cdot f_s \right]$$ of the regularized integral modulo holomorphic functions which vanish at $0$ (equivalently, the $-1$- and $0$-terms of its Laurent expansion) does not change if we replace $f_s$ by $f_{U,s}$; in particular, it does not depend on the choice of $U$. It represents a continuous map from $\mathcal S^+_0([G]_B)$ to the $2$-dimensional space of expressions of the form $\frac{a_{-1}}{s}+a_0$, where $a_i\in\CC$.
\item The same statements are true when $F, \varphi\in \mathcal S^+_0([G])$, for the regularized integral 
$$\int^*_{[G]} F\cdot \varphi_{U,s}.$$
By abuse of notation, we will denote the corresponding terms in its Laurent expansion by 
$$ \left[ \int^*_{[G]} F\cdot \varphi_s \right],$$
although the function $\varphi_s$ is not defined in this case. (Only $\varphi_{U,s}$ is, for a choice of $U$.)
\item Assume that $\left[ \int^*_{[G]} F\cdot \varphi_s \right]= \frac{a_{-1}}{s}+ a_0$. Then these coefficients can also be recovered by ``truncation'', as follows: for every $T\gg 0$, let $U_T$ denote the neighborhood of the cusp in $[H]$ which is the image of the subset $\{b k| \delta^\frac{1}{2}(b)\ge e^T\}\subset H(\adele)$. Then
\begin{equation}\label{reltotruncation}
\int_{[G]\smallsetminus [H]\times U_T} F\cdot \varphi = - a_{-1} T + a_0 + o(e^{-NT}),
\end{equation}
for every $N>0$, where $o(e^{-NT})$ denotes a function whose quotient by $e^{-NT}$ tends to zero, as $T\to \infty$.
\end{itemize}
\end{lemma}

\begin{proof}
The proof is straightforward. I only outline the proof of \eqref{reltotruncation}: For a fixed $T_0$, let
$$A_0 = \int_{[G]\smallsetminus [H]\times U_T} F\cdot \varphi + \int_{B(k)\backslash H(\adele)\times U_{T_0}} (F\cdot \varphi-F^\dagger\cdot \varphi^\dagger),$$
where $F^\dagger, \varphi^\dagger \in \Pi_0=\mathcal S(A^\diag\backslash [G], \delta^0)$ denote the unique elements which are asymptotically equal to $F$, resp.\ $\varphi$, and for the second integral $U_{T_0}$ is considered as a subset of $B(k)\backslash H(\adele)$. It is clear from the rapid decay of the difference $F\varphi-F^\dagger\varphi^\dagger$ that the second summand is of order $ o(e^{-NT_0})$, so for any $T\ge T_0$ we have:
$$ \int_{[G]\smallsetminus [H]\times U_T} F\cdot \varphi  = A_0 + \int_{[H]\times(U_{T_0}\smallsetminus U_T)} F^\dagger \cdot \varphi^\dagger + o(e^{-NT}).$$
The action of $e^{T-T_0}\in A^\diag$ maps $[H]\times U_{T_0}$ bijectively onto $[H]\times U_T$, and since $F^\dagger, \varphi^\dagger$ are invariant with respect to the normalized action of $A^\diag$ -- which means that their inner product is invariant under $A^\diag$-translations -- we have $\int_{[H]\times (U_{T_0}\smallsetminus U_T)} F^\dagger \cdot \varphi^\dagger = A_{-1} (T-T_0)$, where
\begin{equation}\label{Aminusone}
A_{-1} = \int_{A^\diag\backslash [G]_B} F^\dagger\cdot \varphi^\dagger.
\end{equation}
Hence
$$ \int_{[G]\smallsetminus [H]\times U_T} F\cdot \varphi  = A_0 + A_{-1} (T-T_0) + o(e^{-NT_0}).$$

Consider now the integral over $[G]$ when $\varphi$ gets replaced by $\varphi_{U,s}$, where $U=U_{T_0}$. Then $\varphi^\dagger$ gets replaced by $\varphi^\dagger\cdot \Delta^s$, and similarly we have:
$$ \int_{[G]} F\cdot \varphi_{U,s} = A_0 + \int_{[H]\times U_{T_0}} F^\dagger \cdot \varphi^\dagger\Delta^s + O(s),$$
where $O(s)$ refers to the limit as $s\to 0$.
It is easy to see that $\int_{[H]\times U_{T_0}} F^\dagger \cdot \varphi^\dagger\Delta^s = -\frac{A_{-1}}{s}e^{sT_0}$ for the \emph{same} constant $A_{-1}$.

Thus, 
$$ \left[ \int^*_{[G]} F\cdot \varphi_s \right] = -\frac{A_{-1}}{s} + A_0 - A_{-1} T_0,$$
confirming the last claim, with $a_{-1}=-A_{-1}$ and $a_0= A_0-A_{-1}T_0$.
\end{proof}

\begin{remark}
Since we are only modifying the function $\varphi$ in the second variable, the pairing 
$$(F, \varphi)\mapsto \left[ \int^*_{[G]} F\cdot \varphi_s \right]$$
is invariant under the action of $H(\adele)$ on $F$ and $\varphi$ by left multiplication. 

In particular, if we apply this observation to $F=K_{\Phi_1}$, $\varphi = K_{\Phi_2}$, and we write $\Phi_1 = \Phi_1\star \delta_1 = L(\Phi_1 dg) \delta_1$, where $L$ denotes the action by left multiplication and $\delta_1$ is the delta function at $1\in H(\adele)$, we see that 
$$\left[ \int^*_{[G]} F\cdot \varphi_s \right] = \left[\int^*_{[H]} K_\Phi(h,h) \Delta_U^s(h) dh\right].$$ 
Here $\Delta_U$ is the function which is equal to $\Delta$ in the Siegel neighborhood $U=U_{T_0}$, and $1$ otherwise, and $\Phi = \Phi_1^\vee \star \Phi_2$, where $\Phi_1^\vee (g) = \Phi_1 (g^{-1})$. The last statement in the above lemma, then, connects this expression to the coefficients of $T$ for the truncated integral 
$$ \int^*_{[H]\smallsetminus U_T} K_\Phi(h,h) dh.$$
\end{remark}

\begin{definition}
We define a distribution on $H(\adele)\times H(\adele)$, valued in the two-dimensional space of expressions of the form $\frac{a_{-1}}{s}+a_0$, by
\begin{equation}\label{Laur-goal}
\TF(\Phi_1\otimes\Phi_2)=\left[\left< K_{\Phi_1}, (K_{\Phi_2})_s\right>_{[G]}^*\right]=\left[\int^*_{[G]} K_{\Phi_1} (K_{\Phi_2})_s\right],
\end{equation}
in the notation of the above lemma. We will write $a_{-1}=\TF_{-1}(\Phi_1\otimes\Phi_2)$ and $a_0=\TF_0(\Phi_1\otimes\Phi_2)$ for the corresponding coefficients of the Laurent expansion. The functional $\TF_0$ is (the functional of) the \emph{non-invariant Selberg trace formula}.
\end{definition}
Notice that we use the notation $\TF(\Phi_1\otimes\Phi_2)$ for the two-dimensional germ of the Laurent expansion, despite the fact that ``Selberg trace formula'' (to stick with tradition) refers only to the constant coefficient.

The coefficient $\TF_{-1}$ is evidently invariant, and can immediately be computed spectrally:

\begin{theorem}\label{STF-1}
We have 
\begin{equation}
\TF_{-1} (\Phi_1\otimes\Phi_2) = - \frac{1}{2\pi i}\int_{\sigma-i\infty}^{\sigma+i\infty} \left<\pi_s(\Phi_1 dg) , \pi_{-s}(\Phi_2 dg)\right> ds
\end{equation}
for any $\sigma \in \RR$, where $\left<\,\, , \,\,\right>$ here denotes the Hilbert--Schmidt inner product of operators.
\end{theorem}

The Hilbert--Schmidt inner product is understood here as a bilinear pairing between operators on $\pi_s$ and operators on $\pi_{-s}$. Equivalently, recalling \eqref{Mellinisconv}, it is the contraction pairing between $\widecheck{K_{\Phi_1,B}}(s)\in \pi_s\hat\otimes\pi_{-s}$, and $\widecheck{K_{\Phi_2,B}}(-s)\in \pi_{-s}\hat\otimes\pi_s$.

\begin{proof}
This follows directly from Mellin decomposition of the pairing \eqref{Aminusone}. Notice that for $F=K_{\Phi_1}$, $f=K_{\Phi_2}$ we have, in the notation of \eqref{Aminusone}, $F^\dagger=K_{\Phi_1,B}$ and $f^\dagger=K_{\Phi_2,B}$, by Proposition \ref{constantterm-kernel}.
\end{proof}

The strategy for developing the trace formula, now, is the following:

\begin{itemize}
\item On the geometric side, we will compute the geometric expansion of the Laurent coefficients of the expression \eqref{TF-ideal}, perturbed by the parameter $s$, i.e., of \eqref{Laur-goal}.

\item On the spectral side, there remains to compute a spectral expansion for the constant coefficient of the Laurent expansion \eqref{Laur-goal}. We will do this by first replacing $K_{\Phi_2}$ by $\varphi=\Psi(f)$, a pseudo-Eisenstein series. We will eventually arrive at an integral of terms depending only on $\varphi$, not on $f$ itself. By continuity, we will be able to apply this expression to an arbitrary element of $\mathcal S^+_{0}([G])$ -- in particular, to $K_{\Phi_2}$.
\end{itemize}

The equality of the resulting expansions for \eqref{Laur-goal} will then be the non-invariant trace formula. (Strictly speaking, it is the equality of the constant coefficients.) Later, we will also discuss the invariant trace formula.

Hence, our first goal is to compute:
\begin{equation}\label{goalwithpseudo}
\left[\int_{[G]}^* K_{\Phi_1} \Psi(f)_s \right]
\end{equation}

Note that, although it is possible, we will not replace $K_{\Phi_1}$ by an arbitrary element of $\mathcal S^+_0([G])$, thus exploiting some features of its constant term, as analyzed in Proposition \ref{constantterm-kernel}, that will shorten the analysis.

The following is an important starting point, extending the various equivalent descriptions of two-term Laurent expansions in Lemma \ref{constantcoeffs}:
\begin{lemma}\label{differentvariations}
We have
$$\left[\int_{[G]}^* K_{\Phi_1} \Psi(f)_s \right] = \left[\int_{[G]}^* K_{\Phi_1} \Psi(f_s) \right],$$
where $f_s (g_1,g_2) =\Delta^s(g_2) f(g_1,g_2)$, as in Lemma \ref{constantcoeffs}.
\end{lemma}

\begin{proof}
By definition, to compute the left hand side one needs to replace $\Psi(f)$, in a fixed neighborhood $[H]\times U$ of the cusp, by $\Psi(f)(g_1,g_2)\cdot \Delta^s(g_2)$. Let us denote by $\Delta_U$ the function on $[G]$ which is equal to $\Delta(g_2)$ on $[H]\times U$, and equal to $1$ otherwise.
 Consider the map
$$\pi_G: (B\times B)(k)\backslash G(\adele) \to [G].$$

Unravelling the definition of $\Psi$, we see that 
$$\left[\int_{[G]}^* K_{\Phi_1} \Psi(f)_s \right] = \left[\int_{[G]}^* K_{\Phi_1} \Psi(f) \Delta_U^s \right] =$$
$$ = \left[\int_{[G]_B}^* K_{\Phi_1} f \pi_G^* \Delta_U^s \right].$$

The last integral can be split into two parts: one in $[H]\times U$ (where $U$ is now considered as a neighborhood of the cusp in $[H]$), and one over the complement. The latter:
$$\int_{[G]_B\smallsetminus [H]\times U}^* (K_{\Phi_1})_B f \pi_G^* \Delta_U^s $$
is absolutely convergent for every $s$ because of the rapid decay of $f$, and hence 
$$\left[\int_{[G]_B\smallsetminus [H]\times U}^* (K_{\Phi_1})_B f \pi_G^* \Delta_U^s \right]= \int_{[G]_B\smallsetminus [H]\times U} (K_{\Phi_1})_B f .$$

We arrive at 
$$ \left[\int_{[G]}^* K_{\Phi_1} \Psi(f)_s \right]  = \left[\int_{[G]_B}^* (K_{\Phi_1})_B f_{U,s}  \right]$$
in the notation of Lemma \ref{constantcoeffs}, which by that lemma is equal to 
$$ \left[\int_{[G]_B}^* (K_{\Phi_1})_B f_s  \right].$$
By adjunction between constant term and pseudo-Eisenstein series, this proves the claim.
\end{proof}

Thus, \eqref{goalwithpseudo} is equal to
$$\left[ \int_{[G]_B}^* (K_{\Phi_1})_B f_s\right].$$
From Proposition \ref{constantterm-kernel} we know that the constant term $(K_{\Phi_1})_B$ is the sum of two functions $K_{\Phi_1,B}$ and $R_{1/2}K_{\Phi_1,B}$, where $K_{\Phi_1,B}$ is invariant under the normalized action of $A^\diag$ and $R_{1/2}K_{\Phi_1,B}$ is invariant under $A^\adiag$. We analyze the regularized inner product of each of them with $f_s$ separately. For notational convenience, from now on $\Phi=\Phi_1$.

\subsection{Inner product with the anti-diagonal term $R_{1/2}K_{\Phi,B}$.}

Here, the regularized inner product is actually a convergent integral:
$$ \int_{[G]_B} R_{1/2}K_{\Phi,B} \cdot f_s$$
for all $s$, therefore we can immediately take $s=0$, if we wish -- however, to fix ideas it is simpler to keep working with general $s$ for a while, since $s=0$ falls on the unitary spectrum.

Since $R_{1/2}K_{\Phi,B}$ is invariant under $A^\adiag$, we can write it as:
$$ \int_{A^\adiag\backslash [G]_B} R_{1/2}K_{\Phi,B} (g) \cdot \left(\int_{A^\adiag}f_s(ag) da\right) dg.$$

Identifying $(A\times A)/A^\adiag$ with $A$ via either of the two factors, and acting by it on $A^\adiag\backslash[G]$, we recall that 
the function $R_{1/2}K_{\Phi,B}$ is of rapid decay as $\delta(a)\to \infty$ (i.e., towards the cusp), and of moderate growth in the opposite direction, while the function in the inner integral is 
of rapid decay as $\delta(a)\to 0$ and of moderate growth towards the cusp. We can thus apply 
 Theorem \ref{asfinitealmostL2} to spectrally decompose this inner product. Notice that the Mellin transform of the function in the inner integral, under the $(A\times A)/A^\adiag$-action, is
$$ z\mapsto \int_A  \delta^{-\frac{1+z}{2}}(a') \int_{A}f_s((a, a'a^{-1}) g) da da' = $$
$$ = \int_{A\times A}  \delta^{-\frac{1+z}{2}}(a_1 a_2) \int_{A^\adiag}f_s((a_1, a_2) g) da_1 da_2 = \check f_s(z,z).$$
Hence we get a spectral decomposition away from the unitary line:
$$\int_{[G]_B} R_{1/2}K_{\Phi,B} \cdot f_s = \frac{1}{2\pi i}\int_{\sigma - i\infty}^{\sigma+i\infty} \left<\widecheck{R_{1/2}K_{\Phi,B}}(-z), \check f_s(z,z)\right>dz,$$
with $\sigma \gg 0$. (Recall from the comments following \eqref{Mellin-GB-adiag} that the Mellin transform of $f_s$ is understood as a function of two variables, while the Mellin transform of $R_{1/2}K_{\Phi,B}$ is a function of a single variable, living in $(\mathfrak a^*_\CC)^\diag$.) 

Notice that $\check f_s(z,z) \in \pi_z \otimes \pi_z$. Using multiplication by $\Delta^s$ to identify $\pi_z$ with $\pi_{z+s}$, we can also write:
$$\check f_s(z,z) = \check f(z, z-s).$$

Shifting the contour of integration to $\sigma=0$,  we pick up the residues of $\widecheck{R_{1/2}K_{\Phi,B}}$ corresponding to residual automorphic representations (s.\ Proposition \ref{constantterm-kernel}), and the residue corresponding to the pole of the function $(z_1,z_2)\mapsto \check f(z_1, z_2-s)$ along the plane $s+\mathfrak a_\CC^{*,\adiag}$, when $\Re(s)>0$. In the case $s=0$, which is the one of interest to us, we need to use a principal value integral as in \eqref{Mellininv-ext-PV}; this gives:

\begin{proposition}\label{propPhiadiag}
We have 
\begin{eqnarray} \int_{[G]_B} R_{1/2}K_{\Phi,B} \cdot f = \frac{1}{2\pi i} PV\int_{- i\infty}^{+i\infty} \left<\widecheck{R_{1/2}K_{\Phi,B}}(-z), \check f(z,z)\right>dz +\nonumber \\
+ \frac{1}{2} \left<\widecheck{R_{1/2}K_{\Phi,B}}(0), \Res_{z=0} \check f(z,z)\right> - \sum_{\Re(z')>0}   \left< \Res_{z=z'}\widecheck{R_{1/2}K_{\Phi,B}}(-z), \check f(z',z')\right>.\label{decompPhiadiag} \end{eqnarray}
\end{proposition}

To complete the analysis of the inner product with the anti-diagonal term, we explicate the discrete terms; we will also ``symmetrize'' the continuous contribution with respect to the intertwining operators. 

\begin{proposition}\label{discreteterms-explication}
Let $\varphi=\Psi(f)$. The discrete terms in \eqref{decompPhiadiag} can be expressed as follows:
\begin{itemize}
\item \begin{equation}\frac{1}{2} \left<\widecheck{R_{1/2}K_{\Phi,B}}(0), \Res_{z=0} \check f(z,z)\right> = \frac{1}{4}\left< M(0) \pi_0(\Phi dg), \check \varphi^\dagger(0)\right>,\end{equation}
where $\varphi^\dagger=f^\dagger\in \Pi_0= \mathcal S(A\backslash [G], \delta^0)$ is the asymptote of $\varphi$ (or of $f$), according to \eqref{short-SsG}, its Mellin transform is considered as a Hilbert--Schmidt operator by \eqref{HSsubspace}, and the pairing on the right is the pairing of Hilbert--Schmidt operators on $\pi_0$;
\item  \begin{equation}\label{discretecontr} -  \left< \Res_{z=z'}\widecheck{R_{1/2}K_{\Phi,B}}(-z), \check f(z',z')\right>= \left< \pi_{z'}^\disc(\Phi dg), \widecheck{\varphi_B}^\disc(z',z') \right>_{\pi_{z'}^\disc},\end{equation}
where $\varphi=\Psi(f)$, and $\pi_{z'}^\disc$ is the discrete automorphic representation obtained as the residue of Eisenstein series $\mathcal E: \pi_z \to C^\infty([H])$ at $z=z'$, the operator $\pi_{z'}^\disc(\Phi dg)$, here, is identified as an element of $\pi_{z'}^\disc \hat\otimes \pi_{z'}^\disc$ as in Definition \ref{defdiscrete}, and $\widecheck{\varphi_B}^\disc(z',z')$ is the image of $\widecheck{\varphi_B}^\disc(z',z') \in \pi_{z'}\hat\otimes\pi_{z'}$ in $\pi_{z'}^\disc\hat\otimes\pi_{z'}^\disc$ under \eqref{discretequot}.
\end{itemize}

The continuous term can also be written:
$$\frac{1}{2\pi i} \int_{0}^{+i\infty} \left<\widecheck{K_{\Phi,B}}(-z), M_1(-z) \check f(-z,-z) + M_2(z) \check f(z, z)\right>dz.$$

\end{proposition}

\begin{proof}
A simple analysis of Mellin transform on $\Rplus$ shows that 
$$ \lim_{z\to 0} z \int_{A} f((a,a)g) \delta^{-1-\frac{z}{2}}(a) da = f^\dagger(g),$$
and by integrating this relation over the action of $\{1\}\times A$ against a character, one obtains:
$$ \lim_{z\to 0} z \check f(-z'+z, z') = \lim_{z\to 0} z \int_A \int_A  f((a,a a')g) \delta^{-1-\frac{z}{2}}(a) da \delta^{-\frac{1+z'}{2}}(a') da'  = $$
$$ = \int_A f^\dagger(a') \delta^{-\frac{1+z'}{2}}(a') da' = \check f(z').$$
Thus, $\Res_{z=0} \check f(z,z) = \frac{1}{2} \check f(0)$, therefore
$$\frac{1}{2} \left<\widecheck{R_{1/2}K_{\Phi,B}}(0), \Res_{z=0} \check f(z,z)\right> = \frac{1}{4} \left<\widecheck{R_{1/2}K_{\Phi,B}}(0),\check f^\dagger(0)\right>.$$

Recalling from \eqref{diagtoadiag} and from \eqref{Mellinisconv} that $\widecheck{R_{1/2}K_{\Phi,B}}(0) = M(0)\widecheck{K_{\Phi,B}}(0) = M(0) \pi_0(\Phi dg)$, this can also be written as $\frac{1}{4}\left< M(0) \pi_0(\Phi dg), \check f^\dagger(0)\right>$; this proves the first claim.

For the second claim, using again the relation \eqref{diagtoadiag}, one sees that the element  $ - \Res_{z=z'}\widecheck{R_{1/2}K_{\Phi,B}}(-z) \in \pi_{-z'}\hat\otimes \pi_{-z'}$ lies in the subspace $\pi_{z'}^\disc\hat\otimes \pi_{z'}^\disc$ and is equal to $\pi_{z'}^\disc(\Phi dg)$. Thus, its pairing with $\check f(z',z')$ depends only on the image of the latter via the quotient $\Res_{z=z'} M_1(z) \cdot \Res_{z=z'} M_2(z) : \pi_{z'} \hat\otimes\pi_{z'} \to \pi_{z'}^\disc\hat\otimes\pi_{z'}^\disc$, and by Definition \ref{defdiscrete}, this is indeed the pairing
$$\left< \pi_{z'}^\disc(\Phi dg), \check f^\disc(z',z') \right>_{\pi_{z'}^\disc}.$$

Replacing $f$ by $\varphi_B$, the image of $\check f^\disc(z',z')$  in $\pi_{z'}^\disc\hat\otimes \pi_{z'}^\disc$ does not change, since the other summands of \eqref{constantterm-Psi-G} will have zero contribution to this residue, since they involve residues of the expression $M(z)M(-z) = \Id$. 

The statement about the continuous term follows from \eqref{diagtoadiag}.

\end{proof}

\subsection{Inner product with the diagonal term $K_{\Phi,B}$.} \label{ssdiagonal}

Now we return to \eqref{Laur-goal}, and analyze the contribution of the summand $K_{\Phi,B}$ of $(K_\Phi)_B$, according to Proposition \ref{constantterm-kernel}. Since both $K_{\Phi,B}$ and $f_s$ have asymptotically finite behavior with respect to the $A^\diag$ action in the direction of the cusp, with $K_{\Phi,B}$ being of moderate growth and $f_s$ being of rapid decay in all other directions, and since $K_{\Phi,B}$ is $A^\diag$-invariant (for the normalized action), we can write, for $s\ne 0$:

$$ \int^*_{[G]_B} K_{\Phi,B}(g) f_s(g) dg = \int_{A^\diag\backslash[G]_B} K_{\Phi,B}(g) \int_{A}^* f_s((a,a)g) \delta^{-1}(a) da dg= $$
\begin{equation}\label{F1} = \frac{1}{2\pi i}\int_{\sigma - i\infty}^{\sigma+ i \infty} \left<\widecheck{K_{\Phi,B}}(-z), \check f_s(-z,z)\right> dz,
\end{equation}
where $\sigma$ is arbitrary. (We will take $\sigma=0$.)

Notice that, for $s\ne 0$, $\check f_s$ is entire when restricted to $(\mathfrak a_\CC^*)^\adiag$ (which, we recall, we identify with $\mathfrak a_\CC^*$ by restriction of the characters to the first factor of $A$). But, of course, as $s\to 0$ its polar hyperplane tends to coincide with $(\mathfrak a_\CC^*)^\adiag$. This becomes clearer if we again identify $\check f_s(z_1,z_2)$ with $\check f(z_1,z_2-s)$, identifying the spaces $\pi_{z}$ and $\pi_{z+s}$ via multiplication by $\Delta^s$. This identification will become important here, since we are studying non-invariant terms, and we will take it for granted from now on. In particular, the derivative $M'(s)$ of $M(s)$ makes sense as a (non-equivariant) linear map from $\pi_s$ to $\pi_{-s}$.

The residue of the Laurent expansion of \eqref{F1} at $s=0$ was computed in Theorem \ref{STF-1}. Now we compute the constant coefficient. It can be written:
$$ \left[\int^*_{[G]_B} K_{\Phi,B}(g) f_s(g) dg \right]_0 = \frac{1}{2\pi i}\int_{- i\infty}^{i \infty} \left<\widecheck{K_{\Phi,B}}(-z), \left.\frac{d}{ds}\right|_{s=0} s \cdot \check f(z,z-s)\right> dz.$$

This last expression depends on $f$, not just on $\varphi= \Psi(f)$. To transform it into one that depends only on $\varphi$, we are going to use relation \eqref{diagtodiag}, to write it as 
$$ \frac{1}{2\pi i} \int_0^{i\infty} \left< \widecheck{K_{\Phi,B}}(-z) , \left.\frac{d}{ds}\right|_{s=0} s\cdot \check f(z,-z-s) + M_1(-z)M_2(z)\left.\frac{d}{ds}\right|_{s=0} s\cdot \check f(-z,z-s)\right> dz.$$

In the last term, we will replace $z-s$ by $z+s$, since in any case the expression is computing the constant coefficient of its Laurent expansion. Thus, the last term inside the inner product will become
$$M_1(-z)M_2(z)\left.\frac{d}{ds}\right|_{s=0} s\cdot \check f(-z,z+s).$$
Notice that this picks out the constant coefficient of $\check f(-z,z+s)$ at $s=0$ and then applies the intertwining operator $ M_1(-z)M_2(z)$. This is different from taking the constant coefficient of $M_1(-z)M_2(z+s)\check f(-z,z+s)$, and to account for that difference we add and subtract the appropriate term, to obtain:
$$ \frac{1}{2\pi i} \int_0^{i\infty} \left(\left< \widecheck{K_{\Phi,B}}(-z) , \left.\frac{d}{ds}\right|_{s=0} s\cdot \check f(z,-z-s) + \left.\frac{d}{ds}\right|_{s=0} s\cdot  M_1(-z)M_2(z+s)\check f(-z,z+s)\right> \right.$$
$$\left. - \left< \widecheck{K_{\Phi,B}}(-z) , \left(\left.\frac{d}{ds}\right|_{s=0} M_1(-z)M_2(z+s)\right) \lim_{s\to 0}s\check f(-z,z+s)\right>\right) dz.$$

The limit $\lim_{s\to 0}s\check f(-z,z+s)$ is simply $ \check f^\dagger (-z) = \check\varphi^\dagger(-z)$, where $f^\dagger=\varphi^\dagger\in \mathcal S(A^\diag\backslash [G],\delta^0)$ is the asymptotic of $f$ (or of $\varphi=\Psi(f)$). Thus we obtain:

\begin{proposition}\label{propPhidiag}
We have
$$ \left[\int^*_{[G]_B} K_{\Phi,B}(g) f_s(g) dg \right]_0  =  $$
$$ = \frac{1}{2\pi i} \int_0^{i\infty} \left( \left< \widecheck{K_{\Phi,B}}(-z) , \left.\frac{d}{ds}\right|_{s=0} s\left(\check f(z,-z-s) +  M_1(-z)M_2(z+s)\check f(-z,z+s)\right)\right> \right. $$
\begin{equation}\left. - \left< \widecheck{K_{\Phi,B}}(-z) , M_1(-z)M_2'(z) \check\varphi^\dagger(-z)\right> \right)dz\end{equation}

\end{proposition}

\subsection{Spectral side of the Selberg trace formula}

Combining Propositions \ref{propPhiadiag}, \ref{discreteterms-explication}, and \ref{propPhidiag}, we get:
$$\left[\int^*_{[G]} K_\Phi (g) \Psi(f_s) (g) dg \right]_0  = \frac{1}{4}\left< M(0) \pi_0(\Phi dg), \check \varphi^\dagger(0)\right> +
 \sum_{\Re(z')>0} \left< \pi_{z'}^\disc(\Phi dg), \pi_{z'}^\disc(\varphi dg) \right>$$
$$+ \left(\frac{1}{2\pi i} \int_{0}^{+i\infty} \left<\widecheck{K_{\Phi,B}}(-z), M_1(-z) \check f(-z,-z) + M_2(z) \check f(z, z)\right>dz \right.$$
$$ + \frac{1}{2\pi i} \int_0^{i\infty} \left< \widecheck{K_{\Phi,B}}(-z) , \left.\frac{d}{ds}\right|_{s=0} s\cdot \left(\check f(z,-z-s) +  M_1(-z)M_2(z+s)\check f(-z,z+s )\right)\right> $$
$$- \left. \left< \widecheck{K_{\Phi,B}}(-z) , M_1(-z)M_2'(z) \check\varphi^\dagger(-z)\right> dz\right).$$

We can combine the integrands in the second and third lines into
$$ \left< \widecheck{K_{\Phi,B}}(-z) , 
\left.\frac{d}{ds}\right|_{s=0} s\cdot \left(\check f(z,-z-s) + 
 M_1(-z) \check f(-z,-z-s) + \right.\right.$$
$$ \left.\left.+ M_2(z+s) \check f(z, z+s) + M_1(-z)M_2(z+s)\check f(-z,z+s) \right)\right> = $$
$$ = \left< \widecheck{K_{\Phi,B}}(-z) , 
\left.\frac{d}{ds}\right|_{s=0} s \cdot \widecheck{\varphi_B}(z,-z-s) \right>.$$

Now all the integrands are continuous functions of $\varphi$. This implies:

\begin{theorem}\label{spectral-general-ranktwo}
For $\Phi\in \mathcal S(H(\adele))$ and $\varphi\in \mathcal S^+_0([G])_\pr$ with asymptote $\varphi^\dagger\in \mathcal S(A^\diag\backslash [G],\delta^0)$, we have a decomposition:
$$\left[\left<K_\Phi, \varphi_s \right>\right]_0  = \frac{1}{4}\left< M(0) \pi_0(\Phi dg), \check \varphi^\dagger(0)\right> + \sum_{\Re(z')>0} \left< \pi_{z'}^\disc(\Phi dg), \pi_{z'}^\disc(\varphi dg) \right>$$
\begin{eqnarray}
+\frac{1}{2\pi i} \int_0^{i\infty} \left(\left< \widecheck{K_{\Phi,B}}(-z) , 
\left.\frac{d}{ds}\right|_{s=0} s \cdot \widecheck{\varphi_B}(z,-z-s) \right> \right. \nonumber \\ \left.- \left< \widecheck{K_{\Phi,B}}(-z) , M_1(-z)M_2'(z) \check\varphi^\dagger(-z)\right> \right)dz.
\end{eqnarray}
\end{theorem}

\begin{proof}
 When $\varphi = \Psi(f)$ is a pseudo-Eisenstein series, this has already been established. In the general case, by Proposition \ref{Psiprop-sG}, we may always subtract a pseudo-Eisenstein series so that the difference is a Schwartz function. This reduces us to the case $\varphi\in \mathcal S([G]) = \mathcal S([H])\hat\otimes\mathcal S([H])$. But this is the case of the inner product in rank one, since for $\varphi_1, \varphi_2\in \mathcal S([H])$ we have
 $$\left< K_\Phi, \varphi_1 \otimes \varphi_2\right>_{[G]} = \left<\varphi_1, \Phi\star \varphi_2\right>_{[H]}.$$
 Thus, the result is a special case of Theorem \ref{Plancherelrankone}.
 
\end{proof}

Finally, specializing to $\Phi = \Phi_1$, $\varphi = (K_{\Phi_2})_\pr$, where we have written $$K_{\Phi_2} = (K_{\Phi_2})_\pr \oplus (K_{\Phi_2})_\npr$$
according to Proposition \ref{Psiprop-sG}, the last term disappears. Indeed:
\begin{lemma}
 When $\varphi = (K_{\Phi_2})_\pr$, we have $\widecheck{(\varphi)_B}(z,-z-s) = 0$ for generic values of $s$.
\end{lemma}

\begin{proof}
 Indeed, the map $\Phi_2 \mapsto \widecheck{K_{\Phi_2,B}}(z,-z-s)$ is a morphism $\mathcal S(H(\adele))\to \pi_z\hat\otimes \pi_{-z-s}$, which has to be zero for generic $s$ because there is no non-zero invariant pairing between (the duals) $\pi_{-z}$ and $\pi_{z+s}$.
\end{proof}

The term $\widecheck{\varphi_B}^\disc(z',z')$ of \eqref{discretecontr} can be identified with $\pi_{z'}^\disc(\Phi_2 dg)$ under Definition \ref{defdiscrete}. 

Moreover, under the map $\pi_z\hat\otimes \pi_{-z}\to \End(\pi_z)$, and given that the adjunction formula \eqref{Msadjunction} extends to $M'(z)$, the element 
$$M_1(-z)M_2'(z) \widecheck{(K_{\Phi_2})_B^\dagger}(-z) = M_1(-z)M_2'(z)\widecheck{K_{\Phi_2,B}}(-z) \in \pi_z\hat\otimes \pi_{-z}$$ becomes the operator $$M(-z) \pi_{-z}(\Phi_2 dg) M'(z) = \pi_z(\Phi_2 dg) M(-z) M'(z) \in \End(\pi_z).$$
Finally, one easily sees that
$$\left< (K_{\Phi_1})_\npr, (K_{\Phi_2})_\npr\right>_{[G]} = \sum_{\pi \in \hat H^\Aut_\cusp}
 \left< \pi(\Phi_1 dg), \pi(\Phi_2 dg)\right>.$$
Thus, we obtain the spectral side of the non-invariant trace formula:
\begin{theorem}\label{spectral-final}
For $\Phi_1, \Phi_2\in \mathcal S(H(\adele))$ we have a decomposition:
$$\left[\left<K_{\Phi_1}, (K_{\Phi_2})_s\right>\right]_0  =  \sum_{\pi \in \hat H^\Aut_\cusp}
 \left< \pi(\Phi_1 dg), \pi(\Phi_2 dg)\right>+ $$
$$ +  \sum_{\Re(z')>0} \left< \pi_{z'}^\disc(\Phi_1 dg), \pi_{z'}^\disc(\Phi_2 dg) \right>+ \frac{1}{4}\left< M(0) \pi_0(\Phi_1 dg), \pi_0(\Phi_2 dg) \right> $$
\begin{equation}\label{spectral-STF}
- \frac{1}{2\pi i} \int_0^{i\infty} \left< \pi_{-z}(\Phi_1 dg),  \pi_z(\Phi_2 dg) M(-z) M'(z)\right> dz
\end{equation}
\end{theorem}

Of course, this is the same as \eqref{Selberg-spectral} if we set 
$$\Phi(g)= \Phi_1^\vee\star \Phi_2 (g)= \int_{H(\adele)}  \Phi_1(h) \Phi_2(hg) dh.$$

\section{The geometric side}

For the geometric side I have essentially nothing new to add, so I will just give a brief overview. Once one tries to compute \eqref{Laur-goal}:
$$\TF(\Phi_1\otimes\Phi_2)=\left[\left< K_{\Phi_1}, (K_{\Phi_2})_s\right>_{[G]}^*\right]=\left[\int^*_{[G]} K_{\Phi_1} (K_{\Phi_2})_s\right]$$
by expanding the definitions of $K_{\Phi_1}$ and $K_{\Phi_2}$, one arrives at the following expression, where we have set $\Phi(g)= \Phi_1^\vee\star \Phi_2 (g)$:

\begin{equation}\label{geom}
\left[\int_{[H]}^* \sum_{\gamma\in H(k)} {\Phi}(g^{-1}\gamma g) \Delta_U^s(g) dg\right] = \sum_{\mathfrak o} \left[\int_{[H]}^* \sum_{\gamma\in \mathfrak o} {\Phi}(g^{-1}\gamma g) \Delta_U^s(g) dg\right],
\end{equation}
where $\Delta_U$ is the function which is equal to $\Delta$ in a Siegel neighborhood  $U$ of the cusp in $[H]$, and $1$ otherwise, and $\mathfrak o$ runs over all equivalence classes in $H(k)$, where two elements are equivalent if their Jordan decompositions have conjugate semisimple parts. Equivalently (for the group $\PGL_2$), each $\mathfrak o$ corresponds to a $k$-rational point on $\Dfrac{H}{H} = \spec k[H]^{H-\operatorname{conj}}$ -- except for the traceless classes, which all correspond to the same point of $\Dfrac{H}{H}$.

The two-step Laurent expansion represented by the brackets $[\, \, ]$ depends continuously on the integrand without the $\Delta_U^s$ factor (which is an element of $\mathcal S^+_{\frac{1}{2}}([H])$), which in turn is of rapid decay in the parameter $\mathfrak o \in \Dfrac{H}{H}$ (because $\Phi$ is of rapid decay), allowing us to put the brackets inside of the $\mathfrak o$-sum. 
More precisely, the $\mathfrak o$-summand of the integrand is of rapid decay if $\mathfrak o$ is regular elliptic, and asymptotically finite with exponent $1$ (i.e., normalized character $\delta^\frac{1}{2}$ -- unnormalized $\delta^1$) if $\mathfrak o$ is hyperbolic, or the equivalence class of the identity. 

An elliptic class $\mathfrak o$ contributes its orbital integral, multiplied by the volume of stabilizers:
$$ \left[\int_{[H]}^* \sum_{\gamma\in \mathfrak o} {\Phi}(g^{-1}\gamma g) \Delta_U^s(g) dg\right] =  \int_{[H]} \sum_{\gamma\in \mathfrak o} {\Phi}(g^{-1}\gamma g) dg = $$
$$ =\Vol([H_\xi]) \int_{H_\xi(\adele)\backslash H(\adele)} \Phi(g^{-1}\xi g) dg,$$
where $\xi\in \mathfrak o$ is a representative.

Now we analyze the other classes. 
 It is easier on the geometric side to work with truncation, so we use Lemma \ref{constantcoeffs} which here translates to the statement that the coefficients $a_{-1}$ and $a_0$ of the Laurent expansion of the $\mathfrak o$-summand in \eqref{geom} are such that the integral
\begin{equation}
 \int_{[H]\smallsetminus U_T} \sum_{\gamma\in \mathfrak o} {\Phi}(g^{-1}\gamma g)  dg
\end{equation}
is equal to 
\begin{equation}\label{expansion}-a_{-1} T + a_0 + o(e^{-NT})
 \end{equation}
 for every $N\gg 0$, as $T\to \infty$, where $U_T$ is the neighborhood of infinity used in \eqref{reltotruncation}.

Let us study the case of $\mathfrak o=$ strongly regular hyperbolic (i.e., centralizers are split tori) and choose a point $\gamma\in \mathfrak o(k)$ with $M:=H_\gamma \simeq \Gm$, so that its orbit is $\mathfrak o = M\backslash H(k)$. The basic goal is to obtain an expression for the coefficients of the asymptotic expansion \eqref{expansion} which is \emph{local}.

For notational simplicity, but also to avoid confusion, we let $f$ be the restriction of $\Phi$ to $\mathfrak o$, denoted as a function on $M\backslash H(\adele)$; so, we need to evaluate the asymptotic expansion \eqref{expansion} of 
 \begin{equation}\label{geom-truncation}
 \int_{[H]\smallsetminus U_T} \sum_{\gamma\in M\backslash H(k)} f(\gamma g)  dg.
\end{equation}

We described above the asymptotic behavior of the integrand, but we will study it again using its constant term. As usual, we decompose the corresponding calculation into a sum over $B$-orbits:
$$ \left(\sum_{\gamma\in M\backslash H(k)} f(\gamma g)\right)_B = \sum_{\xi \in M(k)\backslash H(k)/B(k)} \sum_{\alpha\in A(k)/A_\xi (k)} \int_{N_\xi\backslash N(\adele)} f(\xi n \alpha g) dn,$$
where $A_\xi$ is the image of $\xi M \xi^{-1} \cap B$ in $A=B/N$.

The first sum is here finite -- in higher rank it is infinite. In any case, it can easily be seen that the summands are of rapid decay, unless $M\xi$ belongs to $B$. (Compare with Example \ref{torus}). Thus:
$$ \left(\sum_{\gamma\in \mathfrak o} f(\gamma g)\right)_B  \sim \sum_{w}  \int_{N(\adele)} f(w  n g) dn,$$
where again $\sim$ means that their difference is a rapidly decaying function in a neighborhood of the cusp in $[H]_B$, and the sum on the right is over double $M\backslash H/B$-cosets (over $k$) with $w^{-1}Mw\subset B$ -- they form a torsor for the Weyl group of $M$, hence the notation. 

Thus, the coefficients of the asymptotic expansion \eqref{expansion} will be the same if instead of truncating we subtract the above asymptotic expansion, restricted to the neighborhood $U_T'$ of the cusp in $B(k)\backslash H(\adele)$ which corresponds to (and maps isomorphically to) $U_T\subset [H]$:
\begin{equation}\label{geom-removeasympt}
 \int_{[H]} \left(\sum_{\gamma\in M\backslash H(k)} f(\gamma g)  -1_{U_T'}(g)\cdot  \sum_{w}  \int_{N(\adele)} f(w  n g) dn\right ) dg.
\end{equation}
We are slightly abusing notation here, since the function $g\mapsto \int_{N(\adele)} f(w  n g) dn$ lives on $[H]_B$ (and hence also on $B(k)\backslash H(\adele)$), not on $[H]$. However, identifying ${U_T}$ and $U_T'$, the restriction of this function to $U_T'$ makes sense as a function on $[H]$.

We now partition $M(k)\backslash H(\adele)$ into three sets, two of them denoted by $V^w_T$, where $w$ runs over $M(k)\backslash H(k)/B(k)$-cosets with $w^{-1}Mw\subset B$, as in the above sum, and their complement by $V_T^\rest$. The sets $V^w_T$ are defined in terms of the maps $M(k)\backslash H(\adele)\to B(k)\backslash H(\adele)$ defined by $w$: 
$$M(k)\backslash H(\adele) \ni x \mapsto w^{-1}x \in B(k)\backslash H(\adele),$$
and are defined as the preimages of $U_T'$ under these maps. 

\begin{lemma}
 For $T\gg 0$, the sets $V^w_T$, $w\in W$, are disjoint. The set $V^\rest_T$ has compact fibers under the map $M(k)\backslash H(\adele)\to M\backslash H(\adele)$.
\end{lemma}

We correspondingly write $f = \sum_w f_T^w + f_T^\rest$ for the decomposition of $f$ into its summands, restricted to these sets. The $w$-summand cancels with the corresponding summand of the asymptotics in \eqref{geom-removeasympt}:

\begin{lemma}
 For each $w$, 
$$\int_{[H]} \left(\sum_{\gamma\in M\backslash H(k)} f_T^w(\gamma g)  -1_{U_T}(g)\cdot  \int_{N(\adele)} f(w  n g) dn\right ) dg = 0.$$
\end{lemma}

Here we identify the function $g\mapsto \int_{N(\adele)} f(w  n g) dn$, restricted to $U_T'\subset B(k)\backslash H(\adele)$, as a function on $U_T\subset [H]$.

\begin{proof}
Writing 
$$\sum_{\gamma\in M\backslash H(k)} f_T^w(\gamma g)  = \sum_{\delta\in B\backslash H(k)} \sum_{\gamma\in M\backslash MwB(k)} f_T^w(\gamma \delta g) $$ $$=  \sum_{\delta\in B\backslash H(k)} 1_{U_T'}(\delta g) \sum_{\gamma\in M\backslash MwB(k)} f(\gamma \delta g),$$
 we have
$$\int_{[H]} \left(\sum_{\gamma\in M\backslash H(k)} f_T^w(\gamma g)  -1_{U_T}(g)\cdot  \int_{N(\adele)} f(w  n g) dn\right ) dg = $$
$$= \int_{B(k)\backslash H(\adele)} 1_{U_T'}(g) \left( \sum_{\gamma\in M\backslash MwB(k)} f(\gamma g) - \int_{N(\adele)} f(w  n g) dn \right) dg.
$$

We have $M\backslash MwB(k) = w N(k)$, and writing the outer integral as $\int_{[H]_B} \int_{[N]}$, the expression vanishes.
\end{proof}

\begin{corollary}
 The expression \eqref{geom-removeasympt} is equal to 
$$ \int_{M\backslash H(\adele)} f(x) v_T(x) dx,$$
where $v_T(x)= \Vol\{a\in [M]|  \forall w \, \, \Delta(w^{-1} a  x) < e^T   \}$, where we remind that $\Delta(bk)=\delta^{\frac{1}{2}}(b)$ is the height function on $B(k)\backslash H(\adele)$.
\end{corollary}

\begin{proof}
This is just the integral 
$$ \int_{[H]} \sum_{\gamma\in M\backslash H(k)} f_T^\rest(\gamma h) dh = \int_{M(k)\backslash H(\adele)} f_T^\rest(h) dh.$$
\end{proof}

Finally, we explicate the $a_{-1}$ and $a_0$ coefficients of the expansion \eqref{expansion}. To formulate, we need to fix some measures. Let $[M]^1$ denote the idele classes of norm one, so that $[M] = [M]^1\times \Rplus \simeq [A]^1 \times \Rplus$, when we identify $M$ with the universal Cartan $A$ though any element $w$. Recall that the embedding $\Rplus\hookrightarrow [A]$ has been fixed to be compatible with the character $\delta^{\frac{1}{2}}$. We take the usual measure $\frac{dx}{x}$ on $\Rplus$, and then choose any compatible measures on $[M]^1$ and $M\backslash H(\adele)$ (compatible with our measure on $[H]$) for the above formulas to hold.

\begin{proposition}\label{reghyper} 
 For $\mathfrak o = M\backslash H$ a strongly regular hyperbolic orbit, writing $f(g)=\Phi(g^{-1}\gamma g)$, the coefficients of \eqref{expansion} are
$$ a_{-1} = - 2 \Vol([M]^1) \int_{M\backslash H (\adele)} f( g) dg$$
(the orbital integral), and
$$ a_0 =  \int_{M\backslash H (\adele)} f(g) v(g) dg,$$
(the weighted orbital integral), where
$$v(g) = v_0(g) = \Vol\{a\in [M]|  \forall w \, \, \Delta(w^{-1} a  g) < 1   \}.$$
\end{proposition}

\begin{proof}
 
 It is easy to see that in the domain (for $T$) where the sets $V_T^w$ are all disjoint (or, more precisely, intersect with measure zero), the expression is linear in $T$, with linear coefficient equal  to $2 \Vol([M]^1) \int_{M\backslash H (\adele)} f( g) dg$. 
 
 By our choice of an Iwasawa pair $(B,K)$ with property \eqref{condBK}, this will already be true for $T=0$.
\end{proof}

\begin{remark}
 In usual expositions of the trace formula, the torus $M$ is identified with a subtorus of $B$, $w$ runs over the Weyl group, so for $w=1$ the condition $\Delta(w^{-1} a  x) < 1$ reads $\delta(a)<1$, and hence
$$ v(g) = - \log(\Delta(w n_g)),$$
where $w$ is the non-trivial element of the Weyl group and $g = a_g\cdot n_g \cdot k_g \in M N K$ is an Iwasawa decomposition for $g$. Indeed, we may assume that $g=n_g$ and then we have $\Delta(w a n_g) = \Delta(a)^{-1} \Delta(w n_g)$, so the conditions in the definition of $v(g)$ translate to $\Delta(a) <1$ and $\Delta(a)> \Delta(wn_g)$.
\end{remark}

\begin{remark}
 If $M$ is chosen inside of the chosen Borel $B$ (and through that, identified with the universal Cartan $A$), we can write the integral over $M\backslash H$ as an integral over $N\times K$, where the measure on $K$ is chosen compatibly with the chosen measure on $M\backslash H$. More precisely, we have
 $$ \int_{M\backslash H(\adele)} \Phi(g^{-1}\gamma g) dg = \int_{N(\adele)}\int_K \Phi(k^{-1} n^{-1} \gamma n k) dk dn $$
 $$ =\int_{N(\adele)} \int_K \Phi(k^{-1} \gamma n k) dk dn,$$
 where we have used the fact that $\gamma$ is regular and rational. 
\end{remark}

There is also a regular hyperbolic orbit (the orbit represented by a with eigenvalues $(1, -1)$) which is not strongly regular -- that means, centralizers are not connected tori, but disconnected groups whose identity components are tori. \emph{This should not be confused with a geometric orbit:} this non-strongly regular hyperbolic orbit belongs to a geometric orbit isomorphic to $\mathcal N_G(M)\backslash H$, where $M$ is a split torus and $\mathcal N_G$ denotes the normalizer, but its $k$-points also contain elliptic elements. The analysis of this hyperbolic orbit proceeds as before, by pulling the function $f$ back to $M\backslash H(\adele)$ via the quotient map $M\backslash H\to \mathcal N_G(M)\backslash H$. The result is the same formula as in Proposition \ref{reghyper}, but with a coefficient of $\frac{1}{2}$.

Finally, we need to analyze the contribution of $\mathfrak o=$ the class of the identity element. 

\begin{remark}
 The explicit calculation of the contribution of the non-regular classes is not really important! The result is not Eulerian and we don't expect to use it directly. Rather, in every comparison of trace formulas, the equality of irregular summands is deduced indirectly from the equality of regular ones. More important is a property of compatibility with Poisson summation (at the level of Lie algebras), that I will not discuss in this paper.
\end{remark}

\begin{proposition}
 The summand of \eqref{geom} corresponding to $\mathfrak o=$ the class of the identity element is equal to
$$\Vol([H]) \Phi(1) + [Z(F_\Phi,1-\frac{s}{2})],$$
where $F_\Phi(x) = \int_K \Phi\left(k^{-1} \begin{pmatrix} 1 & x \\ & 1 \end{pmatrix} k\right) dk$ and  $Z$ denotes the Tate integral
$$ Z(f,s) = \int_{\adele^\times} f(x) |x|^s d^\times x.$$
The measure on $\adele^\times$ is obtained here by fixing compatible measures on $A(\adele)$ and $K$ (compatible with our fixed measure on $H(\adele)$), \emph{identifying $A(\adele)\simeq \adele^\times$ via the character $\delta^{-1}$}.
\end{proposition}

\begin{proof}
Notice that the nilpotent cone can be identified with the affine closure of $Y:= N\backslash H$ (i.e.\ the union of $N\backslash H$ and a ``cusp'' corresponding to the identity element $1$), and that the restriction of $\Phi$ to it is of rapid decay at the ``funnel'' of $Y$. Thus, the sum $\sum_{\gamma\in \mathfrak o} {\Phi}(g^{-1}\gamma g)$ can be thought of as a pseudo-Eisenstein series, except that the function we input does not live in $\mathcal S(N\backslash H(\adele)$, but extends (locally, at every place) to a continuous function on the affine closure. 

Set $f=$ the restriction of $\Phi$ to the regular nilpotent set, identified with a function on $Y(\adele)=N\backslash H(\adele)$.  We split the sum $\sum_{\gamma\in \mathfrak o} {\Phi}(g^{-1}\gamma g)$ into $\Phi(1) + \sum_{\gamma\in Y(k)} f(\gamma g).$ The first term is constant and hence integrable over $[H]$. For the other, to calculate the contribution of $\mathfrak o$ to \eqref{geom}, it is here preferable to use meromorphic continuation, and replace the term $\Delta^s_U$ by simply $\Delta^s$, considered here as a function on $Y(\adele)$. (The details are as in Lemma \ref{differentvariations}.) We obtain:

$$  \left[\int_{[H]}^* \sum_{\gamma\in \mathfrak o} {\Phi}(g^{-1}\gamma g) \Delta_U^s(g) dg\right] = \Vol([H]) \Phi(1) + \left[ \int_{N\backslash H(\adele)}  f(g) \Delta^s(g) dg\right]. $$

Decomposing $\int_{N\backslash H(\adele)}$ as $\int_{A(\adele)} \delta^{-1}(a) \int_K$, and remembering that $\Delta = \delta^\frac{1}{2}$, we obtain:
$$ \Vol([H]) \Phi(1) + \left[ \int_{\adele^\times}  F_\Phi (a) |a|^{1-\frac{s}{2}} d^\times a\right].$$

\end{proof}

It can easily be seen (locally, at every place) that the function $F_\Phi$ is a Schwartz function on $\adele$. The residue of $Z(F_\Phi, 1-\frac{s}{2})$ at $s=0$ will therefore be equal to 
$$- 2 \Vol([\Gm]^1) \widehat{F_\Phi}(0) = - 2 \Vol([\Gm]^1) \int_{\adele} \int_K \Phi\left(k^{-1} \begin{pmatrix} 1 & x \\ & 1 \end{pmatrix} k\right) da dx =$$
$$=-  \Vol([A]^1) \int_{\adele} \int_K \Phi\left(k^{-1} \begin{pmatrix} 1 & x \\ & 1 \end{pmatrix} k\right) da dx.$$

\begin{remark}
 A subtle point here is that one cannot fix measures $dx$ on $\adele$ and $d^\times x$ on $\adele^\times$ which satisfy $d^\times x = \frac{dx}{x}$. Thus, we are using the standard measure on $\adele = N(\adele)$ (which is self-dual with respect to characters of $\adele/k$), and quite an arbitary measure on $A(\adele)=\adele^\times$ (chosen compatibly with the measure on $K$). Typically, given a measure on $[\Gm]$ one takes the measure on $[\Gm]^1$ to be the one compatible with the short exact sequence
 $$ 1\to [\Gm]^1\to [\Gm] \xrightarrow{|\bullet|} \Rplus \to 1,$$
 and the standard measure $\frac{dx}{x}$ on $\Rplus$. However, for the universal Cartan $A$ we have above taken the measure on $[A]^1$ to be the one corresponding to the short exact sequence 
 $$ 1\to [A]^1\to [A] \xrightarrow{\delta^{\frac{1}{2}}} \Rplus \to 1.$$
 This explains the disappearance of the factor $2$ from the last equality, since the identification $A\simeq\Gm$ was here via the character $\delta^{-1}$. 

 (It is clearly a poor choice to be working with $\delta^\frac{s}{2}$ instead of $\delta^s$, which I have followed for historical compatibility.)
\end{remark}

We have shown:

\begin{theorem}\label{geometric-final}
Choose any split torus $M\subset H$. Setting $\Phi(g) = \Phi_1^\vee\star \Phi_2$, we have
\begin{eqnarray}
\TF_{-1}(\Phi_1\otimes\Phi_2)= - \sum_{{\alpha\in A(k)}}  \Vol([A]^1) \int_{N(\adele)} \int_K \Phi\left(k^{-1} n k\right) dk dn,
\end{eqnarray}
and 
\begin{eqnarray}
\TF_0(\Phi_1\otimes\Phi_2)= \sum_{\mathfrak o = [\xi]:\mbox{ elliptic}} \Vol([H_\xi]) \int_{H_\xi(\adele)\backslash H(\adele)} \Phi(g^{-1}\xi g) dg + \nonumber \\
 + \frac{1}{2} \sum_{\alpha\in M(k)} \int_{M\backslash H (\adele)} \Phi(g^{-1} \alpha g) v(g) dg + \nonumber \\ 
 \Vol([H]) \Phi(1) + [Z(F_\Phi,1-\frac{s}{2})]_0,
\end{eqnarray}
where $F_\Phi(x) = \int_K \Phi\left(k^{-1} \begin{pmatrix} 1 & x \\ & 1 \end{pmatrix} k\right) dk$.
\end{theorem}

Equating the spectral and geometric expressions for $\TF_0$ from Theorems \ref{spectral-final} and \ref{geometric-final} (see also \eqref{Selberg-spectral}), we get the Selberg trace formula:

\begin{theorem}
 Let $\Phi \in \mathcal S([H])$, then 
 $$\sum_{\pi \in \hat H^\Aut_\disc} \tr(\pi(\Phi)) + \frac{1}{4} \tr (M(0) \pi_0(\Phi))-  \frac{1}{4\pi i} \int_{-i\infty}^{i\infty} \tr(M(-z)M'(z)\pi_{z}(\Phi)) dz = $$ 
$$ =\sum_{\mathfrak o = [\xi]:\mbox{ elliptic}} \Vol([H_\xi]) \int_{H_\xi(\adele)\backslash H(\adele)} \Phi(g^{-1}\xi g) dg  +
  \frac{1}{2} \sum_{\alpha\in M(k)} \int_{M\backslash H (\adele)} \Phi(g^{-1} \alpha g) v(g) dg + $$ 
\begin{eqnarray}
 \Vol([H]) \Phi(1) + [Z(F_\Phi,1-\frac{s}{2})]_0.
 \end{eqnarray}

\end{theorem}

\section{The invariant trace formula (TO BE ADDED!)}

TO BE ADDED.

\bibliographystyle{alphaurl}
\bibliography{selberg}

\end{document}